\documentclass[10pt]{amsart}

\usepackage[utf8x]{inputenc}
\usepackage[english]{babel}
\usepackage{amsmath,amsfonts,amssymb,amsthm,shuffle}
\usepackage[T1]{fontenc}
\usepackage{lmodern}

\usepackage[top=3.5cm,bottom=3.5cm,left=3.4cm,right=3.4cm]{geometry}

\usepackage[dvipsnames]{xcolor}
\usepackage[hyperindex=true,frenchlinks=true,colorlinks=true,
citecolor=NavyBlue,linkcolor=Mulberry,urlcolor=LimeGreen,linktocpage]{hyperref}

\usepackage{tikz}
\usetikzlibrary{shapes}
\usetikzlibrary{fit}
\usetikzlibrary{decorations.pathmorphing}

\usepackage{cite}
\usepackage{subfigure}
\usepackage{multirow}
\usepackage{enumitem}
\usepackage{todonotes}

\numberwithin{equation}{subsection}

\newtheorem{Theoreme}{Theorem}[section]
\newtheorem{Proposition}[Theoreme]{Proposition}
\newtheorem{Lemme}[Theoreme]{Lemma}

\title[Enveloping operads]{Enveloping operads \\ and \\
bicoloured noncrossing configurations}
\date{\today}

\author{Fr\'ed\'eric Chapoton}
\address{Institut Camille Jordan, Université Claude Bernard Lyon 1,
    F-69622 Villeurbanne cedex, France}
\email{chapoton@math.univ-lyon1.fr}

\author{Samuele Giraudo}
\address{Laboratoire d'Informatique Gaspard-Monge, Université Paris-Est
    Marne-la-Vallée, 5 boulevard Descartes, Champs-sur-Marne,
    77454 Marne-la-Vallée cedex 2, France}
\email{samuele.giraudo@univ-mlv.fr}

\renewcommand{\leq}{\leqslant}
\renewcommand{\geq}{\geqslant}

\newcommand{\Cca}{{\mathcal{C}}}
\newcommand{\Oca}{{\mathcal{O}}}

\newcommand{\La}{{\mathtt{a}}}
\newcommand{\Lb}{{\mathtt{b}}}
\newcommand{\Lc}{{\mathtt{c}}}

\newcommand{\Bfr}{{\mathfrak{B}}}
\newcommand{\Cfr}{{\mathfrak{C}}}
\newcommand{\Dfr}{{\mathfrak{D}}}
\newcommand{\Efr}{{\mathfrak{E}}}
\newcommand{\Gfr}{{\mathfrak{G}}}

\newcommand{\SerieOp}{\operatorname{F}}
\newcommand{\SerieBulles}{\operatorname{B}}
\newcommand{\OpLibre}{\mathbf{Free}}
\newcommand{\OpEnv}{\mathbf{Hull}}
\newcommand{\AntiCol}{\mathbf{Anti}}
\newcommand{\Unite}{\mathbf{1}}
\newcommand{\In}{\mathbf{in}}
\newcommand{\Out}{\mathbf{out}}
\newcommand{\Deg}{\operatorname{deg}}

\newcommand{\Corolle}{\operatorname{c}}
\newcommand{\FAs}{{\mathsf{FAs}}}
\newcommand{\CNCB}{{\mathsf{BNC}}}
\newcommand{\Bulle}{{\mathsf{Bubble}}}
\newcommand{\Bord}{\operatorname{b}}
\newcommand{\Ret}{\operatorname{ret}}
\newcommand{\Cpl}{\operatorname{cpl}}
\newcommand{\Op}[1]{\langle #1 \rangle}
\newcommand{\OpC}[1]{\langle\langle #1 \rangle\rangle}
\newcommand{\mapsfrom}{\rotatebox{180}{$\mapsto$}}
\newcommand{\OpA}{\diamond}
\newcommand{\OpB}{\bullet}

\newcommand{\Sloane}[1]{\href{http://oeis.org/#1}{{\bf #1}}}

\newcommand{\AAA}{\scalebox{.175}{
\begin{tikzpicture}
    \node[shape=coordinate](C1)at(-0.87,-0.50){};
    \node[shape=coordinate](C2)at(0.00,1.0){};
    \node[shape=coordinate](C3)at(0.87,-0.50){};
    \draw[CotePolyg,line width=3pt](C1)--(C3);
    \draw[CotePolyg,line width=3pt](C1)--(C2);
    \draw[CotePolyg,line width=3pt](C2)--(C3);
    \draw[ArcBleu,line width=7pt](C1)--(C2);
    \draw[ArcBleu,line width=7pt](C1)--(C3);
    \draw[ArcBleu,line width=7pt](C2)--(C3);
\end{tikzpicture}}\,}

\newcommand{\AAB}{\scalebox{.175}{
\begin{tikzpicture}
    \node[shape=coordinate](C1)at(-0.87,-0.50){};
    \node[shape=coordinate](C2)at(0.00,1.0){};
    \node[shape=coordinate](C3)at(0.87,-0.50){};
    \draw[CotePolyg,line width=3pt](C1)--(C3);
    \draw[CotePolyg,line width=3pt](C1)--(C2);
    \draw[CotePolyg,line width=3pt](C2)--(C3);
    \draw[ArcBleu,line width=7pt](C1)--(C2);
    \draw[ArcBleu,line width=7pt](C1)--(C3);
\end{tikzpicture}}\,}

\newcommand{\ABA}{\scalebox{.175}{
\begin{tikzpicture}
    \node[shape=coordinate](C1)at(-0.87,-0.50){};
    \node[shape=coordinate](C2)at(0.00,1.0){};
    \node[shape=coordinate](C3)at(0.87,-0.50){};
    \draw[CotePolyg,line width=3pt](C1)--(C3);
    \draw[CotePolyg,line width=3pt](C1)--(C2);
    \draw[CotePolyg,line width=3pt](C2)--(C3);
    \draw[ArcBleu,line width=7pt](C1)--(C2);
    \draw[ArcBleu,line width=7pt](C2)--(C3);
\end{tikzpicture}}\,}

\newcommand{\ABB}{\scalebox{.175}{
\begin{tikzpicture}
    \node[shape=coordinate](C1)at(-0.87,-0.50){};
    \node[shape=coordinate](C2)at(0.00,1.0){};
    \node[shape=coordinate](C3)at(0.87,-0.50){};
    \draw[CotePolyg,line width=3pt](C1)--(C3);
    \draw[CotePolyg,line width=3pt](C1)--(C2);
    \draw[CotePolyg,line width=3pt](C2)--(C3);
    \draw[ArcBleu,line width=7pt](C1)--(C2);
\end{tikzpicture}}\,}

\newcommand{\BAA}{\scalebox{.175}{
\begin{tikzpicture}
    \node[shape=coordinate](C1)at(-0.87,-0.50){};
    \node[shape=coordinate](C2)at(0.00,1.0){};
    \node[shape=coordinate](C3)at(0.87,-0.50){};
    \draw[CotePolyg,line width=3pt](C1)--(C3);
    \draw[CotePolyg,line width=3pt](C1)--(C2);
    \draw[CotePolyg,line width=3pt](C2)--(C3);
    \draw[ArcBleu,line width=7pt](C1)--(C3);
    \draw[ArcBleu,line width=7pt](C2)--(C3);
\end{tikzpicture}}\,}

\newcommand{\BAB}{\scalebox{.175}{
\begin{tikzpicture}
    \node[shape=coordinate](C1)at(-0.87,-0.50){};
    \node[shape=coordinate](C2)at(0.00,1.0){};
    \node[shape=coordinate](C3)at(0.87,-0.50){};
    \draw[CotePolyg,line width=3pt](C1)--(C3);
    \draw[CotePolyg,line width=3pt](C1)--(C2);
    \draw[CotePolyg,line width=3pt](C2)--(C3);
    \draw[ArcBleu,line width=7pt](C1)--(C3);
\end{tikzpicture}}\,}

\newcommand{\BBA}{\scalebox{.175}{
\begin{tikzpicture}
    \node[shape=coordinate](C1)at(-0.87,-0.50){};
    \node[shape=coordinate](C2)at(0.00,1.0){};
    \node[shape=coordinate](C3)at(0.87,-0.50){};
    \draw[CotePolyg,line width=3pt](C1)--(C3);
    \draw[CotePolyg,line width=3pt](C1)--(C2);
    \draw[CotePolyg,line width=3pt](C2)--(C3);
    \draw[ArcBleu,line width=7pt](C2)--(C3);
\end{tikzpicture}}\,}

\newcommand{\BBB}{\scalebox{.175}{
\begin{tikzpicture}
    \node[shape=coordinate](C1)at(-0.87,-0.50){};
    \node[shape=coordinate](C2)at(0.00,1.0){};
    \node[shape=coordinate](C3)at(0.87,-0.50){};
    \draw[CotePolyg,line width=3pt](C1)--(C3);
    \draw[CotePolyg,line width=3pt](C1)--(C2);
    \draw[CotePolyg,line width=3pt](C2)--(C3);
\end{tikzpicture}}\,}

\definecolor{Noir}{RGB}{0,0,0}
\definecolor{Blanc}{RGB}{255,255,255}
\definecolor{Rouge}{RGB}{205,35,38}
\definecolor{Bleu}{RGB}{2,60,195}
\definecolor{Vert}{RGB}{23,163,1}
\definecolor{Violet}{RGB}{181,18,225}
\definecolor{Orange}{RGB}{255,113,15}
\definecolor{Marron}{RGB}{52,46,0}

\tikzstyle{CotePolyg}=[Noir!30,thick,draw,cap=round,line width=.75pt]
\tikzstyle{ArcBleu}=[Bleu!80,thick,draw,cap=round,line width=2pt]
\tikzstyle{ArcRouge}=[Rouge!80,thick,draw,cap=round,line width=1.5pt,dotted]
\tikzstyle{Noeud}=[circle,draw=Bleu!80,fill=Bleu!20,inner sep=0pt,
minimum size=13mm,line width=4pt,font=\Huge]
\tikzstyle{Arete}=[Rouge!80,cap=round,line width=4pt]
\tikzstyle{Feuille}=[rectangle,draw=Noir!70,fill=Noir!20,
inner sep=0pt,minimum size=3mm,line width=2pt]
\tikzstyle{Clair}=[draw=Bleu!100,fill=Blanc!100,font=\Huge]
\tikzstyle{EtiqCoul}=[anchor=center,left,xshift=-2pt,text=Noir,font=\Huge]
\tikzstyle{Marque1}=[draw=Vert!100,fill=Vert!30]
\tikzstyle{Marque2}=[draw=Orange!100,fill=Orange!40]
\tikzstyle{Marque3}=[draw=Rouge!100,fill=Rouge!50]

\begin{document}

\begin{abstract}
  An operad structure on certain bicoloured noncrossing configurations
  in regular polygons is studied. Motivated by this study, a general
  functorial construction of enveloping operad, with input a coloured
  operad and output an operad, is presented. The operad of noncrossing
  configurations is shown to be the enveloping operad of a coloured
  operad of bubbles. Several suboperads are also investigated, and
  described by generators and relations.
\end{abstract}

\maketitle
\tableofcontents

\section*{Introduction}
This article is concerned with some operads in a context of algebraic
combinatorics. The theory of operads started as a device to organize
the complicated structures appearing in algebraic topology, and in
particular the many operations arising on loop spaces and their
homology \cite{boardman_vogt}. Since then, it has been more and more
clear that operads can be used with profit in various other settings,
see \cite{kapranov_icm,loday_vallette} and the references therein. One
striking example is the study of the moduli spaces of complex curves,
where the compactification naturally involves gluing operations
\cite{getzler_kapranov}.
\smallskip

In algebraic combinatorics, there is on the one hand a long tradition of
using associative algebras, words and languages to describe combinatorial
objects and to decompose them into more elementary pieces. On the
other hand, the theory of operads is closely related to various kinds
of trees, and provides a way to create new objects by gluing smaller
ones \cite{Cha08}. One can therefore hope that algebraic combinatorics
can benefit from a larger use of the theory of operads, and maybe there
can be also fruitful interaction in the other way.
\smallskip

In this article, these ideas will get illustrated by some examples of
operads with a combinatorial flavour, and by a general operadic
construction inspired by these examples.
\smallskip

Let us describe our motivation for this work. In a previous
article \cite{Cha07}, the first named author has considered an operad
structure on the objects called noncrossing trees and noncrossing
plants. These objects can be depicted as simple graphs inside regular
polygons, and are some kind of noncrossing configurations that
are well-known combinatorial objects \cite{FN99,FS09}.
The composition of these operads has a very simple graphical
description and it is tempting and easy to generalize this composition
as much as possible, by removing some constraints on the objects. This
leads to a very big operad of noncrossing configurations. This
research initially started as a study of this operad, with possible
aim the description of its suboperads.
\smallskip

This study has led us to the following results. First, we introduce a
general functorial construction from coloured operads to operads,
which is called the {\em enveloping operad}. This can be compared to
the amalgamated product of groups, in the sense that it takes a
compound object to build a unified object in the simplest possible
way, by imposing as few relations as possible. The main interest of
this construction relies on the fact that a lot of properties of an
enveloping operad (as {\em e.g.}, its Hilbert series and a
presentation by generators and relations) can be obtained from its
underlying coloured operad.  \smallskip

Next, we consider the operad $\CNCB$ of bicoloured noncrossing
configurations, defined by a simple graphical composition, and show
that it admits a description as the enveloping operad of a very simple
coloured operad on two colours called $\Bulle$. We also obtain a
presentation by generators and relations of the coloured operad $\Bulle$.
\smallskip

Then this understanding of the operad $\CNCB$ is used to describe in
details some of its suboperads, namely those generated by two
chosen generators among the binary generators of $\CNCB$. This already
gives an interesting family of operads, where one can recognize some
known ones: the operad of noncrossing trees \cite{Cha07}, the operad
of noncrossing plants \cite{Cha07}, the dipterous operad \cite{LR03,Zin12},
and the $2$-associative operad \cite{LR06,Zin12}.
Our main results here are a presentation by generators and
relations for all these suboperads but one, and also the description
of all the generating series. It should be noted
that the presentations are obtained in a case-by-case fashion, using
similar rewriting techniques.
\smallskip

This article is organized as follows. In Section
\ref{sec:operades_enveloppantes}, the general construction of
enveloping operads is given and its properties described. Next, in Section
\ref{sec:operade_CNCB}, we introduce the operad $\CNCB$ and prove that
this operad is isomorphic to an enveloping operad. Finally, in Section
\ref{sec:sous_operades_CNCB}, several suboperads of $\CNCB$ are
considered, in a more or less detailed way.
\medskip

{\it Acknowledgments.}
This research has been done using the open-source software
\texttt{Sage} \cite{sage} and the algebraic combinatorics code
developed by the \texttt{Sage-combinat} community
\cite{sage-combinat}.
The first named author has been supported by the ANR program CARMA
(ANR-12-BS01-0017-02).
\medskip

\section{Enveloping operads of coloured operads} \label{sec:operades_enveloppantes}
The aim of this Section is threefold. We begin by recalling some basic
notions about coloured operads and free coloured operads. Then, we introduce
the main object of this paper: the construction which associates an operad
with a coloured one, namely its enveloping operad. We finally justify the
benefits of seeing an operad $\Oca$ as an enveloping operad of a coloured
one $\Cca$ by reviewing some properties of $\Oca$ that can be deduced from
the ones of $\Cca$.
\medskip

\subsection{Free coloured operads}
Let $k$ be a positive integer. We shall consider in this work non-symmetric
$k$-coloured operads in the category of sets; notice that through a slight
translation, all next notions and results remain valid in the category
of vector spaces. Coloured operads are operads
where a composition $x \circ_i y$ is defined between two elements $x$ and
$y$ of $\Cca$ if and only if the {\em output colour $\Out(y)$} of $y$ is
the same as the $i$th {\em input colour $\In_i(x)$} of $x$; the set of
allowed colours being $[k] := \{1, \dots, k\}$. In the sequel, we only
consider $k$-coloured operads $\Cca$ such that $\Cca(1) = \{\Unite_c : c \in [k]\}$
where $\Unite_c$ is a unit with $c$ as output and input colour, and such
that there are finitely many elements of arity $n$ for any $n \geq 1$.
Since (noncoloured) operads are $1$-coloured operads, the following notions
and notations also work for operads.
\medskip

\subsubsection{Coloured syntax trees} \label{subsubsec:arbres_syntaxiques_colores}
A {\em $k$-coloured collection} is a graded set $C = \uplus_{n \geq 2} C(n)$
together with two maps $\Out, \In_i : C \to [k]$ which respectively associate
with an element $x$ of $C(n)$ the colour of its output and the colour of
its $i$th input, where $i$ lies between $1$ and the arity $|x| := n$ of $x$.
\medskip

For any $k$-coloured collection $C$, we denote by $\OpLibre(C)$ the set
of {\em $k$-coloured syntax trees on $C$}, that are planar rooted trees
such that
\begin{enumerate}
    \item internal nodes of arity $\ell$ are labeled by elements of $C(\ell)$;
    \item for any internal nodes $r$ and $s$ such that $s$ is the $i$th
    child of $r$, we have $\In_i(x) = \Out(y)$ where $x$ (resp. $y$)
    is the label of $r$ (resp. $s$).
\end{enumerate}
\medskip

Let $T$ be a coloured syntax tree on $C$. The {\em arity} $|T|$ of $T$ is
its number of leaves and the {\em degree} $\Deg(T)$ of $T$ is its number
of internal nodes. The leaves of $T$ are numbered from left to right. By
a slight abuse of notation, for any internal node $r$ of $T$, $\Out(r)$
(resp. $\In_i(r)$) denotes the colour $\Out(x)$ (resp. $\In_i(x)$) where
$x$ is the label of $r$. Continuing the same abuse, $\Out(T)$ is the output
colour of the root of $T$ and $\In_i(T)$ is the colour of the input of the
internal node on which the $i$th leaf of $T$ is attached. A {\em corolla}
labeled by $x \in C(\ell)$ is the coloured syntax tree $\Corolle(x)$ on
$C$ consisting in one internal node labeled by $x$ with $\ell$ leaves as
children. Figure \ref{fig:arbre_syntaxique_colore} shows an example of a
$2$-coloured syntax tree.
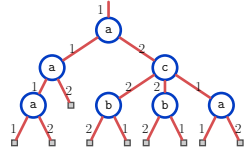
\begin{figure}[ht]
    \centering
    \scalebox{.25}{\begin{tikzpicture}
        \node[Noeud,Clair](0)at(0,-4){$\La$};
        \node[Feuille](1)at(-1,-6){};
        \node[Feuille](1')at(1,-6){};
        \node[Noeud,Clair](2)at(1,-2){$\La$};
        \node[Feuille](3)at(2,-4){};
        \node[Noeud,Clair](4)at(4,0){$\La$};
        \node[Feuille](5)at(3,-6){};
        \node[Noeud,Clair](6)at(4,-4){$\Lb$};
        \node[Feuille](7)at(5,-6){};
        \node[Noeud,Clair](8)at(7,-2){$\Lc$};
        \node[Noeud,Clair](9)at(7,-4){$\Lb$};
        \node[Feuille](10)at(6,-6){};
        \node[Feuille](10')at(8,-6){};
        \node[Feuille](11)at(9,-6){};
        \node[Noeud,Clair](12)at(10,-4){$\La$};
        \node[Feuille](13)at(11,-6){};
        \draw[Arete](0)edge node[EtiqCoul]{$1$}(2);
        \draw[Arete](1)edge node[EtiqCoul]{$1$}(0);
        \draw[Arete](1')edge node[EtiqCoul,right]{$2$}(0);
        \draw[Arete](2)edge node[EtiqCoul]{$1$}(4);
        \draw[Arete](3)edge node[EtiqCoul,right]{$2$}(2);
        \draw[Arete](5)edge node[EtiqCoul]{$2$}(6);
        \draw[Arete](6)edge node[EtiqCoul]{$2$}(8);
        \draw[Arete](7)edge node[EtiqCoul,right]{$1$}(6);
        \draw[Arete](8)edge node[EtiqCoul,right]{$2$}(4);
        \draw[Arete](9)edge node[EtiqCoul]{$2$}(8);
        \draw[Arete](10)edge node[EtiqCoul]{$2$}(9);
        \draw[Arete](10')edge node[EtiqCoul,right]{$1$}(9);
        \draw[Arete](11)edge node[EtiqCoul]{$1$}(12);
        \draw[Arete](12)edge node[EtiqCoul,right]{$1$}(8);
        \draw[Arete](13)edge node[EtiqCoul,right]{$2$}(12);
        \draw[Arete](4)edge node[EtiqCoul]{$1$}(4,1.5);
    \end{tikzpicture}}
    \caption{A $2$-coloured syntax tree on the $2$-coloured collection
    $C := C(2) \uplus C(3)$ defined by $C(2) := \{\La, \Lb\}$, $C(3) := \{\Lc\}$,
    $\Out(\La) := 1$, $\Out(\Lb) := 2$, $\Out(\Lc) := 2$,
    $\In_1(\La) := 1$, $\In_2(\La) := 2$,
    $\In_1(\Lb) := 2$, $\In_2(\Lb) := 1$,
    $\In_1(\Lc) := 2$, $\In_2(\Lc) := 2$, $\In_3(\Lc) := 1$.
    Its arity is $9$, its degree is $7$, its output colour is $1$ and its
    second input colour is $2$.}
    \label{fig:arbre_syntaxique_colore}
\end{figure}
\medskip

We say that a coloured syntax tree $S$ is a {\em subtree} of $T$ if it is
possible to fit $S$ at a certain place of $T$, by possibly superimposing
leaves of $S$ and internal nodes of $T$. Figure \ref{fig:exemple_sous_arbre}
shows a coloured syntax tree an one of its subtrees.
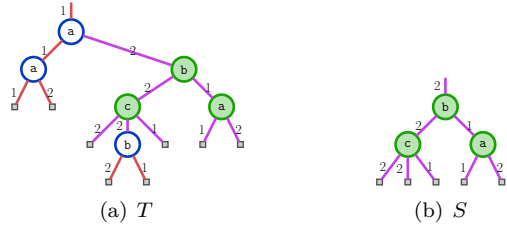
\begin{figure}[ht]
    \subfigure[$T$]{
        \scalebox{.25}{\begin{tikzpicture}
            \node[Feuille](0)at(0.00,-4.00){};
            \node[Feuille](11)at(10.00,-6.00){};
            \node[Feuille](13)at(12.00,-6.00){};
            \node[Feuille](2)at(2.00,-4.00){};
            \node[Feuille](4)at(4.00,-6.00){};
            \node[Feuille](6)at(5.00,-8.00){};
            \node[Feuille](8)at(7.00,-8.00){};
            \node[Feuille](9)at(8.00,-6.00){};
            \node[Noeud,Clair](1)at(1.00,-2.00){$\La$};
            \node[Noeud,Marque1](10)at(9.00,-2.00){$\Lb$};
            \node[Noeud,Marque1](12)at(11.00,-4.00){$\La$};
            \node[Noeud,Clair](3)at(3.00,0.00){$\La$};
            \node[Noeud,Marque1](5)at(6.00,-4.00){$\Lc$};
            \node[Noeud,Clair](7)at(6.00,-6.00){$\Lb$};
            \draw[Arete](0)edge node[EtiqCoul]{$1$}(1);
            \draw[Arete](1)edge node[EtiqCoul]{$1$}(3);
            \draw[Arete,Violet!80](10)edge node[EtiqCoul,right]{$2$}(3);
            \draw[Arete,Violet!80](11)edge node[EtiqCoul]{$1$}(12);
            \draw[Arete,Violet!80](12)edge node[EtiqCoul,right]{$1$}(10);
            \draw[Arete,Violet!80](13)edge node[EtiqCoul,right]{$2$}(12);
            \draw[Arete](2)edge node[EtiqCoul,right]{$2$}(1);
            \draw[Arete,Violet!80](4)edge node[EtiqCoul]{$2$}(5);
            \draw[Arete,Violet!80](5)edge node[EtiqCoul]{$2$}(10);
            \draw[Arete](6)edge node[EtiqCoul]{$2$}(7);
            \draw[Arete,Violet!80](7)edge node[EtiqCoul]{$2$}(5);
            \draw[Arete](8)edge node[EtiqCoul,right]{$1$}(7);
            \draw[Arete,Violet!80](9)edge node[EtiqCoul,right]{$1$}(5);
            \draw[Arete](3)edge node[EtiqCoul]{$1$}(3,1.5);
        \end{tikzpicture}}
    }
    \qquad \qquad
    \subfigure[$S$]{
        \scalebox{.25}{\begin{tikzpicture}
            \node[Feuille](0)at(-.5,-4.00){};
            \node[Feuille](2)at(1.00,-4.00){};
            \node[Feuille](3)at(2.5,-4.00){};
            \node[Feuille](5)at(4.00,-4.00){};
            \node[Feuille](7)at(6.00,-4.00){};
            \node[Noeud,Marque1](1)at(1.00,-2.00){$\Lc$};
            \node[Noeud,Marque1](4)at(3.00,0.00){$\Lb$};
            \node[Noeud,Marque1](6)at(5.00,-2.00){$\La$};
            \draw[Arete,Violet!80](0)edge node[EtiqCoul]{$2$}(1);
            \draw[Arete,Violet!80](1)edge node[EtiqCoul]{$2$}(4);
            \draw[Arete,Violet!80](2)edge node[EtiqCoul]{$2$}(1);
            \draw[Arete,Violet!80](3)edge node[EtiqCoul,right]{$1$}(1);
            \draw[Arete,Violet!80](5)edge node[EtiqCoul]{$1$}(6);
            \draw[Arete,Violet!80](6)edge node[EtiqCoul,right]{$1$}(4);
            \draw[Arete,Violet!80](7)edge node[EtiqCoul,right]{$2$}(6);
            \draw[Arete,Violet!80](4)edge node[EtiqCoul]{$2$}(3,1.5);
        \end{tikzpicture}}
    }
    \caption{A coloured syntax tree $T$ on the coloured collection defined in
    Figure~\ref{fig:arbre_syntaxique_colore} and $S$, one of its subtrees.}
    \label{fig:exemple_sous_arbre}
\end{figure}
\medskip

In what follows, specifically to deal with presentations of coloured operads,
we shall make use of rewrite rules on coloured syntax trees. A
{\em rewrite rule} is a binary relation $\mapsto$ on coloured syntax trees
where $S \mapsto T$ only if the trees $S$ and $T$ have the same arity.
Let $S'$ and $T'$ be two coloured syntax trees. We say that $S'$ can be
{\em rewritten} by $\mapsto$ into $T'$ if there exist two coloured
syntax trees $S$ and $T$ satisfying $S \mapsto T$ and $S'$ has a subtree
$S$ such that, by replacing $S$ by $T$ in $S'$, we obtain $T'$. By a slight
abuse of notation, we denote by $S' \mapsto T'$ this property. We shall
use the standard terminology ({\em confluent}, {\em terminating},
{\em normal form}, {\em etc.}) about rewrite rules (see for
instance \cite{BN98}).
\medskip

\subsubsection{Free coloured operads}
The set $\OpLibre(C)$, together with the units $\{\Unite_c : c \in [k]\}$,
is endowed with a $k$-coloured operad structure defined as follows. For
any coloured syntax trees $S$ and $T$ on $C$, the composition $S \circ_i T$,
defined when the output colour of $T$ is the same as the $i$th input of
$S$, is the coloured syntax tree obtained by grafting the root of $T$ on
the $i$th leaf of $S$. Figure \ref{fig:composition_operade_coloree_libre}
shows an example of such a composition. This forms the
{\em free coloured operad generated by $C$}, denoted by $\OpLibre(C)$.
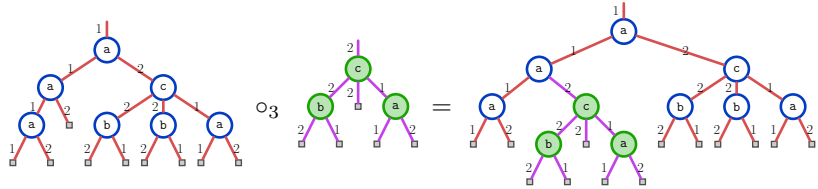
\begin{figure}[ht]
    \begin{equation*}
        \begin{split}\scalebox{.25}{
        \begin{tikzpicture}
            \node[Noeud,Clair](0)at(0,-4){$\La$};
            \node[Feuille](1)at(-1,-6){};
            \node[Feuille](1')at(1,-6){};
            \node[Noeud,Clair](2)at(1,-2){$\La$};
            \node[Feuille](3)at(2,-4){};
            \node[Noeud,Clair](4)at(4,0){$\La$};
            \node[Feuille](5)at(3,-6){};
            \node[Noeud,Clair](6)at(4,-4){$\Lb$};
            \node[Feuille](7)at(5,-6){};
            \node[Noeud,Clair](8)at(7,-2){$\Lc$};
            \node[Noeud,Clair](9)at(7,-4){$\Lb$};
            \node[Feuille](10)at(6,-6){};
            \node[Feuille](10')at(8,-6){};
            \node[Feuille](11)at(9,-6){};
            \node[Noeud,Clair](12)at(10,-4){$\La$};
            \node[Feuille](13)at(11,-6){};
            \draw[Arete](0)edge node[EtiqCoul]{$1$}(2);
            \draw[Arete](1)edge node[EtiqCoul]{$1$}(0);
            \draw[Arete](1')edge node[EtiqCoul,right]{$2$}(0);
            \draw[Arete](2)edge node[EtiqCoul]{$1$}(4);
            \draw[Arete](3)edge node[EtiqCoul,right]{$2$}(2);
            \draw[Arete](5)edge node[EtiqCoul]{$2$}(6);
            \draw[Arete](6)edge node[EtiqCoul]{$2$}(8);
            \draw[Arete](7)edge node[EtiqCoul,right]{$1$}(6);
            \draw[Arete](8)edge node[EtiqCoul,right]{$2$}(4);
            \draw[Arete](9)edge node[EtiqCoul]{$2$}(8);
            \draw[Arete](10)edge node[EtiqCoul]{$2$}(9);
            \draw[Arete](10')edge node[EtiqCoul,right]{$1$}(9);
            \draw[Arete](11)edge node[EtiqCoul]{$1$}(12);
            \draw[Arete](12)edge node[EtiqCoul,right]{$1$}(8);
            \draw[Arete](13)edge node[EtiqCoul,right]{$2$}(12);
            \draw[Arete](4)edge node[EtiqCoul]{$1$}(4,1.5);
        \end{tikzpicture}}\end{split}
        \begin{split} \enspace \circ_3 \enspace \end{split}
        \begin{split}\scalebox{.25}{
        \begin{tikzpicture}
            \node[Feuille](0)at(0,-4){};
            \node[Noeud,Marque1](1)at(1,-2){$\Lb$};
            \node[Feuille](2)at(2,-4){};
            \node[Noeud,Marque1](3)at(3,0){$\Lc$};
            \node[Feuille](4)at(3,-2){};
            \node[Noeud,Marque1](5)at(5,-2){$\La$};
            \node[Feuille](6)at(4,-4){};
            \node[Feuille](7)at(6,-4){};
            \draw[Arete,Violet!80](0)edge node[EtiqCoul]{$2$}(1);
            \draw[Arete,Violet!80](1)edge node[EtiqCoul]{$2$}(3);
            \draw[Arete,Violet!80](2)edge node[EtiqCoul,right]{$1$}(1);
            \draw[Arete,Violet!80](4)edge node[EtiqCoul]{$2$}(3);
            \draw[Arete,Violet!80](5)edge node[EtiqCoul,right]{$1$}(3);
            \draw[Arete,Violet!80](6)edge node[EtiqCoul]{$1$}(5);
            \draw[Arete,Violet!80](7)edge node[EtiqCoul,right]{$2$}(5);
            \draw[Arete,Violet!80](3)edge node[EtiqCoul]{$2$}(3,1.5);
        \end{tikzpicture}}\end{split}
        \enspace \begin{split} = \end{split} \enspace
        \begin{split}\scalebox{.25}{
        \begin{tikzpicture}
            \node[Noeud,Clair](0)at(-3,-4){$\La$};
            \node[Feuille](1)at(-4,-6){};
            \node[Feuille](1')at(-2,-6){};
            \node[Noeud,Clair](2)at(-.5,-2){$\La$};
            \node[Noeud,Clair](4)at(4,0){$\La$};
            \node[Feuille](5)at(6,-6){};
            \node[Noeud,Clair](6)at(7,-4){$\Lb$};
            \node[Feuille](7)at(8,-6){};
            \node[Noeud,Clair](8)at(10,-2){$\Lc$};
            \node[Noeud,Clair](9)at(10,-4){$\Lb$};
            \node[Feuille](10)at(9,-6){};
            \node[Feuille](10')at(11,-6){};
            \node[Feuille](11)at(12,-6){};
            \node[Noeud,Clair](12)at(13,-4){$\La$};
            \node[Feuille](13)at(14,-6){};
            \draw[Arete](0)edge node[EtiqCoul]{$1$}(2);
            \draw[Arete](1)edge node[EtiqCoul]{$1$}(0);
            \draw[Arete](1')edge node[EtiqCoul,right]{$2$}(0);
            \draw[Arete](2)edge node[EtiqCoul]{$1$}(4);
            \draw[Arete](5)edge node[EtiqCoul]{$2$}(6);
            \draw[Arete](6)edge node[EtiqCoul]{$2$}(8);
            \draw[Arete](7)edge node[EtiqCoul,right]{$1$}(6);
            \draw[Arete](8)edge node[EtiqCoul,right]{$2$}(4);
            \draw[Arete](9)edge node[EtiqCoul]{$2$}(8);
            \draw[Arete](10)edge node[EtiqCoul]{$2$}(9);
            \draw[Arete](10')edge node[EtiqCoul,right]{$1$}(9);
            \draw[Arete](11)edge node[EtiqCoul]{$1$}(12);
            \draw[Arete](12)edge node[EtiqCoul,right]{$1$}(8);
            \draw[Arete](13)edge node[EtiqCoul,right]{$2$}(12);
            \draw[Arete](4)edge node[EtiqCoul]{$1$}(4,1.5);
            \node[Feuille](0y)at(-1,-8){};
            \node[Noeud,Marque1](1y)at(0,-6){$\Lb$};
            \node[Feuille](2y)at(1,-8){};
            \node[Noeud,Marque1](3y)at(2,-4){$\Lc$};
            \node[Feuille](4y)at(2,-6){};
            \node[Noeud,Marque1](5y)at(4,-6){$\La$};
            \node[Feuille](6y)at(3,-8){};
            \node[Feuille](7y)at(5,-8){};
            \draw[Arete,Violet!80](0y)edge node[EtiqCoul]{$2$}(1y);
            \draw[Arete,Violet!80](1y)edge node[EtiqCoul]{$2$}(3y);
            \draw[Arete,Violet!80](2y)edge node[EtiqCoul,right]{$1$}(1y);
            \draw[Arete,Violet!80](4y)edge node[EtiqCoul]{$2$}(3y);
            \draw[Arete,Violet!80](5y)edge node[EtiqCoul,right]{$1$}(3y);
            \draw[Arete,Violet!80](6y)edge node[EtiqCoul]{$1$}(5y);
            \draw[Arete,Violet!80](7y)edge node[EtiqCoul,right]{$2$}(5y);
            \draw[Arete,Violet!80](3y)edge node[EtiqCoul,right]{$2$}(2);
        \end{tikzpicture}}\end{split}
    \end{equation*}
    \vspace{-1.5em}
    \caption{An example of a composition in the free coloured operad
    generated by the coloured collection defined in
    Figure \ref{fig:arbre_syntaxique_colore}.}
    \label{fig:composition_operade_coloree_libre}
\end{figure}
\medskip

\subsection{The construction}
Let us now introduce the construction associating a (noncoloured) operad
with a coloured one. We begin by giving the formal definition of what
enveloping operads of coloured operads are, and then, give a combinatorial
interpretation of the construction in terms of anticoloured syntax trees.
\medskip

\subsubsection{Enveloping operads}
Let $\Cca$ be a $k$-coloured operad. We denote by $\Cca^+$ the $k$-coloured
collection $\Cca \setminus \Cca(1)$ and by $\bar \Cca$ the $1$-coloured
collection consisting in the elements of $\Cca^+$ with $1$ as output
and input colours. The {\em enveloping operad} $\OpEnv(\Cca)$ of $\Cca$
is the smallest (noncoloured) operad containing $\Cca^+$. In other terms,
\begin{equation}
    \OpEnv(\Cca) := \OpLibre\left(\bar\Cca\right)/_\equiv,
\end{equation}
where $\equiv$ is the smallest operadic congruence of $\OpLibre\left(\bar\Cca\right)$
satisfying
\begin{equation}
    \Corolle(x) \circ_i \Corolle(y) \equiv \Corolle(x \circ_i y),
\end{equation}
for all $x, y \in \bar\Cca$ such that $x \circ_i y$ is well-defined
in $\Cca$.
\medskip

\subsubsection{Reductions}
Let $T$ be a $1$-coloured syntax tree of $\OpLibre\left(\bar\Cca\right)$
and $e$ be an edge of $T$ connecting two internal nodes $r$ and $s$
respectively labeled by $x$ and $y$, such that $s$ is the $i$th child of
$r$ and, as elements of $\Cca$, $\In_i(x) = \Out(y)$. Then, the {\em reduction}
of $T$ with respect to $e$ is the tree obtained by replacing $r$ and $s$
by an internal node labeled by $x \circ_i y$ (see Figure \ref{fig:reduction}).
This is another element of $\OpLibre\left(\bar\Cca\right)$.
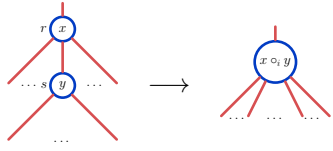
\begin{figure}[ht]
    \begin{equation*}
        \begin{split}
        \scalebox{.25}{\begin{tikzpicture}
            \node[Noeud,Clair](1)at(0,0){$x$};
            \node[Noeud,Clair](2)at(0,-3){$y$};
            \node(1G)at(-3,-3){};
            \node(1D)at(3,-3){};
            \node(2G)at(-3,-6){};
            \node(2D)at(3,-6){};
            \node(v)at(0,1.5){};
            \draw[Arete](1)--(2);
            \draw[Arete](1)--(1G);
            \draw[Arete](1)--(1D);
            \draw[Arete](2)--(2G);
            \draw[Arete](2)--(2D);
            \draw[Arete](1)--(v);
            \node at(-1.75,-3){\Huge $\dots$};
            \node at(1.75,-3){\Huge $\dots$};
            \node at(0,-6){\Huge $\dots$};
            \node at(-1,0){\Huge $r$};
            \node at(-1,-3){\Huge $s$};
        \end{tikzpicture}}
        \end{split}
    \begin{split} \quad \longrightarrow \quad \end{split}
    \begin{split}
        \scalebox{.25}{\begin{tikzpicture}
            \node[Noeud,Clair,minimum size=22mm](1)at(0,0){$x \circ_i y$};
            \node(v)at(0,2){};
            \node(1G)at(-3,-3){};
            \node(1D)at(3,-3){};
            \node(2G)at(-1.5,-3){};
            \node(2D)at(1.5,-3){};
            \draw[Arete](1)--(v);
            \draw[Arete](1)--(1G);
            \draw[Arete](1)--(1D);
            \draw[Arete](1)--(2G);
            \draw[Arete](1)--(2D);
            \node at(-2,-3){\Huge $\dots$};
            \node at(2,-3){\Huge $\dots$};
            \node at(0,-3){\Huge $\dots$};
        \end{tikzpicture}}
        \end{split}
    \end{equation*}
    \vspace{-1.5em}
    \caption{The reduction of $1$-coloured syntax trees. The internal
    node $s$ is the $i$th child of $r$.}
    \label{fig:reduction}
\end{figure}
\medskip

\subsubsection{Anticoloured syntax trees}
For any $k$-coloured collection $C$, we denote by $\AntiCol(C)$ the set
of {\em $k$-anticoloured syntax trees on $C$}, that are planar rooted
trees such that
\begin{enumerate}
    \item internal nodes of arity $\ell$ are labeled by elements of $C(\ell)$;
    \item for any internal nodes $r$ and $s$ such that $s$ is the $i$th
    child of $r$, we have $\In_i(x) \ne \Out(y)$ where $x$ (resp. $y$)
    is the label of $r$ (resp. $s$).
\end{enumerate}
The same terminology as the one introduced in
Section \ref{subsubsec:arbres_syntaxiques_colores} for coloured syntax
trees remains valid for anticoloured syntax trees.
\medskip

\subsubsection{The operad of anticoloured syntax trees}
If $\Cca$ is a coloured operad, the set $\AntiCol(\Cca^+)$, together with
the unit $\Unite$, is endowed with an operad structure for the composition
defined as follows. Let $S$ and $T$ be two anticoloured syntax trees on
$\Cca^+$. If $\Out(T) \ne \In_i(S)$, $S \circ_i T$ is the anticoloured
syntax tree obtained by grafting the root of $T$ on the $i$th leaf of $S$.
Otherwise, when $\Out(T) = \In_i(S)$, $S \circ_i T$ is the anticoloured
syntax tree obtained by grafting the root of $T$ on the $i$th leaf of $S$
and then, by reducing the obtained tree with respect to the edge connecting
the nodes $r$ and $s$, where $r$ is the parent of the $i$th leaf of $S$
and $s$ is the root of $T$ (see Figure \ref{fig:composition_arbres_anticolores}).
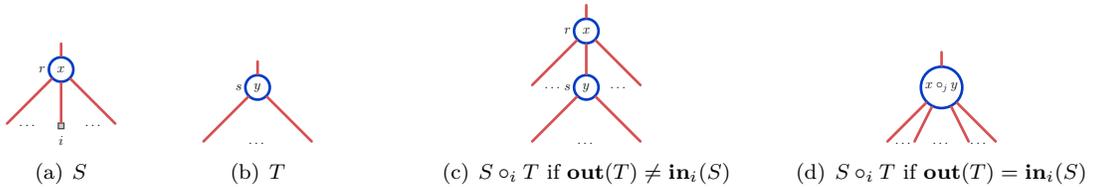
\begin{figure}[ht]
    \centering
    \subfigure[$S$]{
        \scalebox{.25}{\begin{tikzpicture}
            \node[Noeud,Clair](1)at(0,0){$x$};
            \node(v)at(0,1.5){};
            \node(1G)at(-3,-3){};
            \node(1D)at(3,-3){};
            \node[Feuille](f)at(0,-3){};
            \draw[Arete](1)--(1G);
            \draw[Arete](1)--(1D);
            \draw[Arete](1)--(v);
            \draw[Arete](1)--(f);
            \node at(0,-3.75){\Huge $i$};
            \node at(-1.75,-3){\Huge $\dots$};
            \node at(1.75,-3){\Huge $\dots$};
            \node at(-1,0){\Huge $r$};
        \end{tikzpicture}}
    }
    \qquad
    \subfigure[$T$]{
        \scalebox{.25}{\begin{tikzpicture}
            \node[Noeud,Clair](1)at(0,0){$y$};
            \node(v)at(0,1.5){};
            \node(1G)at(-3,-3){};
            \node(1D)at(3,-3){};
            \draw[Arete](1)--(v);
            \draw[Arete](1)--(1G);
            \draw[Arete](1)--(1D);
            \node at(0,-3){\Huge $\dots$};
            \node at(-1,0){\Huge $s$};
        \end{tikzpicture}}
    }
    \qquad \qquad
    \subfigure[$S \circ_i T$ if $\Out(T) \ne \In_i(S)$]{
        \qquad \quad
        \scalebox{.25}{\begin{tikzpicture}
            \node[Noeud,Clair](1)at(0,0){$x$};
            \node[Noeud,Clair](2)at(0,-3){$y$};
            \node(1G)at(-3,-3){};
            \node(1D)at(3,-3){};
            \node(2G)at(-3,-6){};
            \node(2D)at(3,-6){};
            \node(v)at(0,1.5){};
            \draw[Arete](1)--(2);
            \draw[Arete](1)--(1G);
            \draw[Arete](1)--(1D);
            \draw[Arete](2)--(2G);
            \draw[Arete](2)--(2D);
            \draw[Arete](1)--(v);
            \node at(-1.75,-3){\Huge $\dots$};
            \node at(1.75,-3){\Huge $\dots$};
            \node at(0,-6){\Huge $\dots$};
            \node at(-1,0){\Huge $r$};
            \node at(-1,-3){\Huge $s$};
        \end{tikzpicture}}
        \qquad \quad
    }
    \qquad
    \subfigure[$S \circ_i T$ if $\Out(T) = \In_i(S)$]{
        \qquad \quad
        \scalebox{.25}{\begin{tikzpicture}
            \node[Noeud,Clair,minimum size=22mm](1)at(0,0){$x \circ_j y$};
            \node(v)at(0,2){};
            \node(1G)at(-3,-3){};
            \node(1D)at(3,-3){};
            \node(2G)at(-1.5,-3){};
            \node(2D)at(1.5,-3){};
            \draw[Arete](1)--(v);
            \draw[Arete](1)--(1G);
            \draw[Arete](1)--(1D);
            \draw[Arete](1)--(2G);
            \draw[Arete](1)--(2D);
            \node at(-2,-3){\Huge $\dots$};
            \node at(2,-3){\Huge $\dots$};
            \node at(0,-3){\Huge $\dots$};
        \end{tikzpicture}}
        \qquad \quad
    }
    \caption{The two cases for the composition of two anticoloured trees
    $S$ and $T$. Here $j$ is the index of the $i$th leaf of $S$ among
    the children of the internal node $r$.}
    \label{fig:composition_arbres_anticolores}
\end{figure}
\medskip

\begin{Proposition} \label{prop:interpretation_op_env}
    For any coloured operad $\Cca$, the operads $\OpEnv(\Cca)$ and
    $\AntiCol(\Cca^+)$ are isomorphic.
\end{Proposition}
\begin{proof}
    Let $\phi : \OpEnv(\Cca) \to \AntiCol(\Cca^+)$ be the map associating
    with any $\equiv$-equivalence class of $1$-coloured syntax trees on
    $\bar \Cca$, the only anticoloured syntax tree on $\Cca^+$ belonging
    to it. To prove the statement, let us show that $\phi$ is a well-defined
    operad isomorphism.
    \smallskip

    For that, consider the rewrite rule $\mapsto$ on the $1$-coloured syntax
    trees on $\bar \Cca$ by setting $S \mapsto T$ if $T$ can be obtained
    from $S$ by a reduction. Operadic axioms ensure that $\mapsto$ is confluent,
    and since any rewriting decreases the number of internal nodes, $\mapsto$
    is terminating. The normal forms of $\mapsto$ are the trees that cannot
    be reduced, and thus, are anticoloured syntax trees on $\Cca^+$. Since
    by definition of $\equiv$, $S \mapsto T$ implies $S \equiv T$, the
    application $\phi$ is well-defined and is a bijection.
    \smallskip

    Finally, let $[S]_\equiv, [T]_\equiv \in \OpEnv(\Cca)$,
    $S := \phi([S]_\equiv)$, and $T := \phi([T]_\equiv)$. The only
    anticoloured syntax tree in $[S \circ_i T]_\equiv$ is obtained by
    grafting $S$ and $T$ together and performing, if possible, a reduction
    with respect to the edge linking these. Since the obtained tree is also
    the anticoloured syntax tree $S \circ_i T$ of $\AntiCol(\Cca^+)$, $\phi$
    is an operad morphism.
\end{proof}
\medskip

Proposition \ref{prop:interpretation_op_env} implies that the elements of
$\OpEnv(\Cca)$ can be regarded as anticoloured trees, endowed with their
composition defined above. We shall maintain this point of view in the rest
of this paper by setting $\OpEnv(\Cca) := \AntiCol(\Cca^+)$.
\medskip

\subsubsection{Functoriality}
Let $\Cca_1$ and $\Cca_2$ be two $k$-coloured operads. Recall that a map
$\phi : \Cca_1 \to \Cca_2$ is a {\em coloured operad morphism} if it preserves
the arities, commutes with compositions maps and, for any $x, y \in \Cca_1$
and $i \in [|x|]$, if the composition $x \circ_i y$ is defined in $\Cca_1$,
then the composition $\phi(x) \circ_i \phi(y)$ is defined in $\Cca_2$.
\medskip

Given $\phi : \Cca_1 \to \Cca_2$ a coloured operad morphism, let the map
$\OpEnv(\phi) : \OpEnv(\Cca_1) \to \OpEnv(\Cca_2)$ be the unique operad
morphism satisfying
\begin{equation}
    \OpEnv(\phi)(\Corolle(x)) := \Corolle(\phi(x))
\end{equation}
for any $x \in \Cca_1$.
\medskip

\begin{Theoreme} \label{thm:fonctiorialite}
    The construction $\OpEnv$ is a functor from the category of coloured
    operads to the category of operads that preserves injections and
    surjections.
\end{Theoreme}
\begin{proof}
    For any coloured operad $\Cca$, $\OpEnv(\Cca)$ is by definition an
    operad on anticoloured syntax trees on $\Cca^+$. Moreover, by induction
    on the number of internal nodes of the anticoloured syntax trees, it
    follows that for any coloured operad morphism $\phi$, $\OpEnv(\phi)$
    is a well-defined operad morphism.
    \smallskip

    Since $\OpEnv$ is compatible with map composition and sends the identity
    coloured operad morphism to the identity operad morphism, $\OpEnv$ is
    a functor. It is moreover plain that if $\phi$ is an injective
    (resp. surjective) coloured operad morphism, then $\OpEnv(\phi)$ is
    an injective (resp. surjective) operad morphism.
\end{proof}
\medskip

Theorem \ref{thm:fonctiorialite} is rich in consequences:
Propositions \ref{prop:sous_operades_et_quotients}, \ref{prop:famille_generatrice},
\ref{prop:presentation}, \ref{prop:symetries} of next Section directly
rely on it.
\medskip

Notice that $\OpEnv$ is a surjective functor. Indeed, since an anticoloured
syntax tree on a $1$-coloured collection is necessarily a corolla,
for any operad $\Oca$, $\OpEnv(\Oca)$ contains only corollas labeled on $\Oca^+$
and it is therefore isomorphic to $\Oca$.
\medskip

Notice also that $\OpEnv$ is not an injective functor. Let us exhibit
two $2$-coloured operads not themselves isomorphic that produce by $\OpEnv$
two isomorphic operads. Let $\Cca$ be the $2$-coloured operad where
$\Cca(2) := \{\alpha_2\}$ with $\Out(\alpha_2) := 1$ and
$\In_1(\alpha_2) := \In_2(\alpha_2) := 2$. and for all $n \geq 3$,
$\Cca(n) := \emptyset$. Due to the output and input colours of $\alpha_2$,
there is no nontrivial composition in $\Cca$. On the other hand, let
$\FAs$ be the $2$-coloured operad where, for all $n \geq 2$, $\FAs(n) := \{\beta_n\}$
with $\Out(\beta_n) := 1$, $\In_1(\beta_n) := 1$, and $\In_i(\beta_n) := 2$
for all $2 \leq i \leq n$. Nontrivial compositions of $\FAs$ are only
defined for the first position by $\beta_n \circ_1 \beta_m := \beta_{n + m - 1}$,
for any $n, m \geq 2$. One observes that $\OpEnv(\Cca)$ and $\OpEnv(\FAs)$
are both the free operad generated by one element of arity $2$ with no
nontrivial relations, and hence, are isomorphic. The isomorphism
between $\OpEnv(\Cca)$ and $\OpEnv(\FAs)$ can be described by a left-child
right-sibling bijection \cite{CLRS09} between binary trees and planar
rooted trees.
\medskip

\subsubsection{Example}
Consider the $2$-coloured $\FAs$ defined in the previous Section. The
elements of $\OpEnv(\FAs)$ are anticoloured syntax trees on $\FAs^+$.
Because of the output and input colours of the elements of $\FAs$,
$\OpEnv(\FAs)$ contains trees where all internal nodes have no child in
the first position. Figure \ref{fig:exemple_construction} shows some
elements of $\OpEnv(\FAs)$ and two compositions.
\begin{figure}[ht]
    \centering
    \subfigure[A composition without a reduction.]{
        \scalebox{.25}{\raisebox{-11em}{\begin{tikzpicture}
            \node[Feuille](0)at(0.00,-3.00){};
            \node[Feuille](10)at(9.00,-6.00){};
            \node[Feuille](11)at(10.00,-6.00){};
            \node[Feuille](2)at(1.00,-6.00){};
            \node[Feuille](4)at(3.00,-9.00){};
            \node[Feuille](6)at(5.00,-9.00){};
            \node[Feuille](7)at(6.00,-6.00){};
            \node[Feuille](8)at(7.00,-6.00){};
            \node[Noeud,Clair](1)at(3.00,0.00){$\beta_3$};
            \node[Noeud,Clair](3)at(2.00,-3.00){$\beta_2$};
            \node[Noeud,Clair](5)at(4.00,-6.00){$\beta_2$};
            \node[Noeud,Clair](9)at(8.00,-3.00){$\beta_4$};
            \draw[Arete](0)--(1);
            \draw[Arete](10)--(9);
            \draw[Arete](11)--(9);
            \draw[Arete](2)--(3);
            \draw[Arete](3)--(1);
            \draw[Arete](4)--(5);
            \draw[Arete](5)--(3);
            \draw[Arete](6)--(5);
            \draw[Arete](7)--(9);
            \draw[Arete](8)--(9);
            \draw[Arete](9)--(1);
        \end{tikzpicture}}}
        \enspace $\circ_7$ \enspace
        \scalebox{.25}{\raisebox{-3em}{\begin{tikzpicture}
            \node[Feuille](0)at(0.00,-2.00){};
            \node[Feuille](2)at(1.00,-2.00){};
            \node[Feuille](3)at(2.00,-2.00){};
            \node[Noeud,Clair](1)at(1.00,0.00){$\beta_3$};
            \draw[Arete](0)--(1);
            \draw[Arete](2)--(1);
            \draw[Arete](3)--(1);
        \end{tikzpicture}}}
        \enspace = \enspace
        \scalebox{.25}{\raisebox{-13em}{\begin{tikzpicture}
            \node[Feuille](0)at(0.00,-3.75){};
            \node[Feuille](10)at(9.00,-11.25){};
            \node[Feuille](12)at(10.00,-11.25){};
            \node[Feuille](13)at(11.00,-11.25){};
            \node[Feuille](14)at(12.00,-7.50){};
            \node[Feuille](2)at(1.00,-7.50){};
            \node[Feuille](4)at(3.00,-11.25){};
            \node[Feuille](6)at(5.00,-11.25){};
            \node[Feuille](7)at(6.00,-7.50){};
            \node[Feuille](8)at(7.00,-7.50){};
            \node[Noeud,Clair](1)at(3.00,0.00){$\beta_3$};
            \node[Noeud,Clair](11)at(10.00,-7.50){$\beta_3$};
            \node[Noeud,Clair](3)at(2.00,-3.75){$\beta_2$};
            \node[Noeud,Clair](5)at(4.00,-7.50){$\beta_2$};
            \node[Noeud,Clair](9)at(8.00,-3.75){$\beta_4$};
            \draw[Arete](0)--(1);
            \draw[Arete](10)--(11);
            \draw[Arete](11)--(9);
            \draw[Arete](12)--(11);
            \draw[Arete](13)--(11);
            \draw[Arete](14)--(9);
            \draw[Arete](2)--(3);
            \draw[Arete](3)--(1);
            \draw[Arete](4)--(5);
            \draw[Arete](5)--(3);
            \draw[Arete](6)--(5);
            \draw[Arete](7)--(9);
            \draw[Arete](8)--(9);
            \draw[Arete](9)--(1);
        \end{tikzpicture}}}
    }
    \qquad
    \subfigure[A composition with a reduction.]{
        \scalebox{.25}{\raisebox{-3em}{\begin{tikzpicture}
            \node[Feuille](0)at(0.00,-2.00){};
            \node[Feuille](2)at(1.00,-2.00){};
            \node[Feuille](3)at(2.00,-2.00){};
            \node[Noeud,Clair](1)at(1.00,0.00){$\beta_3$};
            \draw[Arete](0)--(1);
            \draw[Arete](2)--(1);
            \draw[Arete](3)--(1);
        \end{tikzpicture}}}
        \enspace $\circ_1$ \enspace
        \scalebox{.25}{\raisebox{-4em}{\begin{tikzpicture}
            \node[Feuille](0)at(0.00,-1.67){};
            \node[Feuille](2)at(2.00,-3.33){};
            \node[Feuille](4)at(4.00,-3.33){};
            \node[Noeud,Clair](1)at(1.00,0.00){$\beta_2$};
            \node[Noeud,Clair](3)at(3.00,-1.67){$\beta_2$};
            \draw[Arete](0)--(1);
            \draw[Arete](2)--(3);
            \draw[Arete](3)--(1);
            \draw[Arete](4)--(3);
        \end{tikzpicture}}}
        \enspace $=$ \enspace
        \scalebox{.25}{\raisebox{-7em}{\begin{tikzpicture}
            \node[Feuille](0)at(0.00,-2.33){};
            \node[Feuille](1)at(1.00,-4.67){};
            \node[Feuille](3)at(3.00,-4.67){};
            \node[Feuille](5)at(5.00,-2.33){};
            \node[Feuille](6)at(6.00,-2.33){};
            \node[Noeud,Clair](2)at(2.00,-2.33){$\beta_2$};
            \node[Noeud,Clair](4)at(4.00,0.00){$\beta_4$};
            \draw[Arete](0)--(4);
            \draw[Arete](1)--(2);
            \draw[Arete](2)--(4);
            \draw[Arete](3)--(2);
            \draw[Arete](5)--(4);
            \draw[Arete](6)--(4);
        \end{tikzpicture}}}
    }
    \caption{Two compositions in the enveloping operad of the $2$-coloured
    operad $\FAs$.}
    \label{fig:exemple_construction}
\end{figure}
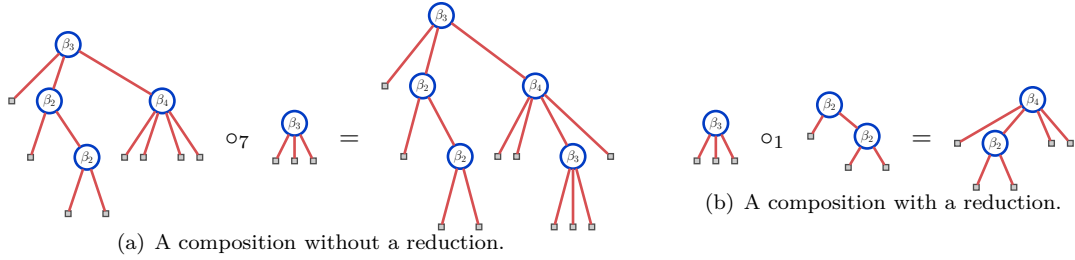
\medskip

\subsection{Bubble decompositions of operads and consequences}
\label{subsec:consequences_decomposition_bulles}
Let $\Oca$ be an operad. We say that $\Cca$ is a {\em $k$-bubble decomposition}
of $\Oca$ if $\Cca$ is a $k$-coloured operad such that $\OpEnv(\Cca)$ and
$\Oca$ are isomorphic. In this case, we say that the elements of $\Cca$
are {\em bubbles}. As we shall show, since a bubble decomposition $\Cca$
of an operad $\Oca$ contains a lot of information about $\Oca$, the study
of $\Oca$ can be confined to the study of $\Cca$. Then, the main interest
of the construction $\OpEnv$ is here: the study of an operad $\Oca$ is
confined to the study of one of its bubble decomposition. Since coloured
operads are more constrained structures than operads, this study is in
most cases simpler than the direct study of the operad itself.
\medskip

\subsubsection{Hilbert series}
The {\em coloured Hilbert series} of $\Cca$ are the commutative series
$\SerieBulles_c(z_1, \dots, z_k)$, $c \in [k]$, defined by
\begin{equation}
    \SerieBulles_c(z_1, \dots, z_k) :=
    \sum_{\substack{x \in \Cca^+ \\ \Out(x) = c}}
    \: \prod_{1 \leq i \leq |x|} z_{\In_i(x)}.
\end{equation}
The coefficient of $z_1^{\alpha_1} \dots z_k^{\alpha_k}$ in
$\SerieBulles_c(z_1, \dots, z_k)$ counts the nontrivial elements of $\Cca$
having $c$ as output colour and $\alpha_d$ inputs of colour $d$ for all
$d \in [k]$.
\medskip

As a side remark, note that one could also define as well noncommutative
analogues of these Hilbert series, so that the order of the input colours
would be taken into account.
\medskip

When $\Cca$ is an operad (or equivalently, a $1$-coloured operad),
its {\em Hilbert series} is
\begin{equation}
    \SerieOp(t) := t + \SerieBulles_1(t)
    = \sum_{n \geq 1} \# \Cca(n) \: t^n.
\end{equation}
\medskip

\begin{Proposition} \label{prop:serie_hilbert}
    Let $\Cca$ be a $k$-coloured operad.
    Then, the Hilbert series $\SerieOp(t)$ of the enveloping operad
    of $\Cca$ satisfies
    \begin{equation}
        \SerieOp(t) = t + \SerieOp_1(t) + \dots + \SerieOp_k(t),
    \end{equation}
    where for all $c \in [k]$, the series $\SerieOp_c(t)$ satisfy
    \begin{equation*}
        \SerieOp_c(t) =
        \SerieBulles_c(\SerieOp(t) - \SerieOp_1(t), \dots,
            \SerieOp(t) - \SerieOp_k(t)).
    \end{equation*}
\end{Proposition}
\begin{proof}
    Since the elements of the enveloping operad of $\Cca$ are the
    anticoloured syntax trees on $\Cca^+$, for all $c \in [k]$, the series
    $\SerieOp_c(t)$ are the series counting the anticoloured syntax trees
    on $\Cca^+$ having $c$ as output colour. The Hilbert series of the
    enveloping operad of $\Cca$ is the sum of the $\SerieOp_c(t)$ plus $t$
    in order to count the unit.
\end{proof}
\medskip

Note that Proposition \ref{prop:serie_hilbert} implies that, if the
coloured Hilbert series of $\Cca$ is algebraic, the Hilbert series of
$\OpEnv(\Cca)$ also is. Nevertheless, as we shall see, rationality is
not preserved.
\medskip

\subsubsection{Suboperads and quotients}
A $k$-coloured operad $\Cca'$ is a {\em coloured suboperad} of $\Cca$ if
for all $n \geq 1$, $\Cca'(n)$ is a subset of $\Cca(n)$ and the units of
$\Cca'$ are the same as those of $\Cca$.
\begin{Proposition} \label{prop:sous_operades_et_quotients}
    Let $\Cca$ a coloured operad and $\Cca'$ be one of its coloured
    suboperads (resp. quotients). Then, the enveloping operad of $\Cca'$
    is a suboperad (resp. quotient) of the enveloping operad of $\Cca$.
\end{Proposition}
\medskip

\subsubsection{Generating sets}
A set $G$ of elements of $\Cca$ is a {\em generating set} of $\Cca$ if
the smallest coloured suboperad of $\Cca$ containing $G$ is $\Cca$
itself and if moreover $G$ is minimal with respect to inclusion for
this property.  Notice that the generating set of $\Cca$ is unique,
given by elements that cannot be written as a non-trivial
composition. Any element $x$ of $\Cca$ can be (non necessarily
uniquely) written as $x = y \circ_i g$ where $y \in \Cca$, $i \in
[|y|]$, and $g \in G$.
\begin{Proposition} \label{prop:famille_generatrice}
    Let $\Cca$ be a coloured operad generated by $G$. Then, the enveloping
    operad of $\Cca$ is generated by
    \begin{equation}
        \OpEnv(G) := \{\Corolle(g) : g \in G\}.
    \end{equation}
\end{Proposition}
\medskip

\subsubsection{Symmetries}
A {\em symmetry} of $\Cca$ is either a coloured operad automorphism or
a coloured operad antiautomorphism. A {\em coloured operad antiautomorphism}
$\phi$ is a bijective map preserving the arities and, for any $x, y \in \Cca$
and $i \in [|x|]$, if the composition $x \circ_i y$ is defined, then the
composition $\phi(x) \circ_{|x| - i + 1} \phi(y)$ also is, and
$\phi(x \circ_i y) = \phi(x) \circ_{|x| - i + 1} \phi(y)$.
The symmetries of $\Cca$ form a group for the composition, called the
{\em group of symmetries} of $\Cca$.
\begin{Proposition} \label{prop:symetries}
    Let $\Cca$ be a coloured operad and $\Gfr$ its group of symmetries.
    Then, the group of symmetries of the enveloping operad of $\Cca$ is
    $\OpEnv(\Gfr)$ where
    \begin{equation}
        \OpEnv(\Gfr) := \{\OpEnv(\phi) : \phi \in \Gfr\}.
    \end{equation}
\end{Proposition}
\medskip

\subsubsection{Presentations by generators and relations}
A {\em presentation by generators and relations} of $\Cca$ is a pair
$(G, \leftrightarrow)$ where $G$ is a $k$-coloured collection and
$\leftrightarrow$  is the finest equivalence relation on $\OpLibre(G)$
such that $\Cca$ is isomorphic to $\OpLibre(G)/_\equiv$, $\equiv$ being
the finest coloured operadic congruence containing $\leftrightarrow$.
\begin{Proposition} \label{prop:presentation}
    Let $\Cca$ be a coloured operad admitting the presentation $(G, \leftrightarrow)$.
    Then, the enveloping operad of $\Cca$ admits the presentation
    $(\OpEnv(G), \leftrightarrow')$, where
    \begin{equation}
        S' \leftrightarrow' T'
        \quad \mbox{if and only if} \quad
        S \leftrightarrow T,
    \end{equation}
    where $S'$ (resp. $T'$) is the coloured syntax tree on $\OpEnv(G)$
    obtained by replacing any node labeled by $x$ of $S$ (resp. $T$) by
    $\Corolle(x)$.
\end{Proposition}
\medskip

\section{The operad of bicoloured noncrossing configurations} \label{sec:operade_CNCB}
In this Section, we shall define an operad over a new kind of noncrossing
configurations. In order to study it and apply the results of
Section \ref{sec:operades_enveloppantes}, we shall see it as an enveloping
operad of a coloured one.
\medskip

\subsection{Bicoloured noncrossing configurations}
Let us start by introducing our new combinatorial object, some of its
properties, and its operadic structure.
\medskip

\subsubsection{Regular polygons}
Let $\Cfr$ be a regular polygon with $n + 1$ vertices. The vertices of
$\Cfr$ are numbered in the clockwise direction from $1$ to $n + 1$. An
{\em arc} of $\Cfr$ is a tuple $(i, j)$ with $1 \leq i < j \leq n + 1$.
We call {\em diagonal} any arc $(i, j)$ different from $(i, i + 1)$ and
$(1, n + 1)$, and {\em edge} any arc of the form $(i, i + 1)$ or $(i, n + 1)$.
The {\em $i$th edge} of $\Cfr$ is the edge $(i, i + 1)$. The edge $(1, n + 1)$
is the {\em base} of $\Cfr$ (drawn at bottommost).
\medskip

\subsubsection{Bicoloured noncrossing configurations}
A {\em bicoloured noncrossing configuration} (abbreviated as {\em BNC})
of size $n$ is a regular polygon $\Cfr$ with $n + 1$ vertices, together
with two sets of distinguished arcs: a set $\Cfr_b$ of {\em blue arcs}
(drawn as thick lines) and a set $\Cfr_r$ of {\em red arcs} (drawn as
dotted lines). If $(i, j) \in \Cfr_b$ (resp. $(i, j) \in \Cfr_r$), we say
that $(i, j)$ is {\em blue} (resp. {\em red}). Otherwise, when
$(i, j) \notin \Cfr_b \cup \Cfr_r$, we say that $(i, j)$ is {\em uncoloured}.
These two sets have to satisfy the three following properties:
\begin{enumerate}
    \item any arc is either blue, red, or uncoloured;
    \item no blue or red arc crosses another blue or red arc;
    \item all red arcs are diagonals.
\end{enumerate}
We say that $\Cfr$ is {\em based} if its base is blue and {\em nonbased}
otherwise. Besides, we impose by definition that there is only one BNC
of size $1$: the segment consisting in one blue arc. Figure \ref{fig:exemple_CNCB}
shows a BNC.
\begin{figure}[ht]
    \centering
    \scalebox{.8}{\begin{tikzpicture}
        \node[shape=coordinate](0)at(-0.3,-0.95){};
        \node[shape=coordinate](1)at(-0.8,-0.58){};
        \node[shape=coordinate](2)at(-1.,-0.){};
        \node[shape=coordinate](3)at(-0.8,0.59){};
        \node[shape=coordinate](4)at(-0.3,0.96){};
        \node[shape=coordinate](5)at(0.31,0.96){};
        \node[shape=coordinate](6)at(0.81,0.59){};
        \node[shape=coordinate](7)at(1.,0.01){};
        \node[shape=coordinate](8)at(0.81,-0.58){};
        \node[shape=coordinate](9)at(0.31,-0.95){};
        \draw[CotePolyg](0)--(1);
        \draw[CotePolyg](1)--(2);
        \draw[CotePolyg](2)--(3);
        \draw[CotePolyg](3)--(4);
        \draw[CotePolyg](4)--(5);
        \draw[CotePolyg](5)--(6);
        \draw[CotePolyg](6)--(7);
        \draw[CotePolyg](7)--(8);
        \draw[CotePolyg](8)--(9);
        \draw[CotePolyg](9)--(0);
        \draw[ArcBleu](0)--(1);
        \draw[ArcBleu](1)--(7);
        \draw[ArcBleu](3)--(5);
        \draw[ArcBleu](6)--(7);
        \draw[ArcBleu](8)--(9);
        \draw[ArcRouge](1)--(5);
        \draw[ArcRouge](1)--(9);
        \node at(-.6,-.9){\scriptsize $1$};
        \node at(.6,-.9){\scriptsize $9$};
        \node at(-1,-.3){\scriptsize $2$};
        \node at(1,-.3){\scriptsize $8$};
        \node at(-1,.3){\scriptsize $3$};
        \node at(1,.3){\scriptsize $7$};
        \node at(-.6,.9){\scriptsize $4$};
        \node at(.6,.9){\scriptsize $6$};
        \node at(0,1.1){\scriptsize $5$};
    \end{tikzpicture}}
    \caption{A nonbased BNC of size $9$. Each edge $(i, i + 1)$ is numbered
    by $i$. Blue arcs are $(1, 2)$, $(2, 8)$, $(4, 6)$, $(7, 8)$, and
    $(9, 10)$, and red arcs are $(2, 6)$ and $(2, 10)$.}
    \label{fig:exemple_CNCB}
\end{figure}
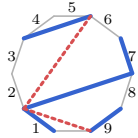
\medskip

\subsubsection{Borders}
When the size of $\Cfr$ is not smaller than $2$, the {\em border} of
$\Cfr$ is the word $\Bord(\Cfr)$ of length $n$ such that, for any
$i \in [n]$, $\Bord(\Cfr)_i := 1$ if the $i$th edge of $\Cfr$ is uncoloured
and $\Bord(\Cfr)_i := 2$ otherwise. For instance, the border of the BNC
of Figure \ref{fig:exemple_CNCB} is $211111212$.
\medskip

\subsubsection{Operad structure}
From now, the {\em arity} $|\Cfr|$ of a BNC $\Cfr$ is its size. Let
$\Cfr$ and $\Dfr$ be two BNCs of respective arities $n$ and $m$, and
$i \in [n]$. The composition $\Cfr \circ_i \Dfr =: \Efr$ is obtained by
gluing the base of $\Dfr$ onto the $i$th edge of $\Cfr$, and then,
\begin{enumerate}
    \item if the base of $\Dfr$ and the $i$th edge of $\Cfr$ are both
    uncoloured, the arc $(i, i + m)$ of $\Efr$ becomes red;
    \item if the base of $\Dfr$ and the $i$th edge of $\Cfr$ are both
    blue, the arc $(i, i + m)$ of $\Efr$ becomes blue;
    \item otherwise, the base of $\Dfr$ and the $i$th edge of $\Cfr$
    have different colours; in this case, the arc $(i, i + m)$ of $\Efr$
    is uncoloured.
\end{enumerate}
For aesthetic reasons, the resulting shape is reshaped to form a regular
polygon. Figure \ref{fig:exemples_composition_CNCB} shows examples of
composition of BNCs.
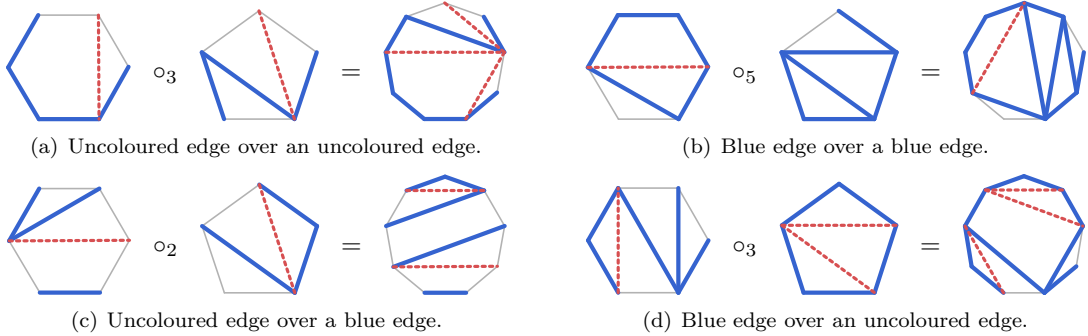
\begin{figure}[ht]
    \subfigure[Uncoloured edge over an uncoloured edge.]{
        \scalebox{.8}{\begin{tikzpicture}
            \node[shape=coordinate](0)at(-0.49,-0.86){};
            \node[shape=coordinate](1)at(-1.,-0.){};
            \node[shape=coordinate](2)at(-0.5,0.87){};
            \node[shape=coordinate](3)at(0.5,0.87){};
            \node[shape=coordinate](4)at(1.,0.01){};
            \node[shape=coordinate](5)at(0.51,-0.86){};
            \draw[CotePolyg](0)--(1);
            \draw[CotePolyg](1)--(2);
            \draw[CotePolyg](2)--(3);
            \draw[CotePolyg](3)--(4);
            \draw[CotePolyg](4)--(5);
            \draw[CotePolyg](5)--(0);
            \draw[ArcBleu](0)--(1);
            \draw[ArcBleu](1)--(2);
            \draw[ArcBleu](4)--(5);
            \draw[ArcBleu](0)--(5);
            \draw[ArcRouge](3)--(5);
        \end{tikzpicture}}
        \enspace \raisebox{1.6em}{$\circ_3$} \enspace
        \scalebox{.8}{\begin{tikzpicture}
            \node[shape=coordinate](0)at(-0.58,-0.8){};
            \node[shape=coordinate](1)at(-0.95,0.31){};
            \node[shape=coordinate](2)at(-0.,1.){};
            \node[shape=coordinate](3)at(0.96,0.31){};
            \node[shape=coordinate](4)at(0.59,-0.8){};
            \draw[CotePolyg](0)--(1);
            \draw[CotePolyg](1)--(2);
            \draw[CotePolyg](2)--(3);
            \draw[CotePolyg](3)--(4);
            \draw[CotePolyg](4)--(0);
            \draw[ArcBleu](0)--(1);
            \draw[ArcBleu](1)--(4);
            \draw[ArcBleu](3)--(4);
            \draw[ArcRouge](2)--(4);
        \end{tikzpicture}}
        \enspace \raisebox{1.6em}{$=$} \enspace
        \scalebox{.8}{\begin{tikzpicture}
            \node[shape=coordinate](0)at(-0.34,-0.93){};
            \node[shape=coordinate](1)at(-0.86,-0.5){};
            \node[shape=coordinate](2)at(-0.98,0.18){};
            \node[shape=coordinate](3)at(-0.64,0.77){};
            \node[shape=coordinate](4)at(-0.,1.){};
            \node[shape=coordinate](5)at(0.65,0.77){};
            \node[shape=coordinate](6)at(0.99,0.18){};
            \node[shape=coordinate](7)at(0.87,-0.49){};
            \node[shape=coordinate](8)at(0.35,-0.93){};
            \draw[CotePolyg](0)--(1);
            \draw[CotePolyg](1)--(2);
            \draw[CotePolyg](2)--(3);
            \draw[CotePolyg](3)--(4);
            \draw[CotePolyg](4)--(5);
            \draw[CotePolyg](5)--(6);
            \draw[CotePolyg](6)--(7);
            \draw[CotePolyg](7)--(8);
            \draw[CotePolyg](8)--(0);
            \draw[ArcBleu](0)--(1);
            \draw[ArcBleu](1)--(2);
            \draw[ArcBleu](2)--(3);
            \draw[ArcBleu](3)--(6);
            \draw[ArcBleu](5)--(6);
            \draw[ArcBleu](7)--(8);
            \draw[ArcBleu](0)--(8);
            \draw[ArcRouge](2)--(6);
            \draw[ArcRouge](4)--(6);
            \draw[ArcRouge](6)--(8);
        \end{tikzpicture}}
    }
    \qquad
    \subfigure[Blue edge over a blue edge.]{
        \scalebox{.8}{\begin{tikzpicture}
            \node[shape=coordinate](0)at(-0.49,-0.86){};
            \node[shape=coordinate](1)at(-1.,-0.){};
            \node[shape=coordinate](2)at(-0.5,0.87){};
            \node[shape=coordinate](3)at(0.5,0.87){};
            \node[shape=coordinate](4)at(1.,0.01){};
            \node[shape=coordinate](5)at(0.51,-0.86){};
            \draw[CotePolyg](0)--(1);
            \draw[CotePolyg](1)--(2);
            \draw[CotePolyg](2)--(3);
            \draw[CotePolyg](3)--(4);
            \draw[CotePolyg](4)--(5);
            \draw[CotePolyg](5)--(0);
            \draw[ArcBleu](1)--(2);
            \draw[ArcBleu](1)--(5);
            \draw[ArcBleu](2)--(3);
            \draw[ArcBleu](3)--(4);
            \draw[ArcBleu](4)--(5);
            \draw[ArcRouge](1)--(4);
        \end{tikzpicture}}
        \enspace \raisebox{1.6em}{$\circ_5$} \enspace
        \scalebox{.8}{\begin{tikzpicture}
            \node[shape=coordinate](0)at(-0.58,-0.8){};
            \node[shape=coordinate](1)at(-0.95,0.31){};
            \node[shape=coordinate](2)at(-0.,1.){};
            \node[shape=coordinate](3)at(0.96,0.31){};
            \node[shape=coordinate](4)at(0.59,-0.8){};
            \draw[CotePolyg](0)--(1);
            \draw[CotePolyg](1)--(2);
            \draw[CotePolyg](2)--(3);
            \draw[CotePolyg](3)--(4);
            \draw[CotePolyg](4)--(0);
            \draw[ArcBleu](0)--(1);
            \draw[ArcBleu](1)--(3);
            \draw[ArcBleu](1)--(4);
            \draw[ArcBleu](2)--(3);
            \draw[ArcBleu](3)--(4);
            \draw[ArcBleu](0)--(4);
        \end{tikzpicture}}
        \enspace \raisebox{1.6em}{$=$} \enspace
        \scalebox{.8}{\begin{tikzpicture}
            \node[shape=coordinate](0)at(-0.34,-0.93){};
            \node[shape=coordinate](1)at(-0.86,-0.5){};
            \node[shape=coordinate](2)at(-0.98,0.18){};
            \node[shape=coordinate](3)at(-0.64,0.77){};
            \node[shape=coordinate](4)at(-0.,1.){};
            \node[shape=coordinate](5)at(0.65,0.77){};
            \node[shape=coordinate](6)at(0.99,0.18){};
            \node[shape=coordinate](7)at(0.87,-0.49){};
            \node[shape=coordinate](8)at(0.35,-0.93){};
            \draw[CotePolyg](0)--(1);
            \draw[CotePolyg](1)--(2);
            \draw[CotePolyg](2)--(3);
            \draw[CotePolyg](3)--(4);
            \draw[CotePolyg](4)--(5);
            \draw[CotePolyg](5)--(6);
            \draw[CotePolyg](6)--(7);
            \draw[CotePolyg](7)--(8);
            \draw[CotePolyg](8)--(0);
            \draw[ArcBleu](1)--(2);
            \draw[ArcBleu](1)--(8);
            \draw[ArcBleu](2)--(3);
            \draw[ArcBleu](3)--(4);
            \draw[ArcBleu](4)--(5);
            \draw[ArcBleu](4)--(8);
            \draw[ArcBleu](5)--(7);
            \draw[ArcBleu](5)--(8);
            \draw[ArcBleu](6)--(7);
            \draw[ArcBleu](7)--(8);
            \draw[ArcRouge](1)--(4);
        \end{tikzpicture}}
    }
    \subfigure[Uncoloured edge over a blue edge.]{
        \scalebox{.8}{\begin{tikzpicture}
            \node[shape=coordinate](0)at(-0.49,-0.86){};
            \node[shape=coordinate](1)at(-1.,-0.){};
            \node[shape=coordinate](2)at(-0.5,0.87){};
            \node[shape=coordinate](3)at(0.5,0.87){};
            \node[shape=coordinate](4)at(1.,0.01){};
            \node[shape=coordinate](5)at(0.51,-0.86){};
            \draw[CotePolyg](0)--(1);
            \draw[CotePolyg](1)--(2);
            \draw[CotePolyg](2)--(3);
            \draw[CotePolyg](3)--(4);
            \draw[CotePolyg](4)--(5);
            \draw[CotePolyg](5)--(0);
            \draw[ArcBleu](1)--(2);
            \draw[ArcBleu](1)--(3);
            \draw[ArcBleu](0)--(5);
            \draw[ArcRouge](1)--(4);
        \end{tikzpicture}}
        \enspace \raisebox{1.6em}{$\circ_2$} \enspace
        \scalebox{.8}{\begin{tikzpicture}
            \node[shape=coordinate](0)at(-0.58,-0.8){};
            \node[shape=coordinate](1)at(-0.95,0.31){};
            \node[shape=coordinate](2)at(-0.,1.){};
            \node[shape=coordinate](3)at(0.96,0.31){};
            \node[shape=coordinate](4)at(0.59,-0.8){};
            \draw[CotePolyg](0)--(1);
            \draw[CotePolyg](1)--(2);
            \draw[CotePolyg](2)--(3);
            \draw[CotePolyg](3)--(4);
            \draw[CotePolyg](4)--(0);
            \draw[ArcBleu](1)--(4);
            \draw[ArcBleu](2)--(3);
            \draw[ArcBleu](3)--(4);
        \draw[ArcRouge](2)--(4);
        \end{tikzpicture}}
        \enspace \raisebox{1.6em}{$=$} \enspace
        \scalebox{.8}{\begin{tikzpicture}
            \node[shape=coordinate](0)at(-0.34,-0.93){};
            \node[shape=coordinate](1)at(-0.86,-0.5){};
            \node[shape=coordinate](2)at(-0.98,0.18){};
            \node[shape=coordinate](3)at(-0.64,0.77){};
            \node[shape=coordinate](4)at(-0.,1.){};
            \node[shape=coordinate](5)at(0.65,0.77){};
            \node[shape=coordinate](6)at(0.99,0.18){};
            \node[shape=coordinate](7)at(0.87,-0.49){};
            \node[shape=coordinate](8)at(0.35,-0.93){};
            \draw[CotePolyg](0)--(1);
            \draw[CotePolyg](1)--(2);
            \draw[CotePolyg](2)--(3);
            \draw[CotePolyg](3)--(4);
            \draw[CotePolyg](4)--(5);
            \draw[CotePolyg](5)--(6);
            \draw[CotePolyg](6)--(7);
            \draw[CotePolyg](7)--(8);
            \draw[CotePolyg](8)--(0);
            \draw[ArcBleu](1)--(6);
            \draw[ArcBleu](2)--(5);
            \draw[ArcBleu](3)--(4);
            \draw[ArcBleu](4)--(5);
            \draw[ArcBleu](0)--(8);
            \draw[ArcRouge](1)--(7);
            \draw[ArcRouge](3)--(5);
        \end{tikzpicture}}
    }
    \qquad
    \subfigure[Blue edge over an uncoloured edge.]{
        \scalebox{.8}{\begin{tikzpicture}
            \node[shape=coordinate](0)at(-0.50,-0.87){};
            \node[shape=coordinate](1)at(-1.00,-0.00){};
            \node[shape=coordinate](2)at(-0.50,0.87){};
            \node[shape=coordinate](3)at(0.50,0.87){};
            \node[shape=coordinate](4)at(1.00,0.00){};
            \node[shape=coordinate](5)at(0.50,-0.87){};
            \draw[CotePolyg](0)--(1);
            \draw[CotePolyg](1)--(2);
            \draw[CotePolyg](2)--(3);
            \draw[CotePolyg](3)--(4);
            \draw[CotePolyg](4)--(5);
            \draw[CotePolyg](5)--(0);
            \draw[ArcBleu](0)--(1);
            \draw[ArcBleu](1)--(2);
            \draw[ArcBleu](2)--(5);
            \draw[ArcBleu](3)--(5);
            \draw[ArcBleu](4)--(5);
            \draw[ArcRouge](0)--(2);
        \end{tikzpicture}}
        \enspace \raisebox{1.6em}{$\circ_3$} \enspace
        \scalebox{.8}{\begin{tikzpicture}
            \node[shape=coordinate](0)at(-0.59,-0.81){};
            \node[shape=coordinate](1)at(-0.95,0.31){};
            \node[shape=coordinate](2)at(-0.00,1.00){};
            \node[shape=coordinate](3)at(0.95,0.31){};
            \node[shape=coordinate](4)at(0.59,-0.81){};
            \draw[CotePolyg](0)--(1);
            \draw[CotePolyg](1)--(2);
            \draw[CotePolyg](2)--(3);
            \draw[CotePolyg](3)--(4);
            \draw[CotePolyg](4)--(0);
            \draw[ArcBleu](0)--(4);
            \draw[ArcBleu](0)--(1);
            \draw[ArcBleu](1)--(2);
            \draw[ArcBleu](2)--(3);
            \draw[ArcBleu](3)--(4);
            \draw[ArcRouge](1)--(3);
            \draw[ArcRouge](1)--(4);
        \end{tikzpicture}}
        \enspace \raisebox{1.6em}{$=$} \enspace
        \scalebox{.8}{\begin{tikzpicture}
            \node[shape=coordinate](0)at(-0.34,-0.94){};
            \node[shape=coordinate](1)at(-0.87,-0.50){};
            \node[shape=coordinate](2)at(-0.98,0.17){};
            \node[shape=coordinate](3)at(-0.64,0.77){};
            \node[shape=coordinate](4)at(-0.00,1.00){};
            \node[shape=coordinate](5)at(0.64,0.77){};
            \node[shape=coordinate](6)at(0.98,0.17){};
            \node[shape=coordinate](7)at(0.87,-0.50){};
            \node[shape=coordinate](8)at(0.34,-0.94){};
            \draw[CotePolyg](0)--(1);
            \draw[CotePolyg](1)--(2);
            \draw[CotePolyg](2)--(3);
            \draw[CotePolyg](3)--(4);
            \draw[CotePolyg](4)--(5);
            \draw[CotePolyg](5)--(6);
            \draw[CotePolyg](6)--(7);
            \draw[CotePolyg](7)--(8);
            \draw[CotePolyg](8)--(0);
            \draw[ArcBleu](0)--(1);
            \draw[ArcBleu](1)--(2);
            \draw[ArcBleu](2)--(3);
            \draw[ArcBleu](2)--(8);
            \draw[ArcBleu](3)--(4);
            \draw[ArcBleu](4)--(5);
            \draw[ArcBleu](5)--(6);
            \draw[ArcBleu](6)--(8);
            \draw[ArcBleu](7)--(8);
            \draw[ArcRouge](0)--(2);
            \draw[ArcRouge](3)--(5);
            \draw[ArcRouge](3)--(6);
        \end{tikzpicture}}
    }
    \caption{Four examples of composition of the operad $\CNCB$.}
    \label{fig:exemples_composition_CNCB}
\end{figure}
\medskip

\begin{Proposition} \label{prop:operade_CNCB}
    The set of the BNCs, together with the composition map $\circ_i$ and the
    BNC of arity $1$ as unit form an operad, denoted by $\CNCB$.
\end{Proposition}
\medskip

\subsection{The coloured operad of bubbles} \label{subsec:operade_Bulle}
We now define a coloured operad involving particular BNCs and perform
a complete study of it.
\medskip

\subsubsection{Bubbles}
A {\em bubble} is a BNC of size no smaller than $2$ with no diagonal
(hence the name).  Figure \ref{fig:exemple_bulle} shows an example of
a bubble.
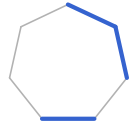
\begin{figure}[ht]
    \centering
    \scalebox{.8}{\begin{tikzpicture}
        \node[shape=coordinate](0)at(-0.43,-0.9){};
        \node[shape=coordinate](1)at(-0.97,-0.22){};
        \node[shape=coordinate](2)at(-0.78,0.63){};
        \node[shape=coordinate](3)at(-0.,1.){};
        \node[shape=coordinate](4)at(0.79,0.63){};
        \node[shape=coordinate](5)at(0.98,-0.22){};
        \node[shape=coordinate](6)at(0.44,-0.9){};
        \draw[CotePolyg](0)--(1);
        \draw[CotePolyg](1)--(2);
        \draw[CotePolyg](2)--(3);
        \draw[CotePolyg](3)--(4);
        \draw[CotePolyg](4)--(5);
        \draw[CotePolyg](5)--(6);
        \draw[CotePolyg](6)--(0);
        \draw[ArcBleu](3)--(4);
        \draw[ArcBleu](4)--(5);
        \draw[ArcBleu](0)--(6);
    \end{tikzpicture}}
    \caption{A based bubble of size $6$. Its border is $111221$.}
    \label{fig:exemple_bulle}
\end{figure}
\medskip

\subsubsection{Coloured operad structure}
Let $\Bfr$ be a bubble of arity $n$. Let us assign input an output colours
to $\Bfr$ in the following way. The output colour $\Out(\Bfr)$ of $\Bfr$
is $1$ if $\Bfr$ is based and $2$ otherwise, and the colour $\In_i(\Bfr)$
of the $i$th input of $\Bfr$ is the $i$th letter of the border of $\Bfr$.
\medskip

Let us denote by $\Unite_1$ and $\Unite_2$ two virtual bubbles of arity $1$
such that $\Out(\Unite_1) := \In_1(\Unite_1) := 1$
and $\Out(\Unite_2) := \In_i(\Unite_2) := 2$.
\medskip

\begin{Proposition} \label{prop:operade_Bulle}
    The set of bubbles, together with the composition map $\circ_i$ of
    $\CNCB$ and the units $\Unite_1$ and $\Unite_2$ form a $2$-coloured
    operad, denoted by $\Bulle$.
\end{Proposition}
\begin{proof}
    Since, as sets, $\Bulle \subseteq \CNCB$ we only have to prove that,
    when defined, the composition $\Bfr_1 \circ_i \Bfr_2$ of two bubbles
    $\Bfr_1$ and $\Bfr_2$ is a bubble. Since this composition is defined
    only if $\Out(\Bfr_2) = \In_i(\Bfr_1)$, there are two possibilities:
    either the base of $\Bfr_2$ is blue and the $i$th edge of $\Bfr_1$ is
    uncoloured, or the base of $\Bfr_2$ is uncoloured and the $i$th edge
    of $\Bfr_2$ is blue. In both cases, no diagonal is added, and hence,
    $\Bfr_1 \circ_i \Bfr_2$ is a bubble.
\end{proof}
\medskip

Notice that any bubble $\Bfr$ is wholly encoded by the pair
$\left(\Out(\Bfr), (\In_i(\Bfr))_{i \in |\Bfr|}\right)$. Therefore, $\Bulle$ is a very
simple coloured operad: for any $n$, the set of elements of arity $n$ is
$[2] \times [2]^n$ and the composition, when defined, is a substitution
in words. Figure \ref{fig:exemple_composition_Bulle} shows a composition
in $\Bulle$.
\begin{figure}[ht]
    \subfigure[A composition of two bubbles.]{
        \scalebox{.8}{\begin{tikzpicture}
            \node[shape=coordinate](0)at(-0.49,-0.86){};
            \node[shape=coordinate](1)at(-1.,-0.){};
            \node[shape=coordinate](2)at(-0.5,0.87){};
            \node[shape=coordinate](3)at(0.5,0.87){};
            \node[shape=coordinate](4)at(1.,0.01){};
            \node[shape=coordinate](5)at(0.51,-0.86){};
            \draw[CotePolyg](0)--(1);
            \draw[CotePolyg](1)--(2);
            \draw[CotePolyg](2)--(3);
            \draw[CotePolyg](3)--(4);
            \draw[CotePolyg](4)--(5);
            \draw[CotePolyg](5)--(0);
            \draw[ArcBleu](0)--(1);
            \draw[ArcBleu](1)--(2);
            \draw[ArcBleu](2)--(3);
            \draw[ArcBleu](0)--(5);
        \end{tikzpicture}}
        \enspace \raisebox{1.5em}{$\circ_3$} \enspace
        \scalebox{.8}{\begin{tikzpicture}
            \node[shape=coordinate](0)at(-0.58,-0.8){};
            \node[shape=coordinate](1)at(-0.95,0.31){};
            \node[shape=coordinate](2)at(-0.,1.){};
            \node[shape=coordinate](3)at(0.96,0.31){};
            \node[shape=coordinate](4)at(0.59,-0.8){};
            \draw[CotePolyg](0)--(1);
            \draw[CotePolyg](1)--(2);
            \draw[CotePolyg](2)--(3);
            \draw[CotePolyg](3)--(4);
            \draw[CotePolyg](4)--(0);
            \draw[ArcBleu](0)--(1);
            \draw[ArcBleu](3)--(4);
        \end{tikzpicture}}
        \enspace \raisebox{1.5em}{$=$} \enspace
        \scalebox{.8}{\begin{tikzpicture}
            \node[shape=coordinate](0)at(-0.34,-0.94){};
            \node[shape=coordinate](1)at(-0.87,-0.50){};
            \node[shape=coordinate](2)at(-0.98,0.17){};
            \node[shape=coordinate](3)at(-0.64,0.77){};
            \node[shape=coordinate](4)at(-0.00,1.00){};
            \node[shape=coordinate](5)at(0.64,0.77){};
            \node[shape=coordinate](6)at(0.98,0.17){};
            \node[shape=coordinate](7)at(0.87,-0.50){};
            \node[shape=coordinate](8)at(0.34,-0.94){};
            \draw[CotePolyg](0)--(1);
            \draw[CotePolyg](1)--(2);
            \draw[CotePolyg](2)--(3);
            \draw[CotePolyg](3)--(4);
            \draw[CotePolyg](4)--(5);
            \draw[CotePolyg](5)--(6);
            \draw[CotePolyg](6)--(7);
            \draw[CotePolyg](7)--(8);
            \draw[CotePolyg](8)--(0);
            \draw[ArcBleu](0)--(8);
            \draw[ArcBleu](0)--(1);
            \draw[ArcBleu](1)--(2);
            \draw[ArcBleu](2)--(3);
            \draw[ArcBleu](5)--(6);
        \end{tikzpicture}}
    }
    \qquad
    \subfigure[The output and input colours of the bubbles.]{
        \quad
        $(1, 22\textcolor{Vert}{2}11)
        \circ_3
        (\textcolor{Vert}{2}, \textcolor{Violet}{{\bf 2112}})
        =
        (1, 22\textcolor{Violet}{{\bf 2112}}11)$
        \quad
    }
    \caption{A composition in the $2$-coloured operad $\Bulle$.}
    \label{fig:exemple_composition_Bulle}
\end{figure}
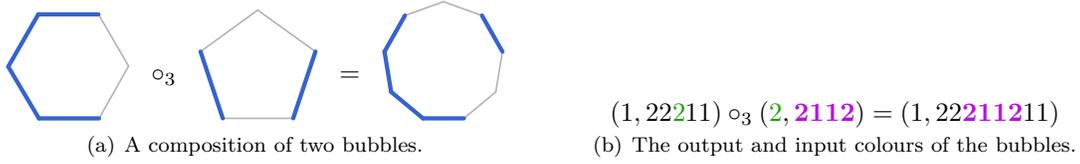
\medskip

\subsubsection{Coloured Hilbert series} \label{sec:serie_bulles_Bulle}
Since $\Bulle$ contains by definition all the bubbles, the coloured Hilbert
series of $\Bulle$ satisfy
\begin{equation} \label{eq:serie_bulles_Bulle}
    \SerieBulles_1(z_1, z_2) = \SerieBulles_2(z_1, z_2) =
    \sum_{n \geq 2} (z_1 + z_2)^n =
    \frac{(z_1 + z_2)^2}{1 - z_1 - z_2}.
\end{equation}
\medskip

\subsubsection{Generating set}
\begin{Proposition} \label{prop:generateurs_Bulle}
    The set
    \begin{equation}
        G_{\Bulle} := \{\AAA, \AAB, \BAA, \BAB, \ABA, \ABB, \BBA, \BBB\}
    \end{equation}
    of bubbles of arity $2$ is the generating set of $\Bulle$.
\end{Proposition}
\begin{proof}
    Let us proceed by induction on the arity $n$ of the bubble $\Bfr$ we
    want to generate. If $n = 2$, since the set of of bubbles of arity $2$
    is $G_{\Bulle}$, $\Bfr$ is generated by $G_{\Bulle}$. If $n \geq 3$,
    let $\Bfr'$ be the bubble obtained from $\Bfr$ by removing its last
    edge. Now, $\Bfr'$ is a bubble of arity $n - 1$, and, by induction
    hypothesis, $\Bfr'$ is generated by $G_{\Bulle}$. Since, for an
    appropriate bubble $g$ of $G_{\Bulle}$, $\Bfr = \Bfr' \circ_{n - 1} g$,
    $\Bfr$ is generated by $G_{\Bulle}$.
\end{proof}
\medskip

\subsubsection{Symmetries}
The {\em complementary} $\Cpl(\Bfr)$ of a bubble $\Bfr$ is the bubble
obtained by swapping the colours of the edges of $\Bfr$. The {\em returned}
$\Ret(\Bfr)$ of $\Bfr$ is the bubble obtained by applying on $\Bfr$ the
reflection through the vertical line passing by its base.
Figure \ref{fig:exemple_symetries} shows examples of these symmetries.
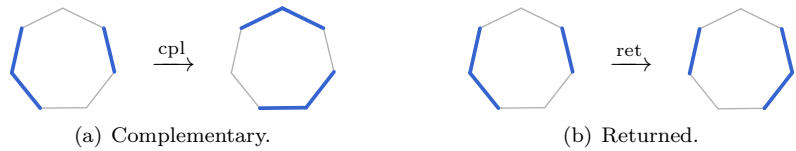
\begin{figure}[ht]
    \subfigure[Complementary.]{
        \scalebox{.7}{\begin{tikzpicture}
            \node[shape=coordinate](1)at(-0.43,-0.90){};
            \node[shape=coordinate](2)at(-0.97,-0.23){};
            \node[shape=coordinate](3)at(-0.78,0.62){};
            \node[shape=coordinate](4)at(0.00,1.0){};
            \node[shape=coordinate](5)at(0.78,0.62){};
            \node[shape=coordinate](6)at(0.98,-0.21){};
            \node[shape=coordinate](7)at(0.45,-0.89){};
            \draw[CotePolyg](1)--(7);
            \draw[CotePolyg](1)--(2);
            \draw[CotePolyg](2)--(3);
            \draw[CotePolyg](3)--(4);
            \draw[CotePolyg](4)--(5);
            \draw[CotePolyg](5)--(6);
            \draw[CotePolyg](6)--(7);
            \draw[ArcBleu](1)--(2);
            \draw[ArcBleu](2)--(3);
            \draw[ArcBleu](5)--(6);
        \end{tikzpicture}}
        \quad \raisebox{1.5em}{$\xrightarrow{\Cpl}$} \quad
        \scalebox{.7}{\begin{tikzpicture}
            \node[shape=coordinate](1)at(-0.43,-0.90){};
            \node[shape=coordinate](2)at(-0.97,-0.23){};
            \node[shape=coordinate](3)at(-0.78,0.62){};
            \node[shape=coordinate](4)at(0.00,1.0){};
            \node[shape=coordinate](5)at(0.78,0.62){};
            \node[shape=coordinate](6)at(0.98,-0.21){};
            \node[shape=coordinate](7)at(0.45,-0.89){};
            \draw[CotePolyg](1)--(7);
            \draw[CotePolyg](1)--(2);
            \draw[CotePolyg](2)--(3);
            \draw[CotePolyg](3)--(4);
            \draw[CotePolyg](4)--(5);
            \draw[CotePolyg](5)--(6);
            \draw[CotePolyg](6)--(7);
            \draw[ArcBleu](1)--(7);
            \draw[ArcBleu](3)--(4);
            \draw[ArcBleu](4)--(5);
            \draw[ArcBleu](6)--(7);
        \end{tikzpicture}}
    } \qquad \qquad
    \subfigure[Returned.]{
        \scalebox{.7}{\begin{tikzpicture}
            \node[shape=coordinate](1)at(-0.43,-0.90){};
            \node[shape=coordinate](2)at(-0.97,-0.23){};
            \node[shape=coordinate](3)at(-0.78,0.62){};
            \node[shape=coordinate](4)at(0.00,1.0){};
            \node[shape=coordinate](5)at(0.78,0.62){};
            \node[shape=coordinate](6)at(0.98,-0.21){};
            \node[shape=coordinate](7)at(0.45,-0.89){};
            \draw[CotePolyg](1)--(7);
            \draw[CotePolyg](1)--(2);
            \draw[CotePolyg](2)--(3);
            \draw[CotePolyg](3)--(4);
            \draw[CotePolyg](4)--(5);
            \draw[CotePolyg](5)--(6);
            \draw[CotePolyg](6)--(7);
            \draw[ArcBleu](1)--(2);
            \draw[ArcBleu](2)--(3);
            \draw[ArcBleu](5)--(6);
        \end{tikzpicture}}
        \quad \raisebox{1.5em}{$\xrightarrow{\Ret}$} \quad
        \scalebox{.7}{\begin{tikzpicture}
            \node[shape=coordinate](1)at(-0.43,-0.90){};
            \node[shape=coordinate](2)at(-0.97,-0.23){};
            \node[shape=coordinate](3)at(-0.78,0.62){};
            \node[shape=coordinate](4)at(0.00,1.0){};
            \node[shape=coordinate](5)at(0.78,0.62){};
            \node[shape=coordinate](6)at(0.98,-0.21){};
            \node[shape=coordinate](7)at(0.45,-0.89){};
            \draw[CotePolyg](1)--(7);
            \draw[CotePolyg](1)--(2);
            \draw[CotePolyg](2)--(3);
            \draw[CotePolyg](3)--(4);
            \draw[CotePolyg](4)--(5);
            \draw[CotePolyg](5)--(6);
            \draw[CotePolyg](6)--(7);
            \draw[ArcBleu](2)--(3);
            \draw[ArcBleu](5)--(6);
            \draw[ArcBleu](6)--(7);
        \end{tikzpicture}}
    }
    \caption{The complementary and the returned of a bubble.}
    \label{fig:exemple_symetries}
\end{figure}
\medskip

\begin{Proposition} \label{prop:symetries_Bulle}
    The group of symmetries of $\Bulle$ is generated by $\Cpl$ and $\Ret$
    and satisfies the relations
    \begin{equation} \label{eq:relations_groupe_symetries}
        \Ret = \Ret^{-1}, \quad
        \Cpl = \Cpl^{-1}, \quad
        \Ret \Cpl = \Cpl \Ret.
    \end{equation}
\end{Proposition}
\begin{proof}
    Since, by Proposition \ref{prop:generateurs_Bulle}, $\Bulle$ is
    generated by $G_{\Bulle}$, any symmetry of $\Bulle$ is {\em a fortiori}
    a bijection on $G_{\Bulle}$. By computer exploration, let us consider
    the $8!$ bijections and keep only the ones that are still well-defined
    as coloured operad morphisms or coloured operad antimorphisms in arity three.
    \smallskip

    There are exactly two bijections that are operad morphisms up to arity
    three: the trivial one and the bijection $\alpha$ sending any $x \in G_{\Bulle}$
    to $\Cpl(x)$. By induction on the arity, it follows that there is a
    unique coloured operad morphism coinciding with $\alpha$ in arity two
    and it is $\Cpl$. Then, since $\Cpl$ is a bijection, $\Cpl$ is an
    automorphism of $\Bulle$.
    \smallskip

    There are exactly two bijections that are operad antimorphisms up to
    arity three: the bijection $\beta$ sending any $x \in G_{\Bulle}$ to
    $\Ret(x)$ and the bijection $\gamma$ sending any $x \in G_{\Bulle}$ to
    $\Ret(\Cpl(x))$. Again by induction on the arity, it follows that there
    is a unique coloured operad antimorphism coinciding with $\beta$
    (resp. $\gamma$) in arity two and it is $\Ret$ (resp. $\Ret \Cpl$).
    Then, since $\Ret$ and $\Ret \Cpl$ are bijections, $\Ret$ and $\Ret \Cpl$
    are antiautomorphisms of $\Bulle$.
    \smallskip

    We have shown that the identity, $\Cpl$, $\Ret$ and $\Ret \Cpl$ are
    the only elements of the group of symmetries of $\Bulle$.
    Relations \eqref{eq:relations_groupe_symetries} between these are obvious.
\end{proof}
\medskip

\subsubsection{Presentation by generators and relations}

\begin{Theoreme} \label{thm:presentation_Bulle}
    The $2$-coloured operad $\Bulle$ admits the presentation
    $(G_{\Bulle}, \leftrightarrow)$ where $\leftrightarrow$ is the
    equivalence relation satisfying
    \begin{equation}
        \Corolle(\BBB) \circ_2 \Corolle(\BAB)
            \enspace \leftrightarrow \enspace \Corolle(\BBB) \circ_1 \Corolle(\BAB)
            \enspace \leftrightarrow \enspace \Corolle(\ABB) \circ_1 \Corolle(\BBB)
            \enspace \leftrightarrow \enspace \Corolle(\BBA) \circ_2 \Corolle(\BBB),
    \end{equation}
    \begin{equation}
        \Corolle(\BBB) \circ_2 \Corolle(\AAB)
            \enspace \leftrightarrow \enspace \Corolle(\BBB) \circ_1 \Corolle(\BAA)
            \enspace \leftrightarrow \enspace \Corolle(\ABB) \circ_1 \Corolle(\BBA)
            \enspace \leftrightarrow \enspace \Corolle(\BBA) \circ_2 \Corolle(\ABB),
    \end{equation}
    \begin{equation}
        \Corolle(\BBB) \circ_2 \Corolle(\AAA)
            \enspace \leftrightarrow \enspace \Corolle(\ABA) \circ_1 \Corolle(\BBA)
            \enspace \leftrightarrow \enspace \Corolle(\BBA) \circ_2 \Corolle(\ABA)
            \enspace \leftrightarrow \enspace \Corolle(\BBA) \circ_1 \Corolle(\BAA),
    \end{equation}
    \begin{equation}
        \Corolle(\BBB) \circ_2 \Corolle(\BAA)
            \enspace \leftrightarrow \enspace \Corolle(\ABA) \circ_1 \Corolle(\BBB)
            \enspace \leftrightarrow \enspace \Corolle(\BBA) \circ_2 \Corolle(\BBA)
            \enspace \leftrightarrow \enspace \Corolle(\BBA) \circ_1 \Corolle(\BAB),
    \end{equation}
    \begin{equation}
        \Corolle(\BBB) \circ_1 \Corolle(\AAB)
            \enspace \leftrightarrow \enspace \Corolle(\ABB) \circ_2 \Corolle(\BAB)
            \enspace \leftrightarrow \enspace \Corolle(\ABB) \circ_1 \Corolle(\ABB)
            \enspace \leftrightarrow \enspace \Corolle(\ABA) \circ_2 \Corolle(\BBB),
    \end{equation}
    \begin{equation}
        \Corolle(\BBB) \circ_1 \Corolle(\AAA)
            \enspace \leftrightarrow \enspace \Corolle(\ABB) \circ_2 \Corolle(\AAB)
            \enspace \leftrightarrow \enspace \Corolle(\ABB) \circ_1 \Corolle(\ABA)
            \enspace \leftrightarrow \enspace \Corolle(\ABA) \circ_2 \Corolle(\ABB),
    \end{equation}
    \begin{equation}
        \Corolle(\ABB) \circ_2 \Corolle(\AAA)
            \enspace \leftrightarrow \enspace \Corolle(\ABA) \circ_2 \Corolle(\ABA)
            \enspace \leftrightarrow \enspace \Corolle(\ABA) \circ_1 \Corolle(\ABA)
            \enspace \leftrightarrow \enspace \Corolle(\BBA) \circ_1 \Corolle(\AAA),
    \end{equation}
    \begin{equation}
        \Corolle(\ABB) \circ_2 \Corolle(\BAA)
            \enspace \leftrightarrow \enspace \Corolle(\ABA) \circ_2 \Corolle(\BBA)
            \enspace \leftrightarrow \enspace \Corolle(\ABA) \circ_1 \Corolle(\ABB)
            \enspace \leftrightarrow \enspace \Corolle(\BBA) \circ_1 \Corolle(\AAB),
    \end{equation}
    \begin{equation}
        \Corolle(\BAB) \circ_2 \Corolle(\BAB)
            \enspace \leftrightarrow \enspace \Corolle(\BAB) \circ_1 \Corolle(\BAB)
            \enspace \leftrightarrow \enspace \Corolle(\AAB) \circ_1 \Corolle(\BBB)
            \enspace \leftrightarrow \enspace \Corolle(\BAA) \circ_2 \Corolle(\BBB),
    \end{equation}
    \begin{equation}
        \Corolle(\BAB) \circ_2 \Corolle(\AAB)
            \enspace \leftrightarrow \enspace \Corolle(\BAB) \circ_1 \Corolle(\BAA)
            \enspace \leftrightarrow \enspace \Corolle(\AAB) \circ_1 \Corolle(\BBA)
            \enspace \leftrightarrow \enspace \Corolle(\BAA) \circ_2 \Corolle(\ABB),
    \end{equation}
    \begin{equation}
        \Corolle(\BAB) \circ_2 \Corolle(\AAA)
            \enspace \leftrightarrow \enspace \Corolle(\AAA) \circ_1 \Corolle(\BBA)
            \enspace \leftrightarrow \enspace \Corolle(\BAA) \circ_2 \Corolle(\ABA)
            \enspace \leftrightarrow \enspace \Corolle(\BAA) \circ_1 \Corolle(\BAA),
    \end{equation}
    \begin{equation}
        \Corolle(\BAB) \circ_2 \Corolle(\BAA)
            \enspace \leftrightarrow \enspace \Corolle(\AAA) \circ_1 \Corolle(\BBB)
            \enspace \leftrightarrow \enspace \Corolle(\BAA) \circ_2 \Corolle(\BBA)
            \enspace \leftrightarrow \enspace \Corolle(\BAA) \circ_1 \Corolle(\BAB),
    \end{equation}
    \begin{equation}
        \Corolle(\BAB) \circ_1 \Corolle(\AAB)
            \enspace \leftrightarrow \enspace \Corolle(\AAB) \circ_2 \Corolle(\BAB)
            \enspace \leftrightarrow \enspace \Corolle(\AAB) \circ_1 \Corolle(\ABB)
            \enspace \leftrightarrow \enspace \Corolle(\AAA) \circ_2 \Corolle(\BBB),
    \end{equation}
    \begin{equation}
        \Corolle(\BAB) \circ_1 \Corolle(\AAA)
            \enspace \leftrightarrow \enspace \Corolle(\AAB) \circ_2 \Corolle(\AAB)
            \enspace \leftrightarrow \enspace \Corolle(\AAB) \circ_1 \Corolle(\ABA)
            \enspace \leftrightarrow \enspace \Corolle(\AAA) \circ_2 \Corolle(\ABB),
    \end{equation}
    \begin{equation}
        \Corolle(\AAB) \circ_2 \Corolle(\AAA)
            \enspace \leftrightarrow \enspace \Corolle(\AAA) \circ_2 \Corolle(\ABA)
            \enspace \leftrightarrow \enspace \Corolle(\AAA) \circ_1 \Corolle(\ABA)
            \enspace \leftrightarrow \enspace \Corolle(\BAA) \circ_1 \Corolle(\AAA),
    \end{equation}
    \begin{equation}
        \Corolle(\AAB) \circ_2 \Corolle(\BAA)
            \enspace \leftrightarrow \enspace \Corolle(\AAA) \circ_2 \Corolle(\BBA)
            \enspace \leftrightarrow \enspace \Corolle(\AAA) \circ_1 \Corolle(\ABB)
            \enspace \leftrightarrow \enspace \Corolle(\BAA) \circ_1 \Corolle(\AAB).
    \end{equation}
\end{Theoreme}
\begin{proof}
    To prove the presentation of the statement, we shall show that there
    exists an operad isomorphism $\phi : \OpLibre(G_{\Bulle})/_\equiv \to \Bulle$.
    \smallskip

    Let us set $\phi\left([\Corolle(g)]_\equiv\right) := g$ for any $g$
    of $G_{\Bulle}$. We observe that for any relation
    $\Corolle(x) \circ_i \Corolle(y) \leftrightarrow \Corolle(z) \circ_j \Corolle(t)$
    of the statement, we have $x \circ_i y = z \circ_j t$. It then follows
    that $\phi$ can be uniquely extended into a coloured operad morphism.
    Moreover, since the image of $\phi$ contains all the generators of
    $\Bulle$, $\phi$ is surjective.
    \smallskip

    Let us now prove that $\phi$ is a bijection. For that, let us orient
    the relation $\leftrightarrow$ by means of the rewrite rule $\mapsto$
    on the coloured syntax trees on $G_{\Bulle}$ satisfying $S \mapsto T$
    if $S \leftrightarrow T$ and $T$ is one of the following sixteen target trees
    \begin{equation} \label{eq:cibles_reecriture_Bulle}
        \begin{split}
            \Corolle(\BBB) \circ_1 \Corolle(\BAB), \enspace
            \Corolle(\BBB) \circ_1 \Corolle(\BAA), \enspace
            \Corolle(\BBA) \circ_1 \Corolle(\BAA), \enspace
            \Corolle(\BBA) \circ_1 \Corolle(\BAB), \\
            \Corolle(\ABB) \circ_1 \Corolle(\ABB), \enspace
            \Corolle(\ABB) \circ_1 \Corolle(\ABA), \enspace
            \Corolle(\ABA) \circ_1 \Corolle(\ABA), \enspace
            \Corolle(\ABA) \circ_1 \Corolle(\ABB), \\
            \Corolle(\BAB) \circ_1 \Corolle(\BAB), \enspace
            \Corolle(\BAB) \circ_1 \Corolle(\BAA), \enspace
            \Corolle(\BAA) \circ_1 \Corolle(\BAA), \enspace
            \Corolle(\BAA) \circ_1 \Corolle(\BAB), \\
            \Corolle(\AAB) \circ_1 \Corolle(\ABB), \enspace
            \Corolle(\AAB) \circ_1 \Corolle(\ABA), \enspace
            \Corolle(\AAA) \circ_1 \Corolle(\ABA), \enspace
            \Corolle(\AAA) \circ_1 \Corolle(\ABB).
        \end{split}
    \end{equation}
    The target trees of $\mapsto$ are the only left comb trees appearing
    in each $\leftrightarrow$-equivalence class of the statement such
    that the colour of the first input of the root is the same
    as the colour of the first input of its child.
    \smallskip

    Let us prove that $\mapsto$ is terminating. Let $\psi$ be the map
    associating the pair $(a(T), b(T))$ with a coloured syntax tree $T$,
    where $a(T)$ is the sum, for each internal node $x$ of $T$, of the
    number of internal nodes in the tree rooted at the right child
    of $x$, and $b(t)$ is the number of internal nodes $x$ of $T$ having
    an internal node $y$ as left child such that $\In_1(x) \ne \In_1(y)$.
    We observe that, for any trees $T_0$ and $T_1$ such that
    $T_0 \mapsto T_1$, $\psi(T_1)$ is lexicographically smaller that
    $\psi(T_0)$. Hence, $\mapsto$ is terminating.
    \smallskip

    The normal forms of $\mapsto$ are the coloured syntax trees on
    $G_{\Bulle}$ that have no subtrees $S$ where the $S$ are the trees
    appearing as a left members of $\mapsto$. These are left comb trees
    $T$ such that for all internal nodes $x$ and $y$ of $T$, $\In_1(x) = \In_1(y)$.
    Pictorially, $T$ is of the form
    \begin{equation}
        \begin{split} T = \enspace \end{split}
        \begin{split}
            \scalebox{.36}{\begin{tikzpicture}[scale=1.1]
                \node(v)at(0,1.5){};
                \node[Noeud,Clair,minimum size=46pt](0)at(0,0){$x_{n-1}$};
                \node[Noeud,Clair,minimum size=46pt](1)at(-2,-2){$x_1$};
                \node[Feuille](2)at(-3.5,-3.5){};
                \node[Feuille](3)at(-.5,-3.5){};
                \node[Feuille](4)at(1.5,-1.5){};
                \draw[Arete](v)edge node[EtiqCoul]{$c$}(0);
                \draw[Arete,dashed](0)edge node[EtiqCoul]{$d_1$}(1);
                \draw[Arete](1)edge node[EtiqCoul]{$d_1$}(2);
                \draw[Arete](1)edge node[EtiqCoul,right]{$d_2$}(3);
                \draw[Arete](0)edge node[EtiqCoul,right]{$d_n$}(4);
            \end{tikzpicture}}
        \end{split}\,,
    \end{equation}
    where $c \in [2]$, $d_i \in [2]$ for all $i \in [n]$, and $x_j \in G_{\Bulle}$
    for all $j \in [n - 1]$. Since $T$ is a coloured syntax tree, given $c$
    and the $d_i$, there is exactly one possibility for all the $x_j$.
    Therefore, there are $f_c(n) := 2^n$ normal forms of $\mapsto$ of arity
    $n$ with $c$ as output colour. This imply that $\OpLibre(G_{\Bulle})/_\equiv$
    contains at most $f_c(n)$ elements of arity $n$ and $c$ as output colour.
    Then, since $f_c(n)$ is also the number of elements of $\Bulle$ with
    arity $n$ and $c$ as output colour (see Section \ref{sec:serie_bulles_Bulle}),
    $\phi$ is a bijection.
\end{proof}
\medskip

\subsection{Properties of the operad of bicoloured noncrossing configurations}
Let us come back on the study of the operad $\CNCB$. We show here that
$\CNCB$ is the enveloping operad of $\Bulle$ and then, by using the results
of Section \ref{subsec:operade_Bulle} together with the ones of
Section \ref{subsec:consequences_decomposition_bulles}, give some of its
properties.
\medskip

\subsubsection{Bubble decomposition}
Let $\Cfr$ be a BNC. An {\em area} of $\Cfr$ is a maximal component of $\Cfr$
without coloured diagonals and bounded by coloured arcs or by uncoloured
edges. Any area $a$ of $\Cfr$ defines a bubble $\Bfr$ consisting
in the edges of $a$. The base of $\Bfr$ is the only edge of $a$ that splits
$\Cfr$ in two parts where one contains the base of $\Cfr$ and the other
contains $a$. Blue edges of $a$ remain blue edges in $\Bfr$ and red edges
of $a$ become uncoloured edges in $\Bfr$.
\medskip

The {\em dual tree} of $\Cfr$ is the planar rooted tree labeled by bubbles
defined as follows. If $\Cfr$ is of size $1$, its dual tree is the leaf.
Otherwise, put an internal node in each area of $\Cfr$ and connect any pair
of nodes that are in adjacent areas. Put also leaves outside $\Cfr$, one
for each edge, except the base, and connect these with the internal nodes
of their adjacent areas. This forms a tree rooted at the node of the area
containing the base of $\Cfr$. Finally, label each internal node of the
tree by the bubble associated with the area containing it.
Figure \ref{fig:exemple_arbre_dual}
shows an example of a BNC and its dual tree.
\begin{figure}[ht]
    \centering
    \tikzstyle{AB}=[ArcBleu,line width=8pt]
    \tikzstyle{NN}=[circle,draw=Vert!100,fill=Vert!85,
        inner sep=0cm,minimum size=.9mm]
    \tikzstyle{F}=[Feuille,minimum size=.6mm,line width=.3pt]
    \tikzstyle{AR}=[Violet!70,cap=round,line width=.75pt]
    \tikzstyle{Noeud}=[circle,draw=Vert!100,fill=Blanc!100,
        inner sep=2mm,minimum size=14mm,line width=8pt]
    \tikzstyle{Arete}=[Violet!70,cap=round,line width=8pt]
    \tikzstyle{ArcBleu}=[Bleu!80,thick,draw,cap=round,line width=1.25pt]
    \tikzstyle{ArcRouge}=[Rouge!80,thick,draw,cap=round,line width=1pt,dotted]
    \subfigure[$\Cfr$]{
        \scalebox{1.8}{\begin{tikzpicture}
            \node[shape=coordinate](0)at(-0.26,-0.97){};
            \node[shape=coordinate](1)at(-0.71,-0.71){};
            \node[shape=coordinate](2)at(-0.97,-0.26){};
            \node[shape=coordinate](3)at(-0.97,0.26){};
            \node[shape=coordinate](4)at(-0.71,0.71){};
            \node[shape=coordinate](5)at(-0.26,0.97){};
            \node[shape=coordinate](6)at(0.26,0.97){};
            \node[shape=coordinate](7)at(0.71,0.71){};
            \node[shape=coordinate](8)at(0.97,0.26){};
            \node[shape=coordinate](9)at(0.97,-0.26){};
            \node[shape=coordinate](10)at(0.71,-0.71){};
            \node[shape=coordinate](11)at(0.26,-0.97){};
            \draw[CotePolyg](0)--(1);
            \draw[CotePolyg](1)--(2);
            \draw[CotePolyg](2)--(3);
            \draw[CotePolyg](3)--(4);
            \draw[CotePolyg](4)--(5);
            \draw[CotePolyg](5)--(6);
            \draw[CotePolyg](6)--(7);
            \draw[CotePolyg](7)--(8);
            \draw[CotePolyg](8)--(9);
            \draw[CotePolyg](9)--(10);
            \draw[CotePolyg](10)--(11);
            \draw[CotePolyg](11)--(0);
            \draw[ArcBleu](0)--(1);
            \draw[ArcBleu](0)--(11);
            \draw[ArcBleu](1)--(2);
            \draw[ArcBleu](2)--(3);
            \draw[ArcBleu](2)--(5);
            \draw[ArcBleu](2)--(11);
            \draw[ArcBleu](3)--(4);
            \draw[ArcBleu](4)--(5);
            \draw[ArcBleu](6)--(7);
            \draw[ArcBleu](6)--(11);
            \draw[ArcBleu](7)--(8);
            \draw[ArcBleu](8)--(9);
            \draw[ArcBleu](9)--(10);
            \draw[ArcBleu](10)--(11);
            \draw[ArcRouge](3)--(5);
            \draw[ArcRouge](6)--(10);
            \node(r)at(.5,-1.3){};
            \node[NN](A)at(-.45,-.7){};
            \draw[AR](A)--(r);
            \node[F](f1)at(-.6,-1.2){};
            \draw[AR](A)--(f1);
            \node[F](f2)at(-1.1,-.5){};
            \draw[AR](A)--(f2);
            \node[NN](B)at(-.25,0){};
            \draw[AR](A)--(B);
            \node[NN](C)at(-.75,.32){};
            \draw[AR](B)--(C);
            \node[F](f3)at(.1,1.2){};
            \draw[AR](B)--(f3);
            \node[NN](D)at(.45,-.25){};
            \draw[AR](B)--(D);
            \node[F](f4)at(-1.15,0){};
            \draw[AR](C)--(f4);
            \node[NN](E)at(-.65,.65){};
            \draw[AR](C)--(E);
            \node[NN](F)at(.7,.15){};
            \draw[AR](D)--(F);
            \node[F](f5)at(.5,-1.1){};
            \draw[AR](D)--(f5);
            \node[F](f6)at(-1.1,.45){};
            \draw[AR](E)--(f6);
            \node[F](f7)at(-.3,1.1){};
            \draw[AR](E)--(f7);
            \node[F](f8)at(.4,1.1){};
            \draw[AR](F)--(f8);
            \node[F](f9)at(1,.7){};
            \draw[AR](F)--(f9);
            \node[F](f10)at(1.1,0){};
            \draw[AR](F)--(f10);
            \node[F](f11)at(1,-.7){};
            \draw[AR](F)--(f11);
        \end{tikzpicture}}
    }
    \qquad \qquad
    \subfigure[The dual tree of $\Cfr$.]{
        \scalebox{0.16}{\begin{tikzpicture}
            \node[Noeud](0)at(0.00,-8.50){
                \begin{tikzpicture}
                \node[Feuille]at(0,0){};
                \end{tikzpicture}};
            \node[Noeud](1)at(2.50,0.00){
                \begin{tikzpicture}[scale=2.5]
                \node[shape=coordinate](A0)at(-0.71,-0.71){};
                \node[shape=coordinate](A1)at(-0.71,0.71){};
                \node[shape=coordinate](A2)at(0.71,0.71){};
                \node[shape=coordinate](A3)at(0.71,-0.71){};
                \draw[CotePolyg](A0)--(A1);
                \draw[CotePolyg](A1)--(A2);
                \draw[CotePolyg](A2)--(A3);
                \draw[CotePolyg](A3)--(A0);
                \draw[AB](A0)--(A1);
                \draw[AB](A0)--(A3);
                \draw[AB](A1)--(A2);
                \draw[AB](A2)--(A3);
                \end{tikzpicture}};
            \node[Noeud](10)at(20.00,-34.00){
                \begin{tikzpicture}
                \node[Feuille]at(0,0){};
                \end{tikzpicture}};
            \node[Noeud](11)at(22.50,-34.00){
                \begin{tikzpicture}
                \node[Feuille]at(0,0){};
                \end{tikzpicture}};
            \node[Noeud](12)at(25.00,-25.50){
                \begin{tikzpicture}[scale=3]
                \node[shape=coordinate](F0)at(-0.59,-0.81){};
                \node[shape=coordinate](F1)at(-0.95,0.31){};
                \node[shape=coordinate](F2)at(-0.00,1.00){};
                \node[shape=coordinate](F3)at(0.95,0.31){};
                \node[shape=coordinate](F4)at(0.59,-0.81){};
                \draw[CotePolyg](F0)--(F1);
                \draw[CotePolyg](F1)--(F2);
                \draw[CotePolyg](F2)--(F3);
                \draw[CotePolyg](F3)--(F4);
                \draw[CotePolyg](F4)--(F0);
                \draw[AB](F0)--(F1);
                \draw[AB](F1)--(F2);
                \draw[AB](F2)--(F3);
                \draw[AB](F3)--(F4);
                \end{tikzpicture}};
            \node[Noeud](13)at(27.50,-34.00){
                \begin{tikzpicture}
                \node[Feuille]at(0,0){};
                \end{tikzpicture}};
            \node[Noeud](14)at(30.00,-34.00){
                \begin{tikzpicture}
                \node[Feuille]at(0,0){};
                \end{tikzpicture}};
            \node[Noeud](15)at(32.50,-17.00){
                \begin{tikzpicture}[scale=2]
                \node[shape=coordinate](D0)at(-0.87,-0.50){};
                \node[shape=coordinate](D1)at(-0.00,1.00){};
                \node[shape=coordinate](D2)at(0.87,-0.50){};
                \draw[CotePolyg](D0)--(D1);
                \draw[CotePolyg](D1)--(D2);
                \draw[CotePolyg](D2)--(D0);
                \draw[AB](D0)--(D2);
                \draw[AB](D1)--(D2);
                \end{tikzpicture}};
            \node[Noeud](16)at(35.00,-25.50){
                \begin{tikzpicture}
                \node[Feuille]at(0,0){};
                \end{tikzpicture}};
            \node[Noeud](2)at(2.50,-8.50){
                \begin{tikzpicture}
                \node[Feuille]at(0,0){};
                \end{tikzpicture}};
            \node[Noeud](3)at(5.00,-25.50){
                \begin{tikzpicture}
                \node[Feuille]at(0,0){};
                \end{tikzpicture}};
            \node[Noeud](4)at(7.50,-17.00){
                \begin{tikzpicture}[scale=2]
                \node[shape=coordinate](C0)at(-0.87,-0.50){};
                \node[shape=coordinate](C1)at(-0.00,1.00){};
                \node[shape=coordinate](C2)at(0.87,-0.50){};
                \draw[CotePolyg](C0)--(C1);
                \draw[CotePolyg](C1)--(C2);
                \draw[CotePolyg](C2)--(C0);
                \draw[AB](C0)--(C2);
                \draw[AB](C0)--(C1);
                \end{tikzpicture}};
            \node[Noeud](5)at(10.00,-34.00){
                \begin{tikzpicture}
                \node[Feuille]at(0,0){};
                \end{tikzpicture}};
            \node[Noeud](6)at(12.50,-25.50){
                \begin{tikzpicture}[scale=2]
                \node[shape=coordinate](E0)at(-0.87,-0.50){};
                \node[shape=coordinate](E1)at(-0.00,1.00){};
                \node[shape=coordinate](E2)at(0.87,-0.50){};
                \draw[CotePolyg](E0)--(E1);
                \draw[CotePolyg](E1)--(E2);
                \draw[CotePolyg](E2)--(E0);
                \draw[AB](E0)--(E1);
                \draw[AB](E1)--(E2);
                \end{tikzpicture}};
            \node[Noeud](7)at(15.00,-34.00){
                \begin{tikzpicture}
                \node[Feuille]at(0,0){};
                \end{tikzpicture}};
            \node[Noeud](8)at(17.50,-8.50){
                \begin{tikzpicture}[scale=2.5]
                \node[shape=coordinate](B0)at(-0.71,-0.71){};
                \node[shape=coordinate](B1)at(-0.71,0.71){};
                \node[shape=coordinate](B2)at(0.71,0.71){};
                \node[shape=coordinate](B3)at(0.71,-0.71){};
                \draw[CotePolyg](B0)--(B1);
                \draw[CotePolyg](B1)--(B2);
                \draw[CotePolyg](B2)--(B3);
                \draw[CotePolyg](B3)--(B0);
                \draw[AB](B0)--(B3);
                \draw[AB](B0)--(B1);
                \draw[AB](B2)--(B3);
                \end{tikzpicture}};
            \node[Noeud](9)at(17.50,-17.00){
                \begin{tikzpicture}
                \node[Feuille]at(0,0){};
                \end{tikzpicture}};
            \draw[Arete](0)--(1);
            \draw[Arete](10)--(12);
            \draw[Arete](11)--(12);
            \draw[Arete](12)--(15);
            \draw[Arete](13)--(12);
            \draw[Arete](14)--(12);
            \draw[Arete](15)--(8);
            \draw[Arete](16)--(15);
            \draw[Arete](2)--(1);
            \draw[Arete](3)--(4);
            \draw[Arete](4)--(8);
            \draw[Arete](5)--(6);
            \draw[Arete](6)--(4);
            \draw[Arete](7)--(6);
            \draw[Arete](8)--(1);
            \draw[Arete](9)--(8);
            \draw[Arete](1)--(2.5,4.5);
        \end{tikzpicture}}
    }
    \caption{A bicoloured noncrossing configuration and its dual tree.}
    \label{fig:exemple_arbre_dual}
\end{figure}
\medskip

\begin{Lemme} \label{lem:arbre_dual_anticolore}
    Let $\Cfr$ be a BNC. The dual tree of $\Cfr$ is an anticoloured
    syntax tree on $\Bulle^+$.
\end{Lemme}
\begin{proof}
    This follows from the definition of dual trees and the fact that a
    blue (resp. uncoloured) edge of a bubble $\Bfr$ is, by definition,
    of colour $1$ (resp. $2$) if it is the base of $\Bfr$ and of colour
    $2$ (resp. $1$) otherwise.
\end{proof}
\medskip

\begin{Lemme} \label{lem:bijection_cncb_arbres_duaux}
    The map between the BNCs of arity $n$ and the anticoloured syntax
    trees on $\Bulle^+$ or arity $n$ sending a BNC to its dual tree
    is a bijection.
\end{Lemme}
\begin{proof}
    By Lemma \ref{lem:arbre_dual_anticolore}, this map is well-defined.
    Let $T$ be an anticoloured syntax tree on $\Bulle^+$. By seeing $T$
    as a $1$-coloured syntax tree, one can perform reductions in $T$ up
    to obtain a corolla labeled by a BNC $x$. The fact that $\CNCB$ is an
    operad ensures that the reductions can be made in any order. Thanks
    to the definition of the composition of $\CNCB$ together with the
    definition of dual trees, the application sending $T$ to $x$ is the
    inverse of the map of the statement.
\end{proof}
\medskip

\begin{Theoreme} \label{thm:decomposition_bulles_CNCB}
    The $2$-coloured operad $\Bulle$ is a $2$-bubble decomposition of
    the operad $\CNCB$.
\end{Theoreme}
\begin{proof}
    This is a consequence of Lemmas \ref{lem:arbre_dual_anticolore}
    and \ref{lem:bijection_cncb_arbres_duaux}: the elements of $\CNCB$
    are anticoloured syntax trees on $\Bulle^+$ and the composition of
    $\CNCB$ translates faithfully into the composition of $\OpEnv(\Bulle)$.
\end{proof}
\medskip

\subsubsection{Enumeration of the bicoloured noncrossing configurations}
By using the fact that, by Theorem \ref{thm:decomposition_bulles_CNCB},
$\Bulle$ is a $2$-bubble decomposition of $\CNCB$, together with
Proposition \ref{prop:serie_hilbert} and the coloured Hilbert
series \eqref{eq:serie_bulles_Bulle} of $\Bulle$, we obtain the following
algebraic equation for the generating series of the BNCs.

\begin{Proposition} \label{prop:serie_CNBC}
    The Hilbert series $\SerieOp$ of $\CNCB$ satisfies
    \begin{equation}
        -t - t^2 + (1 - 4t) \SerieOp - 3 \SerieOp^2 = 0.
    \end{equation}
\end{Proposition}
\medskip

First numbers of BNCs by size are
\begin{equation}
    1, \: 8, \: 80, \: 992, \: 13760, \: 204416, \: 3180800, \:
    51176960, \: 844467200.
\end{equation}
\medskip

One can refine the above enumeration of BNCs in the following way. Let
us add two variables $y_1$ and $y_2$ in the system of the statement of
Proposition \ref{prop:serie_hilbert} for $\CNCB$ to obtain
\begin{equation}
    \begin{split}
        \SerieOp & = t + \SerieOp_1 + \SerieOp_2 \\[.5em]
        \SerieOp_1 & =
            y_1 \frac{(2\SerieOp - \SerieOp_1 - \SerieOp_2)^2}
            {1 - 2\SerieOp + \SerieOp_1 + \SerieOp_2} \\[.75em]
        \SerieOp_2 & =
            y_2 \frac{(2\SerieOp - \SerieOp_1 - \SerieOp_2)^2}
            {1 - 2\SerieOp + \SerieOp_1 + \SerieOp_2}.
    \end{split}
\end{equation}
By solving this system, we find that the generating series
$\SerieOp(t, y_1, y_2)$ satisfies
\begin{equation} \label{eq:equation_raffinement_cncb}
    -t + t^2 - y_1t^2 - y_2t^2 + (1 - 2y_1t - 2y_2t)\SerieOp
    + (-1 - y_1 - y_2)\SerieOp^2 = 0.
\end{equation}
The parameter $y_1$ (resp. $y_2$) counts the internal nodes of anticoloured
trees on $\Bulle^+$ that are labeled by based (resp. nonbased) bubbles.
By a direct translation on the BNCs themselves, $y_1$ counts
the blue diagonals (where a blue base counts as a blue diagonal)
and $y_2$ counts the red diagonals (where an uncoloured base counts as
a red diagonal). We obtain
\begin{multline}
    \SerieOp(t, y_1, y_2) =
    t +
    4(y_1 + y_2)t^2 +
    8(y_1 + 2y_1^2 + 4y_1y_2 + 2y_2^2 + y_2)t^3 \\ +
    16(y_1 + 5y_1^2 + 5y_1^3 + 15y_1^2y_2 + 10y_1y_2 + 15y_1y_2^2 +
       5y_2^3 + 5y_2^2 + y_2)t^4 + \cdots.
\end{multline}
Since, by definition of the dual trees, there is a correspondence
between the areas of a BNC and the internal nodes of its dual tree,
the specialization $\SerieOp(t, y) := \SerieOp(t, y, y)$ satisfying
\begin{equation}
    -t + (1 - 2y) t^2 + (1 - 4yt) \SerieOp + (-1 - 2y) \SerieOp^2 = 0
\end{equation}
counts the BNCs by their areas. We have
\begin{multline}
        \SerieOp(t, y) =
        t + 8y t^2 +
        16(y + 4y^2)t^3 +
        32(y + 10y^2 + 20y^3)t^4 +
        64(y + 18y^2 + 84y^3 + 112y^4)t^5 \\ +
        128(y + 28y^2 + 224y^3 + 672y^4 + 672y^5)t^6 \\ +
        256(y + 40y^2 + 480y^3 + 2400y^4 + 5280y^5 + 4224y^6)t^7 +
        \cdots.
\end{multline}
\medskip

\subsubsection{Others consequences}
Since $\Bulle$ is, by Theorem \ref{thm:decomposition_bulles_CNCB},
a $2$-bubble decomposition of $\CNCB$, we can use the results of
Section \ref{subsec:consequences_decomposition_bulles} to obtain the
generating set, the group of symmetries, and the presentation by
generators and relations of $\CNCB$.
\medskip

Thus, by Propositions \ref{prop:famille_generatrice}
and \ref{prop:generateurs_Bulle}, the generating set of $\CNCB$ is the set
of the eight BNCs of arity $2$.
\medskip

By Propositions \ref{prop:symetries} and \ref{prop:symetries_Bulle},
the group of symmetries of $\CNCB$ is generated by the maps
$\Cpl' := \OpEnv(\Cpl)$ and $\Ret' := \OpEnv(\Ret)$. For any BNC $\Cfr$,
$\Cpl'(\Cfr)$ is the BNC obtained by swapping the colours of the red and
blue diagonals of $\Cfr$, and by swapping the colours of the edges of $\Cfr$.
Moreover, for any BNC $\Cfr$, $\Ret'(\Cfr)$ is the BNC obtained by applying
on $\Cfr$ the reflection through the vertical line passing by its base.
\medskip

Finally, by Proposition \ref{prop:presentation} and
Theorem \ref{thm:presentation_Bulle}, $\CNCB$ admits the presentation
by generators and relations of the statement of Theorem \ref{thm:presentation_Bulle}.
\medskip

\section{Suboperads of the operad of bicoloured noncrossing configurations}
\label{sec:sous_operades_CNCB}
We now study some of the suboperads of $\CNCB$ generated by various sets
of BNCs. We shall mainly focus on the suboperads generated by sets of two
BNCs of arity $2$.
\medskip

\subsection{Overview of the obtained suboperads}
In what follows, we denote by $\Op{G}$ the suboperad of $\CNCB$ generated
by a set $G$ of BNCs and, when $G$ is a set of bubbles, by $\OpC{G}$ the
coloured suboperad of $\Bulle$ generated by $G$.
\medskip

\subsubsection{Orbits of suboperads}
There are $2^8 = 256$ suboperads of $\CNCB$ generated by elements of arity
$2$. The symmetries provided by the group of symmetries of $\CNCB$
(see Proposition \ref{prop:symetries_Bulle}) allow
to gather some of these together. Indeed, if $G_1$ and $G_2$ are two sets
of BNCs and $\phi$ is a map of the group of symmetries of $\CNCB$ such
that $\phi(G_1) = G_2$, the suboperads $\Op{G_1}$  and $\Op{G_2}$ would be
isomorphic or antiisomorphic. We say in this case that these two operads
are {\em equivalent}. There are in this way only $88$ orbits of suboperads
that are pairwise nonequivalent.
\medskip

\subsubsection{Suboperads on one generator}
There are three orbits of suboperads of $\CNCB$ generated by one generator
of arity $2$. The first contains $\Op{\AAA}$. By induction on the arity,
one can show that this operad contains all the triangulations and that it
is free. The second one contains $\Op{\AAB}$. By using similar arguments,
one can show that this operad is also free and isomorphic to the latter.
The third orbit contains $\Op{\BAB}$. This operad contains exactly one element
of any arity, and hence, is the associative operad.
\medskip

\subsubsection{Operads of noncrossing trees and plants}
The first named author defined in \cite{Cha07} an operad involving
noncrossing trees and an operad involving noncrossing plants. As follows
directly from the definition, these operads are the suboperads of $\CNCB$
$\Op{\AAB, \BAA}$ and $\Op{\AAB, \ABA ,\BAA}$ respectively. The operad of
noncrossing trees governs L-algebras, a sort of algebras introduced by
Leroux \cite{Ler11}.
\medskip

\subsubsection{Suboperads on two generators}
The $\binom{8}{2} = 28$ suboperads of $\CNCB$ generated by two BNCs of
arity $2$ form eleven orbits. Table \ref{tab:sous_operades_CNCB}
summarizes some information about these.
\begin{table}[ht]
    \centering
    \begin{tabular}{c|c|c}
        Operad & Dimensions & Presentation \\ \hline \hline
        $\Op{\AAA, \AAB}$ &
            1, 2, 8, 40, 224, 1344, 8448, 54912 & free \\ \hline
        $\Op{\AAA, \BBB}$ &
            1, 2, 8, 40, 216, 1246, 7516, 46838 & quartic or more \\ \hline
        $\Op{\AAA, \ABB}$ &
            \multirow{3}{*}{1, 2, 8, 38, 200, 1124, 6608, 40142} &
            \multirow{3}{*}{cubic} \\
        $\Op{\AAB, \ABB}$ & & \\
        $\Op{\AAB, \BBA}$ & & \\ \hline
        $\Op{\AAA, \BAB}$ &
            1, 2, 7, 31, 154, 820, 4575, 26398 & quadratic \\ \hline
        $\Op{\AAB, \BAA}$ &
            1, 2, 7, 30, 143, 728, 3876, 21318 & quadratic \\ \hline
        $\Op{\AAB, \BAB}$ &
            \multirow{4}{*}{1, 2, 6, 22, 90, 394, 1806, 8558} &
            \multirow{4}{*}{quadratic} \\
        $\Op{\AAA, \ABA}$ & & \\
        $\Op{\BAA, \ABA}$ & & \\
        $\Op{\BAB, \ABA}$ & &
    \end{tabular}
    \bigskip
    \caption{The eleven orbits of suboperads of $\CNCB$ generated by two
    generators of arity $2$, their dimensions and the degrees of
    nontrivial relations between their generators.}
    \label{tab:sous_operades_CNCB}
\end{table}
Some of these operads are well-known operads: the free operad $\Op{\AAA, \AAB}$
on two generators of arity $2$, the operad of noncrossing trees \cite{Cha07,Ler11}
$\Op{\AAB, \BAA}$, the dipterous operad \cite{LR03,Zin12} $\Op{\BAA, \ABA}$,
and the $2$-associative operad \cite{LR06,Zin12} $\Op{\BAB, \ABA}$.
All the Hilbert series of the eleven operads are algebraic, with the genus
of the associated algebraic curve being $0$.
For some of these series, the coefficients form known sequences of \cite{Slo10}.
The first one of Table \ref{tab:sous_operades_CNCB} is Sequence \Sloane{A052701},
the fourth is Sequence \Sloane{A007863}, the fifth is Sequence \Sloane{A006013},
and the sixth is Sequence \Sloane{A006318}.
\medskip

\subsubsection{Suboperads on more than two generators}
Some suboperads of $\CNCB$ generated by more than two generators are very
complicated to study. For instance, the operad $\langle \AAA, \BAB, \BBB\rangle$
has two equivalence classes of nontrivial relations in degree $2$,
three in degree $3$, ten in degree $4$ and seems to have no nontrivial
relations in higher degree (this has been checked until degree $6$).
The operad $\langle \AAA, \ABA, \BAB, \BBB\rangle$ is also complicated
since it has four equivalence classes of nontrivial relations in degree
$2$, sixteen in degree $3$ and seems to have no nontrivial relations in
higher degree (this has been checked until degree $6$).
\medskip

\subsection{Suboperads generated by two elements of arity 2}
For any of the eleven nonequivalent suboperads of $\CNCB$ generated by
two elements of arity $2$, we compute its dimensions and provide a
presentation by generators and relations by passing through a bubble
decomposition of it.
\medskip

\subsubsection{Outline of the study}
Let $\Op{G}$ be one of these suboperads. Since, by Theorem \ref{thm:decomposition_bulles_CNCB},
$\Bulle$ is a bubble decomposition of $\CNCB$ and $\Op{G}$ is generated
by bubbles, $\OpC{G}$ is a bubble decomposition of $\Op{G}$. We shall
compute the dimensions and establish the presentation by generators and
relations of $\OpC{G}$ to obtain in return, by Propositions \ref{prop:serie_hilbert}
and \ref{prop:presentation}, the dimensions and the presentation by generators
and relations of $\Op{G}$.
\medskip

To compute the dimensions of $\OpC{G}$, we shall furnish a description
of its elements and then deduce from the description its coloured Hilbert
series. In what follows, we will only detail this for the first orbit.
The computations for the other orbits are analogous.
Table \ref{tab:sous_operades_bulles_CNCB} shows the first coefficients of
the coloured Hilbert series of the eleven coloured suboperads. All of
these series are rational.
\begin{table}[ht]
    \centering
    \begin{tabular}{c|c|c}
        Coloured operad & Based bubbles & Nonbased bubbles \\ \hline \hline
        $\OpC{\AAA, \AAB}$ &
            2, 2, 2, 2, 2, 2, 2 &
            0, 0, 0, 0, 0, 0, 0 \\ \hline
        $\OpC{\AAA, \BBB}$ &
            1, 2, 5, 10, 21, 42, 85 & 1, 2, 5, 10, 21, 42, 85 \\ \hline
        $\OpC{\AAA, \ABB}$ &
            \multirow{3}{*}{1, 2, 4, 8, 16, 32, 64} &
            \multirow{3}{*}{1, 2, 4, 8, 16, 32, 64} \\
        $\OpC{\AAB, \ABB}$ & & \\
        $\OpC{\AAB, \BBA}$ & & \\ \hline
        $\OpC{\AAA, \BAB}$ &
            2, 3, 5, 8, 13, 21, 34 &
            0, 0, 0, 0, 0, 0, 0 \\ \hline
        $\OpC{\AAB, \BAA}$ &
            2, 3, 4, 5, 6, 7, 8 &
            0, 0, 0, 0, 0, 0, 0 \\ \hline
        $\OpC{\AAB, \BAB}$ &
            2, 4, 8, 16, 32, 64, 128 &
            0, 0, 0, 0, 0, 0, 0 \\ \hline
        $\OpC{\AAA, \ABA}$ &
            \multirow{3}{*}{1, 1, 1, 1, 1, 1, 1} &
            \multirow{3}{*}{1, 1, 1, 1, 1, 1, 1} \\
        $\OpC{\BAA, \ABA}$ & & \\
        $\OpC{\BAB, \ABA}$ & &
    \end{tabular}
    \bigskip
    \caption{The eleven orbits of $2$-coloured suboperads of $\Bulle$
    generated by two generators of arity $2$ and the number of their bubbles,
    based and nonbased.}
    \label{tab:sous_operades_bulles_CNCB}
\end{table}
\medskip

To establish the presentation of $\OpC{G}$, we shall use the same strategy
as the one used for the proof of the presentation of $\Bulle$ (see the
proof of Theorem \ref{thm:presentation_Bulle}). Recall that this consists
in exhibiting an orientation $\mapsto$ of the presentation we want to prove
such that $\mapsto$ is a terminating rewrite rule on coloured syntax trees
and its normal forms are in bijection with the elements of $\OpC{G}$.
We call these rewrite rules {\em good orientations}. In what follows, we
will exhibit a good orientation for any of the studied operads, except
the first and the third ones.
\medskip

\subsubsection{First orbit}
This orbit consists of the operads $\Op{\AAA, \AAB}$, $\Op{\AAA, \BAA}$,
$\Op{\BBB, \ABB}$, and $\Op{\BBB, \BBA}$. We choose $\Op{\AAA, \AAB}$
as a representative of the orbit.
\medskip

\begin{Proposition} \label{prop:bulles_aaa_aab}
    The set of bubbles of $\OpC{\AAA, \AAB}$ is the set of based bubbles
    such that all edges of the border except possibly the last one are
    blue. Moreover, the coloured Hilbert series of $\OpC{\AAA, \AAB}$
    satisfy
    \begin{equation}
        \SerieBulles_1(z_1, z_2) =
        \frac{z_1z_2 + z_2^2}{1 - z_2}
        \qquad \mbox{ and } \qquad
        \SerieBulles_2(z_1, z_2) = 0.
    \end{equation}
\end{Proposition}
\begin{proof}
    Let us prove by induction on the arity that any bubble of $\OpC{\AAA, \AAB}$
    satisfies the statement of the Proposition. The two bubbles $\AAA$
    and $\AAB$ of arity $2$ satisfy the statement. Let $\Bfr$ be a bubble
    of arity $n \geq 3$ satisfying the statement, $i \in [n]$, and
    $\Bfr' := \Bfr \circ_i g$ where $g \in \{\AAA, \AAB\}$. Then, since
    the composition $\Bfr \circ_i g$ is well-defined, we have $i = n$ and
    the last edge of $\Bfr$ is uncoloured. Thus, $\Bfr'$ is obtained from
    $\Bfr$ by replacing its last uncoloured edge by two blue edges of by
    a blue edge followed by an uncoloured edge, according to the colour
    of the last edge of $g$. This shows that any bubble of $\OpC{\AAA, \AAB}$
    satisfies the statement.
    \smallskip

    Conversely, let us prove by induction on the arity that any bubble
    satisfying the statement of the Proposition belongs to $\OpC{\AAA, \AAB}$.
    This is true for the bubbles $\AAA$ and $\AAB$ since they are generators
    of $\OpC{\AAA, \AAB}$. Let $\Bfr$ be a bubble of arity $n \geq 3$
    satisfying the statement and $\Bfr'$ be the bubble of arity $n - 1$
    obtained by replacing the two last edges of $\Bfr$ by an uncoloured
    edge. If the last edge of $\Bfr$ is blue, we have
    $\Bfr = \Bfr' \circ_{n - 1} \AAA$ and otherwise,
    $\Bfr = \Bfr' \circ_{n - 1} \AAB$. Since $\Bfr'$ satisfies the statement,
    it is by induction hypothesis, an element of $\OpC{\AAA, \AAB}$. Thus,
    $\Bfr$ also is.
    \smallskip

    Finally, the expressions for the coloured Hilbert series of
    $\OpC{\AAA, \AAB}$ follow directly from the above description of its
    elements.
\end{proof}
\medskip

\begin{Proposition} \label{prop:serie_Hilbert_aaa_aab}
    The Hilbert series $\SerieOp$ of $\Op{\AAA, \AAB}$ satisfies
    \begin{equation} \label{eq:serie_Hilbert_aaa_aab}
        t - \SerieOp + 2 \SerieOp^2 = 0.
    \end{equation}
\end{Proposition}
\medskip

\begin{Theoreme} \label{thm:presentation_aaa_aab}
    The operad $\Op{\AAA, \AAB}$ is the free operad generated by two
    generators of arity $2$.
\end{Theoreme}
\begin{proof}
    The elements of arity $n$ of the free operad $\Oca$ generated by two
    generators of arity $2$ are binary trees with $n$ leaves and such
    that internal nodes can be labeled in two different ways. Hence,
    the Hilbert series $\SerieOp$ of $\Oca$ satisfies
    \eqref{eq:serie_Hilbert_aaa_aab} and by Proposition \ref{prop:serie_Hilbert_aaa_aab},
    $\Op{\AAA, \AAB}$ and $\Oca$ have the same Hilbert series. Thus, since
    there is no nontrivial relation between the generators of
    $\Op{\AAA, \AAB}$ in degree $2$, there is no nontrivial relation
    in $\Op{\AAA, \AAB}$ of higher degree. Then, $\Op{\AAA, \AAB}$ and
    $\Oca$ are isomorphic.
\end{proof}
\medskip

\subsubsection{Second orbit}
This orbit consists of the operad $\Op{\AAA, \BBB}$.
\medskip

\begin{Proposition} \label{prop:bulles_aaa_bbb}
    The set of based (resp. nonbased) bubbles of $\OpC{\AAA, \BBB}$ of
    arity $n$ is the set of based (resp. nonbased) bubbles having
    at least two consecutive edges of the border of a same colour and the
    number of blue (resp. uncoloured) edges of the border is congruent to
    $1 - n$ modulo $3$.
    Moreover, the coloured Hilbert series of $\OpC{\AAA, \BBB}$ satisfy
    \begin{equation}
        \SerieBulles_1(z_1, z_2) =
        \frac{z_1 + z_2^2}{1 - 3z_1z_2 - z_1^3 - z_2^3} -
        \frac{z_2}{1 - z_1z_2}
    \end{equation}
    and
    \begin{equation}
        \SerieBulles_2(z_1, z_2) =
        \frac{z_2 + z_1^2}{1 - 3z_1z_2 - z_1^3 - z_2^3} -
        \frac{z_1}{1 - z_1z_2}.
    \end{equation}
\end{Proposition}
\medskip

\begin{Proposition} \label{prop:serie_Hilbert_aaa_bbb}
    The Hilbert series $\SerieOp$ of $\Op{\AAA, \BBB}$ satisfies
    \begin{equation}
        4t - 2t^2 - t^3 + t^4 + (-4 + 4t - t^2 + 2t^3)\SerieOp +
        (6 + t)\SerieOp^2 + (1 - 2t)\SerieOp^3 - \SerieOp^4
         = 0.
    \end{equation}
\end{Proposition}
\medskip

\begin{Proposition} \label{prop:relations_aaa_bbb}
    The operad $\Op{\AAA, \BBB}$ does not admit nontrivial relations
    between its generators in degree two, three, five and six. It admits
    the following non trivial relations between its generators in degree
    four:
    \begin{equation}
        ((\AAA \circ_2 \BBB) \circ_3 \AAA) \circ_3 \BBB
        \enspace = \enspace
        ((\AAA \circ_1 \BBB) \circ_1 \AAA) \circ_2 \BBB,
    \end{equation}
    \begin{equation}
        ((\AAA \circ_2 \BBB) \circ_2 \AAA) \circ_4 \AAA
        \enspace = \enspace
        ((\AAA \circ_1 \BBB) \circ_1 \AAA) \circ_3 \AAA,
    \end{equation}
    \begin{equation}
        ((\AAA \circ_2 \BBB) \circ_2 \AAA) \circ_3 \BBB
        \enspace = \enspace
        ((\AAA \circ_1 \BBB) \circ_1 \AAA) \circ_4 \BBB,
    \end{equation}
    \begin{equation}
        ((\AAA \circ_1 \BBB) \circ_3 \BBB) \circ_4 \AAA
        \enspace = \enspace
        ((\AAA \circ_1 \BBB) \circ_2 \AAA) \circ_2 \BBB,
    \end{equation}
    \begin{equation}
        ((\BBB \circ_2 \AAA) \circ_3 \BBB) \circ_3 \AAA
        \enspace = \enspace
        ((\BBB \circ_1 \AAA) \circ_1 \BBB) \circ_2 \AAA,
    \end{equation}
    \begin{equation}
        ((\BBB \circ_2 \AAA) \circ_2 \BBB) \circ_4 \BBB
        \enspace = \enspace
        ((\BBB \circ_1 \AAA) \circ_1 \BBB) \circ_3 \BBB,
    \end{equation}
    \begin{equation}
        ((\BBB \circ_2 \AAA) \circ_2 \BBB) \circ_3 \AAA
        \enspace = \enspace
        ((\BBB \circ_1 \AAA) \circ_1 \BBB) \circ_4 \AAA,
    \end{equation}
    \begin{equation}
        ((\BBB \circ_1 \AAA) \circ_3 \AAA) \circ_4 \BBB
        \enspace = \enspace
        ((\BBB \circ_1 \AAA) \circ_2 \BBB) \circ_2 \AAA.
    \end{equation}
\end{Proposition}
\begin{proof}
    This statement is proven with the help of the computer. All compositions
    between the generators $\AAA$ and $\BBB$ are computed up to degree
    six and relations thus established.
\end{proof}
\medskip

Proposition \ref{prop:relations_aaa_bbb} does not provide a presentation
by generators and relations of $\Op{\AAA, \BBB}$. The methods employed
in this article fail to establish the presentation of $\OpC{\AAA, \BBB}$
because it is not possible to define a good orientation of the relations
of the statement of Proposition \ref{prop:relations_aaa_bbb}.
Indeed, in degree six, all the orientations have no less than $7518$ normal
forms whereas they should be $7516$. Nevertheless, these relations seem
to be the only nontrivial ones; this may be proved by using the Knuth-Bendix
completion algorithm (see \cite{BK70,BN98}) over an appropriate orientation
of the relations.
\medskip

\subsubsection{Third orbit}
This orbit consists of the operads $\Op{\AAA, \ABB}$, $\Op{\AAA, \BBA}$,
$\Op{\AAB, \BBB}$, and $\Op{\BAA, \BBB}$. We choose $\Op{\AAA, \ABB}$
as a representative of the orbit.
\medskip

\begin{Proposition} \label{prop:bulles_aaa_abb}
    The set of based (resp. nonbased) bubbles of $\OpC{\AAA, \ABB}$
    of arity $n$ is the set of based (resp. nonbased) bubbles such that first
    edge is blue and the number of uncoloured edges of the border is congruent
    to $n$ (resp. $n + 1$) modulo $2$.
    Moreover, the coloured Hilbert series of $\OpC{\AAA, \ABB}$ satisfy
    \begin{equation}
        \SerieBulles_1(z_1, z_2) =
        \frac{z_2^2}{1 - 2z_1 + z_1^2 - z_2^2}
        \qquad \mbox{ and } \qquad
        \SerieBulles_2(z_1, z_2) =
        \frac{z_1z_2 - z_1^2z_2 + z_2^3}{1 - 2z_1 + z_1^2 - z_2^2}.
    \end{equation}
\end{Proposition}
\medskip

\begin{Proposition} \label{prop:serie_Hilbert_aaa_abb}
    The Hilbert series $\SerieOp$ of $\Op{\AAA, \ABB}$ satisfies
    \begin{equation}
        2t - t^2 + (2t - 2)\SerieOp + 3\SerieOp^2 = 0.
    \end{equation}
\end{Proposition}
\medskip

\begin{Theoreme} \label{thm:presentation_aaa_abb}
    The operad $\Op{\AAA, \ABB}$ admits the presentation
    $(\{\AAA, \ABB\}, \leftrightarrow)$ where $\leftrightarrow$ is the
    equivalence relation satisfying
    \begin{equation}
        (\Corolle(\ABB) \circ_2 \Corolle(\AAA)) \circ_3 \Corolle(\ABB)
        \enspace \leftrightarrow \enspace
        (\Corolle(\ABB) \circ_1 \Corolle(\ABB)) \circ_2 \Corolle(\AAA),
    \end{equation}
    \begin{equation}
        (\Corolle(\AAA) \circ_2 \Corolle(\ABB)) \circ_3 \Corolle(\AAA)
        \enspace \leftrightarrow \enspace
        (\Corolle(\AAA) \circ_1 \Corolle(\ABB)) \circ_2 \Corolle(\AAA).
    \end{equation}
\end{Theoreme}
\medskip

The proof of Theorem \ref{thm:presentation_aaa_abb} relies on the good
orientation
\begin{equation}
    \begin{split}\scalebox{.35}{\begin{tikzpicture}[scale=.85]
        \node[Feuille](0)at(0.,-1.5){};
        \node[Noeud,Clair](1)at(1.5,0.){$\scalebox{2}{\ABB}$};
        \node[Feuille](2)at(3.,-3.){};
        \node[Noeud,Clair](3)at(4.5,-1.5){$\scalebox{2}{\AAA}$};
        \node[Feuille](4)at(6.,-4.5){};
        \node[Noeud,Clair](5)at(7.5,-3.){$\scalebox{2}{\ABB}$};
        \node[Feuille](6)at(9.,-4.5){};
        \draw[Arete](0)--(1);
        \draw[Arete](2)--(3);
        \draw[Arete](3)--(1);
        \draw[Arete](4)--(5);
        \draw[Arete](5)--(3);
        \draw[Arete](6)--(5);
    \end{tikzpicture}}\end{split}
    \begin{split} \quad \mapsfrom \quad \end{split}
    \begin{split}\scalebox{.35}{\begin{tikzpicture}[scale=.85]
        \node[Feuille](0)at(0.,-3.){};
        \node[Noeud,Clair](1)at(1.5,-1.5){$\scalebox{2}{\ABB}$};
        \node[Feuille](2)at(3.,-4.5){};
        \node[Noeud,Clair](3)at(4.5,-3.){$\scalebox{2}{\AAA}$};
        \node[Feuille](4)at(6.,-4.5){};
        \node[Noeud,Clair](5)at(7.5,0.){$\scalebox{2}{\ABB}$};
        \node[Feuille](6)at(9.,-1.5){};
        \draw[Arete](0)--(1);
        \draw[Arete](1)--(5);
        \draw[Arete](2)--(3);
        \draw[Arete](3)--(1);
        \draw[Arete](4)--(3);
        \draw[Arete](6)--(5);
    \end{tikzpicture}}\end{split}\,,
\end{equation}
\begin{equation}
    \begin{split}\scalebox{.35}{\begin{tikzpicture}[scale=.85]
        \node[Feuille](0)at(0.,-1.5){};
        \node[Noeud,Clair](1)at(1.5,0.){$\scalebox{2}{\AAA}$};
        \node[Feuille](2)at(3.,-3.){};
        \node[Noeud,Clair](3)at(4.5,-1.5){$\scalebox{2}{\ABB}$};
        \node[Feuille](4)at(6.,-4.5){};
        \node[Noeud,Clair](5)at(7.5,-3.){$\scalebox{2}{\AAA}$};
        \node[Feuille](6)at(9.,-4.5){};
        \draw[Arete](0)--(1);
        \draw[Arete](2)--(3);
        \draw[Arete](3)--(1);
        \draw[Arete](4)--(5);
        \draw[Arete](5)--(3);
        \draw[Arete](6)--(5);
    \end{tikzpicture}}\end{split}
    \begin{split} \quad \mapsto \quad \end{split}
    \begin{split}\scalebox{.35}{\begin{tikzpicture}[scale=.85]
        \node[Feuille](0)at(0.,-3.){};
        \node[Noeud,Clair](1)at(1.5,-1.5){$\scalebox{2}{\ABB}$};
        \node[Feuille](2)at(3.,-4.5){};
        \node[Noeud,Clair](3)at(4.5,-3.){$\scalebox{2}{\AAA}$};
        \node[Feuille](4)at(6.,-4.5){};
        \node[Noeud,Clair](5)at(7.5,0.){$\scalebox{2}{\AAA}$};
        \node[Feuille](6)at(9.,-1.5){};
        \draw[Arete](0)--(1);
        \draw[Arete](1)--(5);
        \draw[Arete](2)--(3);
        \draw[Arete](3)--(1);
        \draw[Arete](4)--(3);
        \draw[Arete](6)--(5);
    \end{tikzpicture}}\end{split}\,.
\end{equation}
\medskip

From the above presentation, we deduce that any $\Op{\AAA, \ABB}$-algebra
is a set $S$ equipped with two binary operations $\OpA$ and $\OpB$ satisfying,
for any $x, y, z, t \in S$,
\begin{equation}
    x \OpA (y \OpB (z \OpA t)) = (x \OpA (y \OpB z)) \OpA t,
\end{equation}
\begin{equation}
    x \OpB (y \OpA (z \OpB t)) = (x \OpA (y \OpB z)) \OpB t.
\end{equation}
\medskip

\subsubsection{Fourth orbit}
This orbit consists of the operads $\Op{\AAB, \ABB}$ and $\Op{\BAA, \BBA}$.
We choose $\Op{\AAB, \ABB}$ as a representative of the orbit.
\medskip

\begin{Proposition} \label{prop:bulles_aab_abb}
    The set of bubbles of $\OpC{\AAB, \ABB}$ is the set of bubbles such
    that first edge is blue and last edge uncoloured.
    Moreover, the coloured Hilbert series of $\OpC{\AAB, \ABB}$ satisfy
    \begin{equation}
        \SerieBulles_1(z_1, z_2) =
        \frac{z_1z_2}{1 - z_1 - z_2}
        \qquad \mbox{ and } \qquad
        \SerieBulles_2(z_1, z_2) =
        \frac{z_1z_2}{1 - z_1 - z_2}.
    \end{equation}
\end{Proposition}
\medskip

\begin{Proposition} \label{prop:serie_Hilbert_aab_abb}
    The Hilbert series $\SerieOp$ of $\Op{\AAB, \ABB}$ satisfies
    \begin{equation}
        2t - t^2 + (2t - 2)\SerieOp + 3\SerieOp^2 = 0.
    \end{equation}
\end{Proposition}
\medskip

\begin{Theoreme} \label{thm:presentation_aab_abb}
    The operad $\Op{\AAB, \ABB}$ admits the presentation
    $(\{\AAB, \ABB\}, \leftrightarrow)$ where $\leftrightarrow$ is the
    equivalence relation satisfying
    \begin{equation}
        (\Corolle(\ABB) \circ_2 \Corolle(\AAB)) \circ_2 \Corolle(\ABB)
        \enspace \leftrightarrow \enspace
        (\Corolle(\ABB) \circ_1 \Corolle(\ABB)) \circ_2 \Corolle(\AAB),
    \end{equation}
    \begin{equation}
        (\Corolle(\AAB) \circ_2 \Corolle(\AAB)) \circ_2 \Corolle(\ABB)
        \enspace \leftrightarrow \enspace
        (\Corolle(\AAB) \circ_1 \Corolle(\ABB)) \circ_2 \Corolle(\AAB).
    \end{equation}
\end{Theoreme}
\medskip

The proof of Theorem \ref{thm:presentation_aab_abb} relies on the good
orientation
\begin{equation}
    \begin{split}\scalebox{.35}{\begin{tikzpicture}[scale=.85]
        \node[Feuille](0)at(0.,-1.5){};
        \node[Noeud,Clair](1)at(1.5,0.){\scalebox{2}{$\ABB$}};
        \node[Feuille](2)at(3.,-4.5){};
        \node[Noeud,Clair](3)at(4.5,-3.){\scalebox{2}{$\ABB$}};
        \node[Feuille](4)at(6.,-4.5){};
        \node[Noeud,Clair](5)at(7.5,-1.5){\scalebox{2}{$\AAB$}};
        \node[Feuille](6)at(9.,-3.){};
        \draw[Arete](0)--(1);
        \draw[Arete](2)--(3);
        \draw[Arete](3)--(5);
        \draw[Arete](4)--(3);
        \draw[Arete](5)--(1);
        \draw[Arete](6)--(5);
    \end{tikzpicture}}\end{split}
    \begin{split} \quad \mapsto \quad \end{split}
    \begin{split}\scalebox{.35}{\begin{tikzpicture}[scale=.85]
        \node[Feuille](0)at(0.,-3.){};
        \node[Noeud,Clair](1)at(1.5,-1.5){\scalebox{2}{$\ABB$}};
        \node[Feuille](2)at(3.,-4.5){};
        \node[Noeud,Clair](3)at(4.5,-3.){\scalebox{2}{$\AAB$}};
        \node[Feuille](4)at(6.,-4.5){};
        \node[Noeud,Clair](5)at(7.5,0.){\scalebox{2}{$\ABB$}};
        \node[Feuille](6)at(9.,-1.5){};
        \draw[Arete](0)--(1);
        \draw[Arete](1)--(5);
        \draw[Arete](2)--(3);
        \draw[Arete](3)--(1);
        \draw[Arete](4)--(3);
        \draw[Arete](6)--(5);
    \end{tikzpicture}}\end{split}\,,
\end{equation}
\begin{equation}
    \begin{split}\scalebox{.35}{\begin{tikzpicture}[scale=.85]
        \node[Feuille](0)at(0.,-1.5){};
        \node[Noeud,Clair](1)at(1.5,0.){\scalebox{2}{$\AAB$}};
        \node[Feuille](2)at(3.,-4.5){};
        \node[Noeud,Clair](3)at(4.5,-3.){\scalebox{2}{$\ABB$}};
        \node[Feuille](4)at(6.,-4.5){};
        \node[Noeud,Clair](5)at(7.5,-1.5){\scalebox{2}{$\AAB$}};
        \node[Feuille](6)at(9.,-3.){};
        \draw[Arete](0)--(1);
        \draw[Arete](2)--(3);
        \draw[Arete](3)--(5);
        \draw[Arete](4)--(3);
        \draw[Arete](5)--(1);
        \draw[Arete](6)--(5);
    \end{tikzpicture}}\end{split}
    \begin{split} \quad \mapsto \quad \end{split}
    \begin{split}\scalebox{.35}{\begin{tikzpicture}[scale=.85]
        \node[Feuille](0)at(0.,-3.){};
        \node[Noeud,Clair](1)at(1.5,-1.5){\scalebox{2}{$\ABB$}};
        \node[Feuille](2)at(3.,-4.5){};
        \node[Noeud,Clair](3)at(4.5,-3.){\scalebox{2}{$\AAB$}};
        \node[Feuille](4)at(6.,-4.5){};
        \node[Noeud,Clair](5)at(7.5,0.){\scalebox{2}{$\AAB$}};
        \node[Feuille](6)at(9.,-1.5){};
        \draw[Arete](0)--(1);
        \draw[Arete](1)--(5);
        \draw[Arete](2)--(3);
        \draw[Arete](3)--(1);
        \draw[Arete](4)--(3);
        \draw[Arete](6)--(5);
    \end{tikzpicture}}\end{split}\,.
\end{equation}
\medskip

From the above presentation, we deduce that any $\Op{\AAB, \ABB}$-algebra
is a set $S$ equipped with two binary operations $\OpA$ and $\OpB$ satisfying,
for any $x, y, z, t \in S$,
\begin{equation}
    x \OpA ((y \OpA z) \OpB t) = (x \OpA (y \OpB z)) \OpA t,
\end{equation}
\begin{equation}
    x \OpB ((y \OpA z) \OpB t) = (x \OpA (y \OpB z)) \OpB t.
\end{equation}
\medskip

\subsubsection{Fifth orbit}
This orbit consists of the operads $\Op{\AAB, \BBA}$ and $\Op{\BAA, \ABB}$.
We choose $\Op{\AAB, \BBA}$ as a representative of the orbit.
\medskip

\begin{Proposition} \label{prop:bulles_aab_bba}
    The set of based (resp. nonbased) bubbles of $\OpC{\AAB, \BBA}$
    is the set of based (resp. nonbased) bubbles such that penultimate
    edge is blue (resp. uncoloured) and the last edge is uncoloured
    (resp. blue).
    Moreover, the coloured Hilbert series of $\OpC{\AAB, \BBA}$ satisfy
    \begin{equation}
        \SerieBulles_1(z_1, z_2) =
        \frac{z_1z_2}{1 - z_1 - z_2}
        \qquad \mbox{ and } \qquad
        \SerieBulles_2(z_1, z_2) =
        \frac{z_1z_2}{1 - z_1 - z_2}.
    \end{equation}
\end{Proposition}
\medskip

\begin{Proposition} \label{prop:serie_Hilbert_aab_bba}
    The Hilbert series $\SerieOp$ of $\Op{\AAB, \BBA}$ satisfies
    \begin{equation}
        2t - t^2 + (2t - 2)\SerieOp + 3\SerieOp^2 = 0.
    \end{equation}
\end{Proposition}
\medskip

\begin{Theoreme} \label{thm:presentation_aab_bba}
    The operad $\Op{\AAB, \BBA}$ admits the presentation
    $(\{\AAB, \AAB\}, \leftrightarrow)$ where $\leftrightarrow$ is the
    equivalence relation satisfying
    \begin{equation}
        (\Corolle(\AAB) \circ_2 \Corolle(\AAB)) \circ_2 \Corolle(\BBA)
        \enspace \leftrightarrow \enspace
        (\Corolle(\AAB) \circ_1 \Corolle(\BBA)) \circ_1 \Corolle(\AAB),
    \end{equation}
    \begin{equation}
        (\Corolle(\BBA) \circ_2 \Corolle(\BBA)) \circ_2 \Corolle(\AAB)
        \enspace \leftrightarrow \enspace
        (\Corolle(\BBA) \circ_1 \Corolle(\AAB)) \circ_1 \Corolle(\BBA).
    \end{equation}
\end{Theoreme}
\medskip

The proof of Theorem \ref{thm:presentation_aab_bba} relies on the good
orientation
\begin{equation}
    \begin{split}\scalebox{.35}{\begin{tikzpicture}[scale=.85]
        \node[Feuille](0)at(0.,-1.5){};
        \node[Noeud,Clair](1)at(1.5,0.){$\scalebox{2}{\AAB}$};
        \node[Feuille](2)at(3.,-4.5){};
        \node[Noeud,Clair](3)at(4.5,-3.){$\scalebox{2}{\BBA}$};
        \node[Feuille](4)at(6.,-4.5){};
        \node[Noeud,Clair](5)at(7.5,-1.5){$\scalebox{2}{\AAB}$};
        \node[Feuille](6)at(9.,-3.){};
        \draw[Arete](0)--(1);
        \draw[Arete](2)--(3);
        \draw[Arete](3)--(5);
        \draw[Arete](4)--(3);
        \draw[Arete](5)--(1);
        \draw[Arete](6)--(5);
    \end{tikzpicture}}\end{split}
    \begin{split} \quad \mapsto \quad \end{split}
    \begin{split}\scalebox{.35}{\begin{tikzpicture}[scale=.85]
        \node[Feuille](0)at(0.,-4.5){};
        \node[Noeud,Clair](1)at(1.5,-3.){$\scalebox{2}{\AAB}$};
        \node[Feuille](2)at(3.,-4.5){};
        \node[Noeud,Clair](3)at(4.5,-1.5){$\scalebox{2}{\BBA}$};
        \node[Feuille](4)at(6.,-3.){};
        \node[Noeud,Clair](5)at(7.5,0.){$\scalebox{2}{\AAB}$};
        \node[Feuille](6)at(9.,-1.5){};
        \draw[Arete](0)--(1);
        \draw[Arete](1)--(3);
        \draw[Arete](2)--(1);
        \draw[Arete](3)--(5);
        \draw[Arete](4)--(3);
        \draw[Arete](6)--(5);
    \end{tikzpicture}}\end{split}\,,
\end{equation}
\begin{equation}
    \begin{split}\scalebox{.35}{\begin{tikzpicture}[scale=.85]
        \node[Feuille](0)at(0.,-1.5){};
        \node[Noeud,Clair](1)at(1.5,0.){$\scalebox{2}{\BBA}$};
        \node[Feuille](2)at(3.,-4.5){};
        \node[Noeud,Clair](3)at(4.5,-3.){$\scalebox{2}{\AAB}$};
        \node[Feuille](4)at(6.,-4.5){};
        \node[Noeud,Clair](5)at(7.5,-1.5){$\scalebox{2}{\BBA}$};
        \node[Feuille](6)at(9.,-3.){};
        \draw[Arete](0)--(1);
        \draw[Arete](2)--(3);
        \draw[Arete](3)--(5);
        \draw[Arete](4)--(3);
        \draw[Arete](5)--(1);
        \draw[Arete](6)--(5);
    \end{tikzpicture}}\end{split}
    \begin{split} \quad \mapsto \quad \end{split}
    \begin{split}\scalebox{.35}{\begin{tikzpicture}[scale=.85]
        \node[Feuille](0)at(0.,-4.5){};
        \node[Noeud,Clair](1)at(1.5,-3.){$\scalebox{2}{\BBA}$};
        \node[Feuille](2)at(3.,-4.5){};
        \node[Noeud,Clair](3)at(4.5,-1.5){$\scalebox{2}{\AAB}$};
        \node[Feuille](4)at(6.,-3.){};
        \node[Noeud,Clair](5)at(7.5,0.){$\scalebox{2}{\BBA}$};
        \node[Feuille](6)at(9.,-1.5){};
        \draw[Arete](0)--(1);
        \draw[Arete](1)--(3);
        \draw[Arete](2)--(1);
        \draw[Arete](3)--(5);
        \draw[Arete](4)--(3);
        \draw[Arete](6)--(5);
    \end{tikzpicture}}\end{split}\,.
\end{equation}
\medskip

From the above presentation, we deduce that any $\Op{\AAB, \BBA}$-algebra
is a set $S$ equipped with two binary operations $\OpA$ and $\OpB$ satisfying,
for any $x, y, z, t \in S$,
\begin{equation}
    x \OpB ((y \OpA z) \OpB t) = ((x \OpB y) \OpA z) \OpB t,
\end{equation}
\begin{equation}
    x \OpA ((y \OpB z) \OpA t) = ((x \OpA y) \OpB z) \OpA t.
\end{equation}
\medskip

\subsubsection{Sixth orbit}
This orbit consists of the operads $\Op{\AAA, \BAB}$ and $\Op{\ABA, \BBB}$.
We choose $\Op{\AAA, \BAB}$ as a representative of the orbit.
\medskip

\begin{Proposition} \label{prop:bulles_aaa_bab}
    The set of bubbles of $\OpC{\AAA, \BAB}$ is the set of based bubbles
    such that maximal sequences of blues edges of the border have even
    length.Moreover, the coloured Hilbert series of $\OpC{\AAA, \BAB}$
    satisfy
    \begin{equation}
        \SerieBulles_1(z_1, z_2) =
        \frac{z_1^2 + z_2^2 + z_1z_2^2}{1 - z_1 - z_2^2}
        \qquad \mbox{ and } \qquad
        \SerieBulles_2(z_1, z_2) = 0.
    \end{equation}
\end{Proposition}
\medskip

\begin{Proposition} \label{prop:serie_Hilbert_aaa_bab}
    The Hilbert series $\SerieOp$ of $\Op{\AAA, \BAB}$ satisfies
    \begin{equation}
        t + (t - 1) \SerieOp + \SerieOp^2 + \SerieOp^3 = 0.
    \end{equation}
\end{Proposition}
\medskip

\begin{Theoreme} \label{thm:presentation_aaa_bab}
    The operad $\Op{\AAA, \BAB}$ admits the presentation
    $(\{\AAA, \BAB\}, \leftrightarrow)$ where $\leftrightarrow$ is the
    equivalence relation satisfying
    \begin{equation}
        \Corolle(\BAB) \circ_1 \Corolle(\BAB)
        \enspace \leftrightarrow \enspace
        \Corolle(\BAB) \circ_2 \Corolle(\BAB).
    \end{equation}
\end{Theoreme}
\medskip

The proof of Theorem \ref{thm:presentation_aaa_bab} relies on the good
orientation
\begin{equation}
    \begin{split}\scalebox{.35}{\begin{tikzpicture}[scale=.85]
        \node[Feuille](0)at(0.,-3.){};
        \node[Noeud,Clair](1)at(1.5,-1.5){$\scalebox{2}{\BAB}$};
        \node[Feuille](2)at(3.,-3.){};
        \node[Noeud,Clair](3)at(4.5,0.){$\scalebox{2}{\BAB}$};
        \node[Feuille](4)at(6.,-1.5){};
        \draw[Arete](0)--(1);
        \draw[Arete](1)--(3);
        \draw[Arete](2)--(1);
        \draw[Arete](4)--(3);
    \end{tikzpicture}}\end{split}
    \begin{split} \quad \mapsfrom \quad \end{split}
    \begin{split}\scalebox{.35}{\begin{tikzpicture}[scale=.85]
        \node[Feuille](0)at(0.,-1.5){};
        \node[Noeud,Clair](1)at(1.5,0.){$\scalebox{2}{\BAB}$};
        \node[Feuille](2)at(3.,-3.){};
        \node[Noeud,Clair](3)at(4.5,-1.5){$\scalebox{2}{\BAB}$};
        \node[Feuille](4)at(6.,-3.){};
        \draw[Arete](0)--(1);
        \draw[Arete](2)--(3);
        \draw[Arete](3)--(1);
        \draw[Arete](4)--(3);
    \end{tikzpicture}}\end{split}\,.
\end{equation}
\medskip

From the above presentation, we deduce that any $\Op{\AAA, \BAB}$-algebra
is a set $S$ equipped with two binary operations $\OpA$ and $\OpB$ satisfying,
for any $x, y, z \in S$,
\begin{equation}
   (x \OpA y) \OpA z = x \OpA (y \OpA z).
\end{equation}
No relation involves $\OpB$.
\medskip

\subsubsection{Seventh orbit}
This orbit consists of the operads $\Op{\AAB, \BAA}$ and $\Op{\ABB, \BBA}$.
We choose $\Op{\AAB, \BAA}$ as a representative of the orbit.
\medskip

\begin{Proposition} \label{prop:bulles_aab_baa}
    The set of bubbles of $\OpC{\AAB, \BAA}$ is the set of based bubbles
    having exactly one uncoloured edge in the border.
    Moreover, the coloured Hilbert series of $\OpC{\AAB, \BAA}$ satisfy
    \begin{equation}
        \SerieBulles_1(z_1, z_2) =
        \frac{2z_1z_2 - z_1z_2^2}{(1 - z_2)^2}
        \qquad \mbox{ and } \qquad
        \SerieBulles_2(z_1, z_2) = 0.
    \end{equation}
\end{Proposition}
\medskip

\begin{Proposition} \label{prop:serie_Hilbert_aab_baa}
    The Hilbert series $\SerieOp$ of $\Op{\AAB, \BAA}$ satisfies
    \begin{equation}
        t - \SerieOp + 2 \SerieOp^2 - \SerieOp^3 = 0.
    \end{equation}
\end{Proposition}
\medskip

\begin{Theoreme} \label{thm:presentation_aab_baa}
    The operad $\Op{\AAB, \BAA}$ admits the presentation
    $(\{\AAB, \BAA\}, \leftrightarrow)$ where $\leftrightarrow$ is the
    equivalence relation satisfying
    \begin{equation}
        \Corolle(\BAA) \circ_1 \Corolle(\AAB)
        \enspace \leftrightarrow \enspace
        \Corolle(\AAB) \circ_2 \Corolle(\BAA).
    \end{equation}
\end{Theoreme}
\medskip

The proof of Theorem \ref{thm:presentation_aab_baa} relies on the good
orientation
\begin{equation}
    \begin{split}\scalebox{.35}{\begin{tikzpicture}[scale=.85]
        \node[Feuille](0)at(0.,-3.){};
        \node[Noeud,Clair](1)at(1.5,-1.5){$\scalebox{2}{\AAB}$};
        \node[Feuille](2)at(3.,-3.){};
        \node[Noeud,Clair](3)at(4.5,0.){$\scalebox{2}{\BAA}$};
        \node[Feuille](4)at(6.,-1.5){};
        \draw[Arete](0)--(1);
        \draw[Arete](1)--(3);
        \draw[Arete](2)--(1);
        \draw[Arete](4)--(3);
    \end{tikzpicture}}\end{split}
    \begin{split} \quad \mapsfrom \quad \end{split}
    \begin{split}\scalebox{.35}{\begin{tikzpicture}[scale=.85]
        \node[Feuille](0)at(0.,-1.5){};
        \node[Noeud,Clair](1)at(1.5,0.){$\scalebox{2}{\AAB}$};
        \node[Feuille](2)at(3.,-3.){};
        \node[Noeud,Clair](3)at(4.5,-1.5){$\scalebox{2}{\BAA}$};
        \node[Feuille](4)at(6.,-3.){};
        \draw[Arete](0)--(1);
        \draw[Arete](2)--(3);
        \draw[Arete](3)--(1);
        \draw[Arete](4)--(3);
    \end{tikzpicture}}\end{split}\,.
\end{equation}
\medskip

From the above presentation, we deduce that any $\Op{\AAB, \BAA}$-algebra
is a set $S$ equipped with two binary operations $\OpA$ and $\OpB$ satisfying,
for any $x, y, z \in S$,
\begin{equation}
   (x \OpB y) \OpA z = x \OpB (y \OpA z).
\end{equation}
This relation is the one of L-algebras \cite{Cha07,Ler11}.
\medskip

\subsubsection{Eighth orbit}
This orbit consists of the operads $\Op{\AAB, \BAB}$, $\Op{\BAA, \BAB}$,
$\Op{\ABA, \ABB}$, and $\Op{\ABA, \BBA}$. We choose $\Op{\AAB, \BAB}$
as a representative of the orbit.
\medskip

\begin{Proposition} \label{prop:bulles_aab_bab}
    The set of bubbles of $\OpC{\AAB, \BAB}$ is the set of based bubbles
    such that last edge is uncoloured. Moreover, the coloured Hilbert
    series of $\OpC{\AAB, \BAB}$ satisfy
    \begin{equation}
        \SerieBulles_1(z_1, z_2) =
        \frac{z_1^2 + z_1z_2}{1 - z_1 - z_2}
        \qquad \mbox{ and } \qquad
        \SerieBulles_2(z_1, z_2) = 0.
    \end{equation}
\end{Proposition}
\medskip

\begin{Proposition} \label{prop:serie_Hilbert_aab_bab}
    The Hilbert series $\SerieOp$ of $\Op{\AAB, \BAB}$ satisfies
    \begin{equation}
        t - (1 - t) \SerieOp + \SerieOp^2 = 0.
    \end{equation}
\end{Proposition}
\medskip

\begin{Theoreme} \label{thm:presentation_aab_bab}
    The operad $\Op{\AAB, \BAB}$ admits the presentation
    $(\{\AAB, \BAB\}, \leftrightarrow)$ where $\leftrightarrow$ is the
    equivalence relation satisfying
    \begin{equation}
        \Corolle(\BAB) \circ_1 \Corolle(\BAB)
        \enspace \leftrightarrow \enspace
        \Corolle(\BAB) \circ_2 \Corolle(\BAB),
    \end{equation}
    \begin{equation}
        \Corolle(\BAB) \circ_1 \Corolle(\AAB)
        \enspace \leftrightarrow \enspace
        \Corolle(\AAB) \circ_2 \Corolle(\BAB).
    \end{equation}
\end{Theoreme}
\medskip

The proof of Theorem \ref{thm:presentation_aab_bab} relies on the good
orientation
\begin{equation}
    \begin{split}\scalebox{.35}{\begin{tikzpicture}[scale=.85]
        \node[Feuille](0)at(0.,-3.){};
        \node[Noeud,Clair](1)at(1.5,-1.5){$\scalebox{2}{\BAB}$};
        \node[Feuille](2)at(3.,-3.){};
        \node[Noeud,Clair](3)at(4.5,0.){$\scalebox{2}{\BAB}$};
        \node[Feuille](4)at(6.,-1.5){};
        \draw[Arete](0)--(1);
        \draw[Arete](1)--(3);
        \draw[Arete](2)--(1);
        \draw[Arete](4)--(3);
    \end{tikzpicture}}\end{split}
    \begin{split} \quad \mapsto \quad \end{split}
    \begin{split}\scalebox{.35}{\begin{tikzpicture}[scale=.85]
        \node[Feuille](0)at(0.,-1.5){};
        \node[Noeud,Clair](1)at(1.5,0.){$\scalebox{2}{\BAB}$};
        \node[Feuille](2)at(3.,-3.){};
        \node[Noeud,Clair](3)at(4.5,-1.5){$\scalebox{2}{\BAB}$};
        \node[Feuille](4)at(6.,-3.){};
        \draw[Arete](0)--(1);
        \draw[Arete](2)--(3);
        \draw[Arete](3)--(1);
        \draw[Arete](4)--(3);
    \end{tikzpicture}}\end{split}\,,
\end{equation}
\begin{equation}
    \begin{split}\scalebox{.35}{\begin{tikzpicture}[scale=.85]
        \node[Feuille](0)at(0.,-3.){};
        \node[Noeud,Clair](1)at(1.5,-1.5){$\scalebox{2}{\AAB}$};
        \node[Feuille](2)at(3.,-3.){};
        \node[Noeud,Clair](3)at(4.5,0.){$\scalebox{2}{\BAB}$};
        \node[Feuille](4)at(6.,-1.5){};
        \draw[Arete](0)--(1);
        \draw[Arete](1)--(3);
        \draw[Arete](2)--(1);
        \draw[Arete](4)--(3);
    \end{tikzpicture}}\end{split}
    \begin{split} \quad \mapsto \quad \end{split}
    \begin{split}\scalebox{.35}{\begin{tikzpicture}[scale=.85]
        \node[Feuille](0)at(0.,-1.5){};
        \node[Noeud,Clair](1)at(1.5,0.){$\scalebox{2}{\AAB}$};
        \node[Feuille](2)at(3.,-3.){};
        \node[Noeud,Clair](3)at(4.5,-1.5){$\scalebox{2}{\BAB}$};
        \node[Feuille](4)at(6.,-3.){};
        \draw[Arete](0)--(1);
        \draw[Arete](2)--(3);
        \draw[Arete](3)--(1);
        \draw[Arete](4)--(3);
    \end{tikzpicture}}\end{split}\,.
\end{equation}
\medskip

From the above presentation, we deduce that any $\Op{\AAB, \BAB}$-algebra
is a set $S$ equipped with two binary operations $\OpA$ and $\OpB$ satisfying,
for any $x, y, z \in S$,
\begin{equation}
   (x \OpA y) \OpA z = x \OpA (y \OpA z),
\end{equation}
\begin{equation}
   (x \OpB y) \OpA z = x \OpB (y \OpA z).
\end{equation}
These relations are similar to the ones of duplicial algebras \cite{Lod08,Zin12}.
However, in a duplicial algebra the operation $\OpB$ must be associative.
\medskip

\subsubsection{Ninth orbit}
This orbit consists of the operads $\Op{\AAA, \ABA}$ and $\Op{\BAB, \BBB}$.
We choose $\Op{\AAA, \ABA}$ as a representative of the orbit.
\medskip

\begin{Proposition} \label{prop:bulles_aaa_aba}
    The set of bubbles of $\OpC{\AAA, \ABA}$ is the set of bubbles
    such that all edges of the border are blue. Moreover, the coloured
    Hilbert series of $\OpC{\AAA, \ABA}$ satisfy
    \begin{equation}
        \SerieBulles_1(z_1, z_2) =
        \frac{z_2^2}{1 - z_2}
        \qquad \mbox{ and } \qquad
        \SerieBulles_2(z_1, z_2) =
        \frac{z_2^2}{1 - z_2}.
    \end{equation}
\end{Proposition}
\medskip

\begin{Proposition} \label{prop:serie_Hilbert_aaa_aba}
    The Hilbert series $\SerieOp$ of $\Op{\AAA, \ABA}$ satisfies
    \begin{equation}
        t - (1 - t) \SerieOp + \SerieOp^2 = 0.
    \end{equation}
\end{Proposition}
\medskip

\begin{Theoreme} \label{thm:presentation_aaa_aba}
    The operad $\Op{\AAA, \ABA}$ admits the presentation
    $(\{\AAA, \ABA\}, \leftrightarrow)$ where $\leftrightarrow$ is the
    equivalence relation satisfying
    \begin{equation}
        \Corolle(\ABA) \circ_1 \Corolle(\ABA)
        \enspace \leftrightarrow \enspace
        \Corolle(\ABA) \circ_2 \Corolle(\ABA),
    \end{equation}
    \begin{equation}
        \Corolle(\AAA) \circ_1 \Corolle(\ABA)
        \enspace \leftrightarrow \enspace
        \Corolle(\AAA) \circ_2 \Corolle(\ABA).
    \end{equation}
\end{Theoreme}
\medskip

The proof of Theorem \ref{thm:presentation_aaa_aba} relies on the good
orientation
\begin{equation}
    \begin{split}\scalebox{.35}{\begin{tikzpicture}[scale=.85]
        \node[Feuille](0)at(0.,-3.){};
        \node[Noeud,Clair](1)at(1.5,-1.5){$\scalebox{2}{\ABA}$};
        \node[Feuille](2)at(3.,-3.){};
        \node[Noeud,Clair](3)at(4.5,0.){$\scalebox{2}{\ABA}$};
        \node[Feuille](4)at(6.,-1.5){};
        \draw[Arete](0)--(1);
        \draw[Arete](1)--(3);
        \draw[Arete](2)--(1);
        \draw[Arete](4)--(3);
    \end{tikzpicture}}\end{split}
    \begin{split} \quad \mapsfrom \quad \end{split}
    \begin{split}\scalebox{.35}{\begin{tikzpicture}[scale=.85]
        \node[Feuille](0)at(0.,-1.5){};
        \node[Noeud,Clair](1)at(1.5,0.){$\scalebox{2}{\ABA}$};
        \node[Feuille](2)at(3.,-3.){};
        \node[Noeud,Clair](3)at(4.5,-1.5){$\scalebox{2}{\ABA}$};
        \node[Feuille](4)at(6.,-3.){};
        \draw[Arete](0)--(1);
        \draw[Arete](2)--(3);
        \draw[Arete](3)--(1);
        \draw[Arete](4)--(3);
    \end{tikzpicture}}\end{split}\,,
\end{equation}
\begin{equation}
    \begin{split}\scalebox{.35}{\begin{tikzpicture}[scale=.85]
        \node[Feuille](0)at(0.,-3.){};
        \node[Noeud,Clair](1)at(1.5,-1.5){$\scalebox{2}{\ABA}$};
        \node[Feuille](2)at(3.,-3.){};
        \node[Noeud,Clair](3)at(4.5,0.){$\scalebox{2}{\AAA}$};
        \node[Feuille](4)at(6.,-1.5){};
        \draw[Arete](0)--(1);
        \draw[Arete](1)--(3);
        \draw[Arete](2)--(1);
        \draw[Arete](4)--(3);
    \end{tikzpicture}}\end{split}
    \begin{split} \quad \mapsfrom \quad \end{split}
    \begin{split}\scalebox{.35}{\begin{tikzpicture}[scale=.85]
        \node[Feuille](0)at(0.,-1.5){};
        \node[Noeud,Clair](1)at(1.5,0.){$\scalebox{2}{\AAA}$};
        \node[Feuille](2)at(3.,-3.){};
        \node[Noeud,Clair](3)at(4.5,-1.5){$\scalebox{2}{\ABA}$};
        \node[Feuille](4)at(6.,-3.){};
        \draw[Arete](0)--(1);
        \draw[Arete](2)--(3);
        \draw[Arete](3)--(1);
        \draw[Arete](4)--(3);
    \end{tikzpicture}}\end{split}\,.
\end{equation}
\medskip

From the above presentation, we deduce that any $\Op{\AAA, \ABA}$-algebra
is a set $S$ equipped with two binary operations $\OpA$ and $\OpB$ satisfying,
for any $x, y, z \in S$,
\begin{equation}
   (x \OpA y) \OpA z = x \OpA (y \OpA z),
\end{equation}
\begin{equation}
   (x \OpA y) \OpB z = x \OpB (y \OpA z).
\end{equation}
\medskip

\subsubsection{Tenth orbit}
This orbit consists of the operads $\Op{\AAB, \ABA}$, $\Op{\BAA, \ABA}$,
$\Op{\BAB, \ABB}$, and $\Op{\BAB, \BBA}$. We choose $\Op{\BAA, \ABA}$
as a representative of the orbit.
\medskip

\begin{Proposition} \label{prop:bulles_aab_aba}
    The set of based (resp. nonbased) bubbles of $\OpC{\BAA, \ABA}$ is
    the set of based (resp. nonbased) bubbles such that first edge is
    uncoloured (resp. blue) and the other edges of the border are blue.
    Moreover, the coloured Hilbert series of $\OpC{\BAA, \ABA}$ satisfy
    \begin{equation}
        \SerieBulles_1(z_1, z_2) =
        \frac{z_1z_2}{1 - z_2}
        \qquad \mbox{ and } \qquad
        \SerieBulles_2(z_1, z_2) =
        \frac{z_2^2}{1 - z_2}.
    \end{equation}
\end{Proposition}
\medskip

\begin{Proposition} \label{prop:serie_Hilbert_aab_aba}
    The Hilbert series $\SerieOp$ of $\Op{\BAA, \ABA}$ satisfies
    \begin{equation}
        t - (1 - t) \SerieOp + \SerieOp^2 = 0.
    \end{equation}
\end{Proposition}
\medskip

\begin{Theoreme} \label{thm:presentation_aab_aba}
    The operad $\Op{\BAA, \ABA}$ admits the presentation
    $(\{\BAA, \ABA\}, \leftrightarrow)$ where $\leftrightarrow$ is the
    equivalence relation satisfying
    \begin{equation}
        \Corolle(\ABA) \circ_1 \Corolle(\ABA)
        \enspace \leftrightarrow \enspace
        \Corolle(\ABA) \circ_2 \Corolle(\ABA),
    \end{equation}
    \begin{equation}
        \Corolle(\BAA) \circ_1 \Corolle(\BAA)
        \enspace \leftrightarrow \enspace
        \Corolle(\BAA) \circ_2 \Corolle(\ABA).
    \end{equation}
\end{Theoreme}
\medskip

The proof of Theorem \ref{thm:presentation_aab_aba} relies on the good
orientation
\begin{equation}
    \begin{split}\scalebox{.35}{\begin{tikzpicture}[scale=.85]
        \node[Feuille](0)at(0.,-3.){};
        \node[Noeud,Clair](1)at(1.5,-1.5){$\scalebox{2}{\ABA}$};
        \node[Feuille](2)at(3.,-3.){};
        \node[Noeud,Clair](3)at(4.5,0.){$\scalebox{2}{\ABA}$};
        \node[Feuille](4)at(6.,-1.5){};
        \draw[Arete](0)--(1);
        \draw[Arete](1)--(3);
        \draw[Arete](2)--(1);
        \draw[Arete](4)--(3);
    \end{tikzpicture}}\end{split}
    \begin{split} \quad \mapsto \quad \end{split}
    \begin{split}\scalebox{.35}{\begin{tikzpicture}[scale=.85]
        \node[Feuille](0)at(0.,-1.5){};
        \node[Noeud,Clair](1)at(1.5,0.){$\scalebox{2}{\ABA}$};
        \node[Feuille](2)at(3.,-3.){};
        \node[Noeud,Clair](3)at(4.5,-1.5){$\scalebox{2}{\ABA}$};
        \node[Feuille](4)at(6.,-3.){};
        \draw[Arete](0)--(1);
        \draw[Arete](2)--(3);
        \draw[Arete](3)--(1);
        \draw[Arete](4)--(3);
    \end{tikzpicture}}\end{split}\,,
\end{equation}
\begin{equation}
    \begin{split}\scalebox{.35}{\begin{tikzpicture}[scale=.85]
        \node[Feuille](0)at(0.,-3.){};
        \node[Noeud,Clair](1)at(1.5,-1.5){$\scalebox{2}{\BAA}$};
        \node[Feuille](2)at(3.,-3.){};
        \node[Noeud,Clair](3)at(4.5,0.){$\scalebox{2}{\BAA}$};
        \node[Feuille](4)at(6.,-1.5){};
        \draw[Arete](0)--(1);
        \draw[Arete](1)--(3);
        \draw[Arete](2)--(1);
        \draw[Arete](4)--(3);
    \end{tikzpicture}}\end{split}
    \begin{split} \quad \mapsto \quad \end{split}
    \begin{split}\scalebox{.35}{\begin{tikzpicture}[scale=.85]
        \node[Feuille](0)at(0.,-1.5){};
        \node[Noeud,Clair](1)at(1.5,0.){$\scalebox{2}{\BAA}$};
        \node[Feuille](2)at(3.,-3.){};
        \node[Noeud,Clair](3)at(4.5,-1.5){$\scalebox{2}{\ABA}$};
        \node[Feuille](4)at(6.,-3.){};
        \draw[Arete](0)--(1);
        \draw[Arete](2)--(3);
        \draw[Arete](3)--(1);
        \draw[Arete](4)--(3);
    \end{tikzpicture}}\end{split}\,.
\end{equation}
\medskip

From the above presentation, we deduce that any $\Op{\AAB, \ABA}$-algebra
is a set $S$ equipped with two binary operations $\OpA$ and $\OpB$ satisfying,
for any $x, y, z \in S$,
\begin{equation}
   (x \OpA y) \OpA z = x \OpA (y \OpA z),
\end{equation}
\begin{equation}
   (x \OpB y) \OpB z = x \OpB (y \OpA z).
\end{equation}
These relations are the ones of dipterous algebras \cite{LR03,Zin12}.
\medskip

\subsubsection{Eleventh orbit}
This orbit consists of the operad $\Op{\BAB, \ABA}$.
\medskip

\begin{Proposition} \label{prop:bulles_bab_aba}
    The set of based (resp. nonbased) bubbles of $\OpC{\BAB, \ABA}$ is
    the set of based (resp. nonbased) bubbles such that all edges of the
    border are uncoloured (resp. blue). Moreover, the coloured Hilbert
    series of $\OpC{\BAB, \ABA}$ satisfy
    \begin{equation}
        \SerieBulles_1(z_1, z_2) =
        \frac{z_1^2}{1 - z_1}
        \qquad \mbox{ and } \qquad
        \SerieBulles_2(z_1, z_2) =
        \frac{z_2^2}{1 - z_2}.
    \end{equation}
\end{Proposition}
\medskip

\begin{Proposition} \label{prop:serie_Hilbert_bab_aba}
    The Hilbert series $\SerieOp$ of $\Op{\BAB, \ABA}$ satisfies
    \begin{equation}
        t - (1 - t) \SerieOp + \SerieOp^2 = 0.
    \end{equation}
\end{Proposition}
\medskip

\begin{Theoreme} \label{thm:presentation_bab_aba}
    The operad $\Op{\BAB, \ABA}$ admits the presentation
    $(\{\BAB, \ABA\}, \leftrightarrow)$ where $\leftrightarrow$ is the
    equivalence relation satisfying
    \begin{equation}
        \Corolle(\ABA) \circ_1 \Corolle(\ABA)
        \enspace \leftrightarrow \enspace
        \Corolle(\ABA) \circ_2 \Corolle(\ABA),
    \end{equation}
    \begin{equation}
        \Corolle(\BAB) \circ_1 \Corolle(\BAB)
        \enspace \leftrightarrow \enspace
        \Corolle(\BAB) \circ_2 \Corolle(\BAB).
    \end{equation}
\end{Theoreme}
\medskip

The proof of Theorem \ref{thm:presentation_bab_aba} relies on the good
orientation
\begin{equation}
    \begin{split}\scalebox{.35}{\begin{tikzpicture}[scale=.85]
        \node[Feuille](0)at(0.,-3.){};
        \node[Noeud,Clair](1)at(1.5,-1.5){$\scalebox{2}{\ABA}$};
        \node[Feuille](2)at(3.,-3.){};
        \node[Noeud,Clair](3)at(4.5,0.){$\scalebox{2}{\ABA}$};
        \node[Feuille](4)at(6.,-1.5){};
        \draw[Arete](0)--(1);
        \draw[Arete](1)--(3);
        \draw[Arete](2)--(1);
        \draw[Arete](4)--(3);
    \end{tikzpicture}}\end{split}
    \begin{split} \quad \mapsfrom \quad \end{split}
    \begin{split}\scalebox{.35}{\begin{tikzpicture}[scale=.85]
        \node[Feuille](0)at(0.,-1.5){};
        \node[Noeud,Clair](1)at(1.5,0.){$\scalebox{2}{\ABA}$};
        \node[Feuille](2)at(3.,-3.){};
        \node[Noeud,Clair](3)at(4.5,-1.5){$\scalebox{2}{\ABA}$};
        \node[Feuille](4)at(6.,-3.){};
        \draw[Arete](0)--(1);
        \draw[Arete](2)--(3);
        \draw[Arete](3)--(1);
        \draw[Arete](4)--(3);
    \end{tikzpicture}}\end{split}\,,
\end{equation}
\begin{equation}
    \begin{split}\scalebox{.35}{\begin{tikzpicture}[scale=.85]
        \node[Feuille](0)at(0.,-3.){};
        \node[Noeud,Clair](1)at(1.5,-1.5){$\scalebox{2}{\BAB}$};
        \node[Feuille](2)at(3.,-3.){};
        \node[Noeud,Clair](3)at(4.5,0.){$\scalebox{2}{\BAB}$};
        \node[Feuille](4)at(6.,-1.5){};
        \draw[Arete](0)--(1);
        \draw[Arete](1)--(3);
        \draw[Arete](2)--(1);
        \draw[Arete](4)--(3);
    \end{tikzpicture}}\end{split}
    \begin{split} \quad \mapsfrom \quad \end{split}
    \begin{split}\scalebox{.35}{\begin{tikzpicture}[scale=.85]
        \node[Feuille](0)at(0.,-1.5){};
        \node[Noeud,Clair](1)at(1.5,0.){$\scalebox{2}{\BAB}$};
        \node[Feuille](2)at(3.,-3.){};
        \node[Noeud,Clair](3)at(4.5,-1.5){$\scalebox{2}{\BAB}$};
        \node[Feuille](4)at(6.,-3.){};
        \draw[Arete](0)--(1);
        \draw[Arete](2)--(3);
        \draw[Arete](3)--(1);
        \draw[Arete](4)--(3);
    \end{tikzpicture}}\end{split}\,.
\end{equation}
\medskip

From the above presentation, we deduce that any $\Op{\BAB, \ABA}$-algebra
is a set $S$ equipped with two binary operations $\OpA$ and $\OpB$ satisfying,
for any $x, y, z \in S$,
\begin{equation}
   (x \OpA y) \OpA z = x \OpA (y \OpA z),
\end{equation}
\begin{equation}
   (x \OpB y) \OpB z = x \OpB (y \OpB z).
\end{equation}
These relations are the ones of two-associative algebras \cite{LR06,Zin12}.
\medskip

\bibliographystyle{alpha}
\bibliography{Bibliographie}

\newcommand{\etalchar}[1]{$^{#1}$}
\begin{thebibliography}{CLRS09}

\bibitem[BN98]{BN98}
F.~Baader and T.~Nipkow.
\newblock {\em {Term Rewriting and All That}}.
\newblock Cambridge University Press, 1998.

\bibitem[BV73]{boardman_vogt}
J.~M. Boardman and R.~M. Vogt.
\newblock {\em Homotopy invariant algebraic structures on topological spaces}.
\newblock Lecture Notes in Mathematics, Vol. 347. Springer-Verlag, Berlin,
  1973.

\bibitem[Cha07]{Cha07}
F.~Chapoton.
\newblock {The anticyclic operad of moulds}.
\newblock {\em Int. Math. Res. Notices}, 20:Art. ID rnm078, 36, 2007.

\bibitem[Cha08]{Cha08}
F.~Chapoton.
\newblock {Operads and algebraic combinatorics of trees}.
\newblock {\em S\'em. Lothar. Combin.}, 58:Art. B58c, 27, 2008.

\bibitem[CLRS09]{CLRS09}
T.~Cormen, C.~Leiserson, R.~Rivest, and C.~Stein.
\newblock {\em Introduction to algorithms}.
\newblock MIT Press, third edition, 2009.

\bibitem[FN99]{FN99}
P.~Flajolet and M.~Noy.
\newblock {Analytic combinatorics of non-crossing configurations}.
\newblock {\em Discrete Math.}, 204(1-3):203--229, 1999.

\bibitem[FS09]{FS09}
P.~Flajolet and R.~Sedgewick.
\newblock {\em {Analytic Combinatorics}}.
\newblock Cambridge University Press, 2009.

\bibitem[GK95]{getzler_kapranov}
E.~Getzler and M.~M. Kapranov.
\newblock Cyclic operads and cyclic homology.
\newblock In {\em Geometry, topology, \& physics}, Conf. Proc. Lecture Notes
  Geom. Topology, IV, pages 167--201. Int. Press, Cambridge, MA, 1995.

\bibitem[Kap98]{kapranov_icm}
M.~Kapranov.
\newblock Operads and algebraic geometry.
\newblock {\em Doc. Math.}, Extra Vol. II:277--286 (electronic), 1998.

\bibitem[KB70]{BK70}
D.~Knuth and P.~Bendix.
\newblock {Simple Word Problems in Universal Algebra}.
\newblock In {\em Computational {P}roblems in {A}bstract {A}lgebra ({P}roc.
  {C}onf., {O}xford, 1967)}, pages 263--297. Pergamon, Oxford, 1970.

\bibitem[Ler11]{Ler11}
P.~Leroux.
\newblock L-algebras, triplicial-algebras, within an equivalence of categories
  motivated by graphs.
\newblock {\em Comm. Algebra}, 39(8):2661--2689, 2011.

\bibitem[Lod08]{Lod08}
J.-L. Loday.
\newblock {\em Generalized bialgebras and triples of operads}, volume 320 of
  {\em Ast\'erisque}.
\newblock Soci\'et\'e Math\'ematique de France, 2008.

\bibitem[LR03]{LR03}
J.-L. Loday and M.~Ronco.
\newblock {Alg\`ebres de Hopf colibres}.
\newblock {\em C. R. Math. Acad. Sci. Paris}, 337(3):153--158, 2003.

\bibitem[LR06]{LR06}
J.-L. Loday and M.~Ronco.
\newblock On the structure of cofree {H}opf algebras.
\newblock {\em J. Reine Angew. Math.}, 592:123--155, 2006.

\bibitem[LV12]{loday_vallette}
J.-L. Loday and B.~Vallette.
\newblock {\em Algebraic operads}, volume 346 of {\em Grundlehren der
  Mathematischen Wissenschaften [Fundamental Principles of Mathematical
  Sciences]}.
\newblock Springer, Heidelberg, 2012.

\bibitem[S{\etalchar{+}}13]{sage}
W.\thinspace{}A. Stein et~al.
\newblock {\em {S}age {M}athematics {S}oftware ({V}ersion 5.10)}.
\newblock The Sage Development Team, 2013.
\newblock \url{http://www.sagemath.org}.

\bibitem[SCc08]{sage-combinat}
The {S}age-{C}ombinat community.
\newblock {S}age-{C}ombinat: enhancing {S}age as a toolbox for computer
  exploration in algebraic combinatorics, 2008.
\newblock \url{http://combinat.sagemath.org}.

\bibitem[Slo]{Slo10}
N.~J.~A. Sloane.
\newblock {The On-Line Encyclopedia of Integer Sequences}.
\newblock \url{http://www.research.att.com/~njas/sequences/}.

\bibitem[Zin12]{Zin12}
G.~W. Zinbiel.
\newblock Encyclopedia of types of algebras 2010.
\newblock {\em Proc. Int. Conf., Nankai Series in Pure, Applied Mathematics and
  Theoretical Physics, World Scientific, Singapore}, 9:217--298, 2012.

\end{thebibliography}

\end{document}